\documentclass[11pt]{article}

\usepackage{grffile}

\usepackage{amsmath,amsfonts,amscd,a4wide}
\usepackage{amsmath}
\usepackage{amssymb}
\usepackage{graphicx}
\usepackage{amsthm}
\usepackage{color}
\usepackage{color}

\usepackage{color}
\usepackage{transparent}
\usepackage[skip=10pt,font=footnotesize]{caption}
\usepackage{subcaption}
\usepackage{enumitem}
\graphicspath{{Figures/}}

\usepackage[margin=0.5in]{geometry}

\newtheorem{Lemma}{Lemma}[section]
\newtheorem{Theorem}{Theorem}
\newtheorem{Proposition}[Lemma]{Proposition}
\newtheorem{Corollary}[Lemma]{Corollary}
\newtheorem{Remark}{Remark}[section]

 {\begin{trivlist} \item[]{\bf Proof. }}%
 {\hspace*{\fill}$\rule{.4\baselineskip}{.4\baselineskip}$\end{trivlist}}

\newenvironment{Acknowledgment}%
 {\begin{trivlist}\item[]\textbf{Acknowledgments.}}{\end{trivlist}}

\makeatletter\@addtoreset{figure}{section}\makeatother

\makeatletter \@addtoreset{equation}{section} \makeatother

\newcommand{\R}{\mathbb{R}}

\newcommand{\N}{\mathbb{N}}

        \newcommand{\mc}[1]{\mathcal{#1}}

        \newcommand{\tl}[1]{\tilde{#1}}
        \newcommand{\lp}{\left}
        \newcommand{\rp}{\right}
                \newcommand{\la}{\lp\langle}
        \newcommand{\ra}{\rp\rangle}
        \newcommand{\beq}{\begin{equation}}
        \newcommand{\eeq}{\end{equation}}
        \newcommand{\ba}{\begin{align}}
        \newcommand{\ea}{\end{align}}
        
        \newcommand{\p}{\partial}

        \newcommand{\re}{\mathrm{e}}
        
        \newcommand{\eps}{\epsilon}

        \renewcommand{\Re}{\mathrm{Re}}

                        \newcommand{\rs}{\mathrm{s}}
                        \newcommand{\rcs}{\mathrm{cs}}

                                       \newcommand{\rcu}{\mathrm{cu}}
                                                \newcommand{\ru}{\mathrm{u}}

\title{Fronts in the wake of a parameter ramp: slow passage through pitchfork and fold bifurcations}
\author{Ryan Goh\thanks{Department of Mathematics and Statistics, Boston University, 111 Cummington Mall, Boston,  MA 02215, USA; \texttt{rgoh@bu.edu}.},
 Tasso J. Kaper\thanks{Department of Mathematics and Statistics, Boston University, 111 Cummington Mall, Boston,  MA 02215, USA.},
  Arnd Scheel\thanks{School of Mathematics, University of Minnesota, 206 Church St. SE, Minneapolis,  MN 55455, USA.}, 
  Theodore Vo\thanks{School of Mathematics, Monash University, Clayton, Victoria 3800, Australia~}.}

\begin{document}

\maketitle

\abstract{

This work studies front formation in the Allen-Cahn equation with a parameter heterogeneity which slowly varies in space. In particular, we consider a heterogeneity which mediates the local stability of the zero state and subsequent pitchfork bifurcation to a non-trivial state.   For slowly-varying ramps which are either rigidly propagating in time or stationary, we rigorously establish existence and stability of positive, monotone fronts and give leading order expansions for their interface location. For non-zero ramp speeds, and sufficiently small ramp slopes, the front location is determined by the local transition between convective and absolute instability of the base state and leads to an O(1) delay beyond the instantaneous pitchfork location before the system jumps to a nontrivial state. The slow ramp induces a further delay of the interface controlled by a slow-passage through a fold of strong- and weak-stable eigenspaces of the associated linearization. We introduce projective coordinates to de-singularize the dynamics near the trivial state and track relevant invariant manifolds all the way to the fold point. We then use geometric singular perturbation theory and blow-up techniques to locate the desired intersection of invariant manifolds. For stationary ramps, the front is governed by the slow passage through the instantaneous pitchfork bifurcation with inner expansion given by the unique Hastings-McLeod connecting solution of Painlev\'{e}'s second equation. We once again use geometric singular perturbation theory and blow-up to track invariant manifolds into a neighborhood of the non-hyperbolic point where the ramp passes through zero and to locate intersections.

% normally hyperbolic invariant manifolds with a set of projective coordinates to track invariant manifolds through a non-hyperbolic point, all the way to the fold location. Then we use geometric blow-up to 
%

%
%For stationary parameter ramps, we 
%
%
%-fronts in Allen-Cahn presence of a slowly varying spatial ramp
%- for both stationary and rigidly propagating ramps
%- rigorously give existence and stability of monotone, and positive fronts, and determine the front interface location
%- for non-zero quenching speeds, and sufficiently small ramp slopes, find the leading order front location is determined by the local transition between convective and absolute stability, second order, $\epsilon$-dependent correction is determined by a slow passage through a saddle-node of strong and weak stable eigenspaces. 
%- for zero speed, dynamics governed by slow passage through a pitchfork bifurcation, with inner expansion given by the Hastings-McLeod unique connecting solution of Painlev\'{e}'s second equation.
%-Use techniques of geometric singular perturbation theory (normally hyperbolic invariant manifolds, Fenichel theory, and blow-up) to construct fronts. 



 }

%\underline{To-DO}
%\begin{itemize}
%%\item Decide on Section 1.3 - 1.4, and 5 - 7. Separate to a new paper?
%%\item Edit Abstract - too long now?
%%\item Proofread, address any of the blue comments throughout
%%\item Which proof to use in Prop. 6.2?
%%\item Any additional figures/schematics for Sections 5-7? (I.e. the in-out chart dynamics for $K_1$ or $K_3$? Or the heteroclinic dynamics for $K_2$?
%%\item Numerics: Sec 8:  $\alpha\in (0,1/2)$ numerics? (these are a little subtle and one has to be careful about round-off)
%%\item Shorten some of the routine calculations throughout?
%
%%\item Finish Discussion section 8: 
%\begin{itemize}
%%\item Sec 8.2: add a plot of solution profile with oscillating tail?
%%\item Sec 8.3: Another figure and some discussion of invasion in $\alpha\in(0,1/2)$ regime? Better/longer numerics in Fig. 6.4? Improve argument about envelope speeds/tail decay to understand the acceleration in a slow-ramp? 
%%\item Sec. 8.3: $c\sim \epsilon^{1/3}$
%\end{itemize}
%\end{itemize}

\section{Introduction }

%%%%%
The interaction of coherent structures, such as fronts, patterns, and waves, with spatio-temporal heterogeneity has recently attracted much interest in many scientific domains. Generally, one is interested how heterogeneities can nucleate, perturb, and mediate structures formed in a system. One such process which particularly motivates this work is that of directional quenching. Here, a heterogeneity travels across a medium, either controlled by the experimenter or another part of the system, rendering a stable equilibrium state unstable and hence nucleating the formation of a coherent structure in its wake. The speed and shape of the quenching mechanism then directly controls the structure formed in the wake. Examples of such mechanisms arise in fluid systems, phase-separative systems, chemical reactions, as well as biological applications; see \cite{goh21a} for a recent review. 

While the quenching heterogeneity often varies sharply in space, so that that medium is rendered strongly unstable at the quenching location, heterogeneities which are slowly varying in space are also prevalent in many applications. To name a few specific examples, we mention wavenumber selection in Rayleigh-Benard convection with slowly varying Rayleigh constant \cite{kramer82,riecke1987perfect}, oscillations in fluid flow past a slowly developing obstacle \cite{hunt91,chomaz99}, stripe orientation in morphogenesis due to gradients in production rates and parameters \cite{HISCOCK2015408}, and formation of cortex domain boundaries via spatially varying signal gradients \cite{feng21}. See also \cite{kuske} for theory about patterns in slowly varying environments and more applications.  A different but related set of phenomena arise in slowly-varying temporal heterogeneities, where pattern-formation is dynamically mediated with the slow evolution of some parameter, with examples arising in ecological systems \cite{rietkerk21}, soft-matter defects \cite{stoop2018defect}, and cosmological studies \cite{kibble1976topology,zurich85}.   Here the background medium is slowly rendered unstable in some fashion leading to a variety of effects, such as the selection of a specific wavenumber of striped pattern, the pinning of a front interface between two states at a certain location, or also the suppression of defect formation throughout the resulting coherent structure. See also \cite{knobloch2015problems}  and references therein for a recent review of related problems of pattern formation on time-varying domains.

\paragraph{Allen-Cahn model equation}
In this work, we wish to rigorously study front solutions in a prototypical partial differential equation with a slowly-varying directional-quenching mechanism. We study such fronts in the scalar Allen-Cahn equation as it will serve as an approachable but still relevant setting to rigorously characterize the interaction of the front with the slowly varying quench, without dealing with unnecessary technical complications of more realistic equations.  We expect our results to have bearing on similar interactions in other prototypical pattern forming systems with supercritical nonlinearities such as the Ginzburg-Landau equation and the Swift-Hohenberg equation, as well as more realistic models for the phenomena mentioned above.  We also remark that the sharp quenched case has been considered in Allen-Cahn, in both one- and two-spatial dimensions, in the works \cite{monteiro2017phase,monteiro2018horizontal}.
Our equation takes the form
\begin{align}\label{e:ac}
u_t = u_{xx} + \mu(x - ct) u - u^3,\quad (x,t)\in \R\times \R_+,\\
\mu(\xi) := -\tanh(\epsilon\xi),\quad 0<\epsilon \ll 1.
\end{align}
Here, as $\epsilon$ is small, the parameter heterogeneity, or ``ramp", slowly varies from -1 at $\xi := x- ct =+\infty$ to 1 at $\xi = -\infty$,  making the equilibrium $u = 0$ locally stable for $\xi:= x - ct>0$ and locally unstable for $\xi= x-ct<0$. Further, $c$ is an external control parameter which controls the speed at which the quench rigidly propagates through the medium. This particular quenching function is chosen as it is the solution of a simple first-order differential equation \eqref{e:tw0b}. While this quenching function simplifies the technical analysis, we expect similar phenomena to occur in a neighborhood of $\mu = 0$ with other slowly varying quenching terms, such as $\mu(\xi) = -\epsilon \xi$.

We study the formation of traveling front solutions $u(x-ct)$ which converge to $0$ at $x\rightarrow+\infty$ and $1$ at $x\rightarrow-\infty$.  Front solutions of this type satisfy the autonomous travelling wave ordinary differential equation
\begin{align}
0&= u_{\xi\xi} + cu_\xi + \mu u - u^3,\label{e:tw0a}\\
0&=\mu_\xi +\epsilon(1-\mu^2),\quad \mu(0) = 0.\label{e:tw0b}
\end{align}
%Here, we have used the fact that $\mu(\xi)$ is an explicit solution of a simple first-order equation to formulate the problem as an autonomous system. 
We report on front solutions for quenching speeds $c\in[0,2)$,
beginning with the dynamic quench with $c\in (0,2)$
in Sections \ref{ss:phen}-\ref{ss:thm1},
and then for the stationary quench $c=0$
in Sections \ref{ss:hm}-\ref{ss:thm2}.

\subsection{Fronts formed by a dynamic quench with $c\in (0,2)$: Phenomena
and numerics}\label{ss:phen}

The moving fronts created by a dynamic quench with $c \in (0,2)$
may be understood heuristically and numerically, as follows.
Figure \ref{f:profs} depicts front solutions to \eqref{e:tw0a}-\eqref{e:tw0b} 
obtained through numerical continuation in AUTO07p \cite{doedel2007auto} for a range of $\epsilon$ and $c$ values. 
For $\mathcal{O}(1)$ values of $c \in (0,2)$, we observe that, for large negative $\xi$, the solution tracks the quasi-stationary, or frozen coefficient, equilibrium value $\sqrt{\mu(\xi)}$. 
At some negative value of $\xi$, 
the solution profile quickly jumps down to values close to zero. We will later refer to this location as the front interface and denote the corresponding $\mu$ and $\xi$-values as $\mu_\mathrm{fr}$ and $\xi_\mathrm{fr}$, respectively; 
see \eqref{e:mfr}. 
For $\mathcal{O}(1)$ values of $c>0$ and for $0<\epsilon\ll1$, 
the central observation is that 
the front remains close to zero for an interval 
of length $\mathcal{O}(1)$ in $\mu$, or $\mathcal{O}(\epsilon^{-1})$  in $\xi<0$,
where $\mu(\xi)>0$ and the trivial state is unstable.  
The leading-order size of this interval may be determined asymptotically,
by studying the transition from absolute to convective instability.

\begin{figure}[h!]
\centering
\includegraphics[trim = 0.05cm 0.05cm 0.05cm 0.05cm,clip,width=0.5\textwidth]{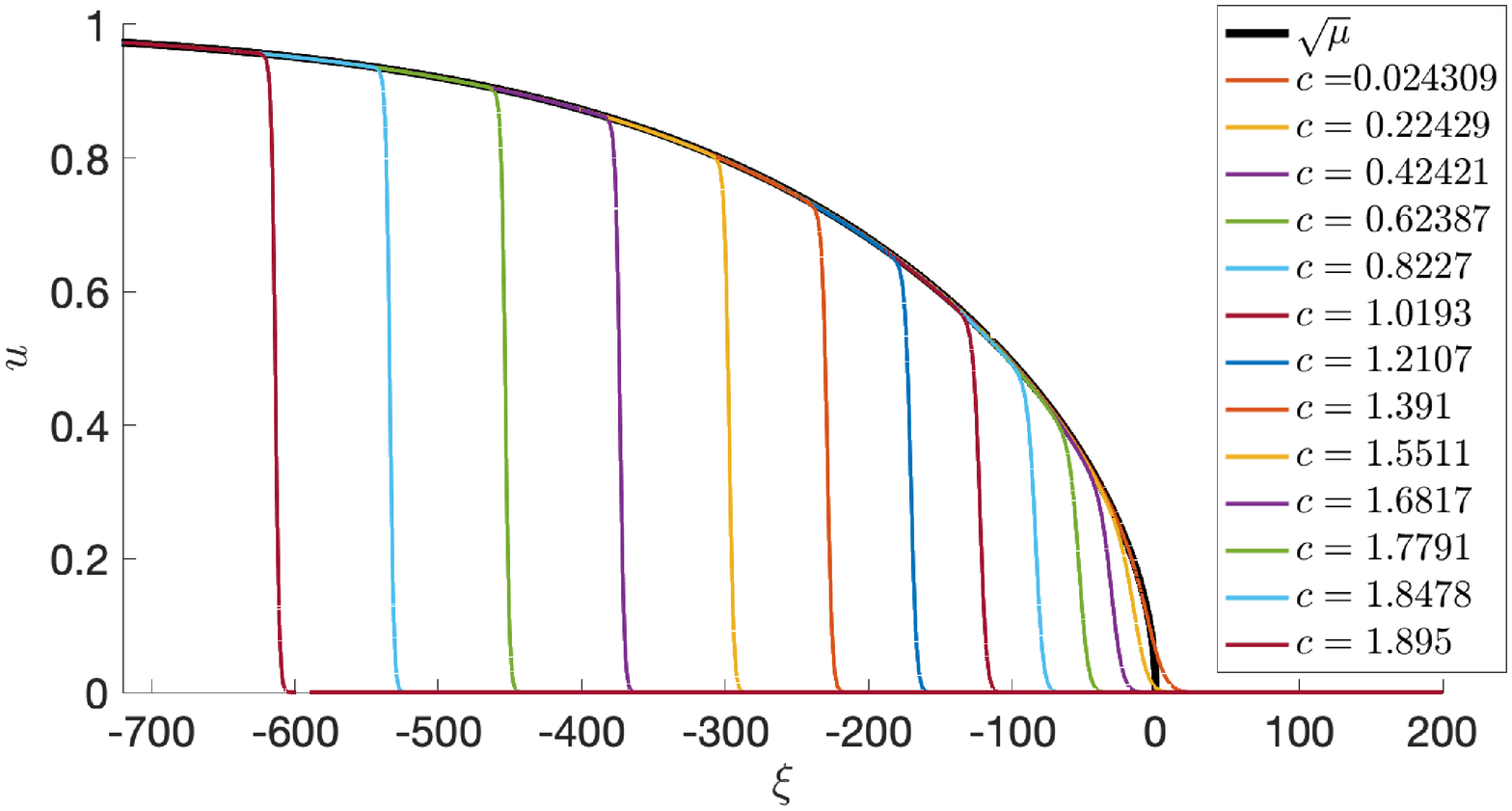}\hspace{-0.2in}
\includegraphics[trim = 0.05cm 0.05cm 0.05cm 0.05cm,clip,width=0.5\textwidth]{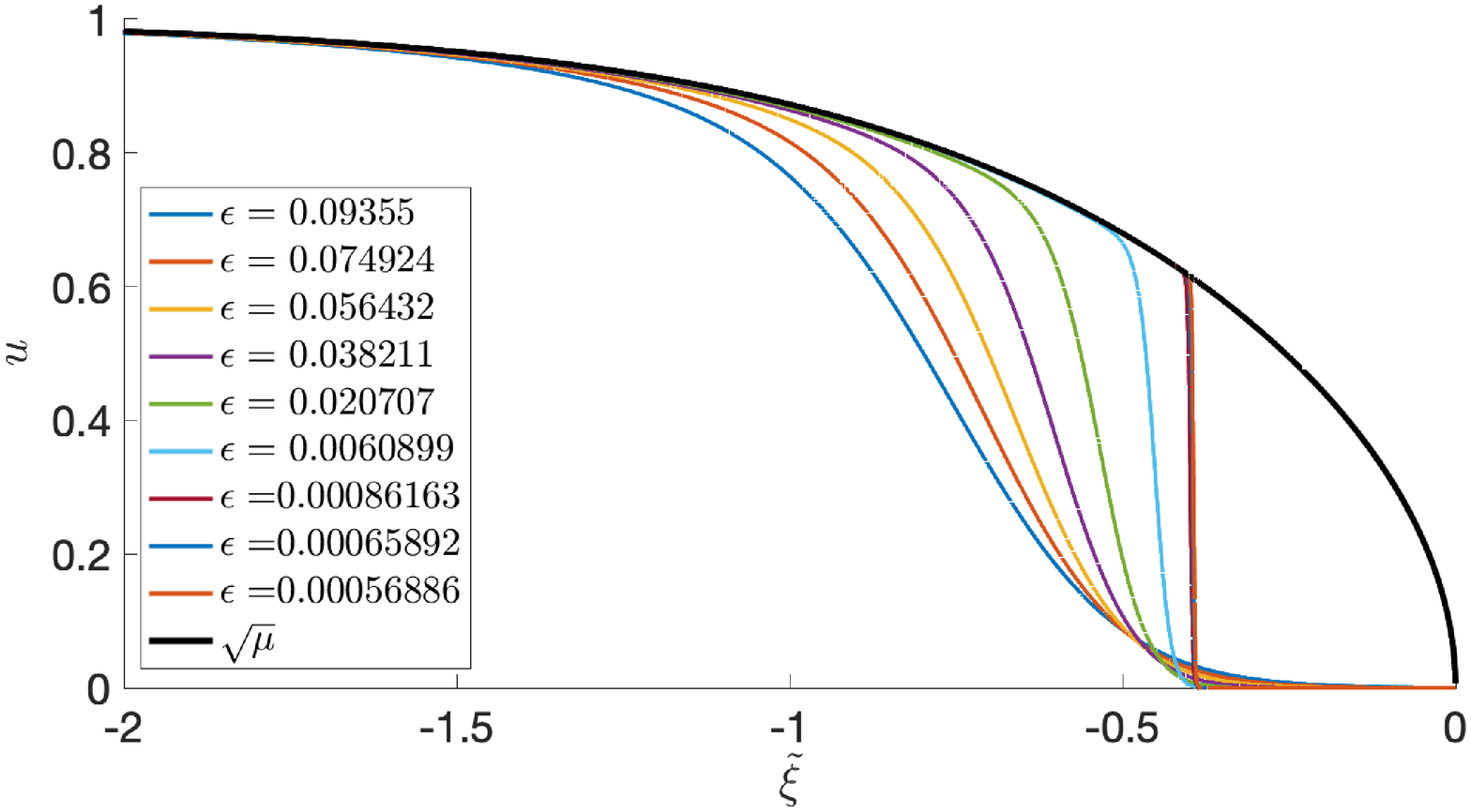}\hspace{-0.2in}\\
\caption{Results of numerical continuation of the traveling wave of \eqref{e:tw0a}-\eqref{e:tw0b} using AUTO07p. Left: Solutions $u(\xi)$ for a range of $c$-values (values in legend) with $\epsilon = 0.0025$ fixed. Right: Solutions $u$ against the rescaled variable $\tilde \xi = \epsilon\xi$ for a range of $\epsilon$ values (values in legend) with $c = 1.2$ fixed, along with $\mu(\, \tilde \xi \,)^{1/2}$ for $\tilde\xi<0$. %The front locations match closely with the theoretical prediction given by \eqref{e:mufr}.
 }\label{f:profs}
\end{figure}

\begin{figure}[h!]
\centering
\includegraphics[trim = 0.0cm 0.0cm 0.0cm 0.0cm,clip,width=0.32\textwidth]{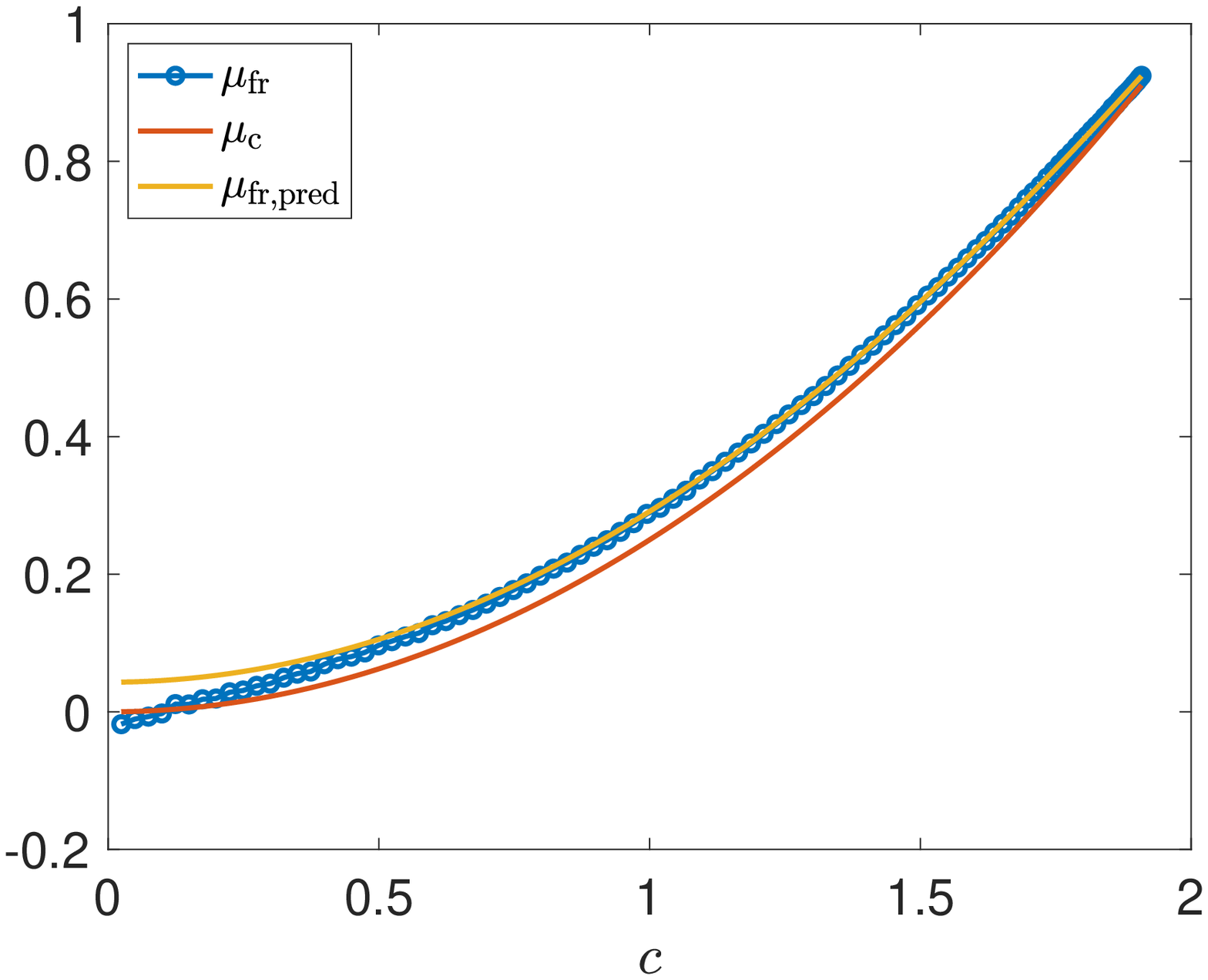}\hspace{-0.05in}
\includegraphics[trim = 0.0cm 0.0cm 0.0cm 0.0cm,clip,width=0.33\textwidth]{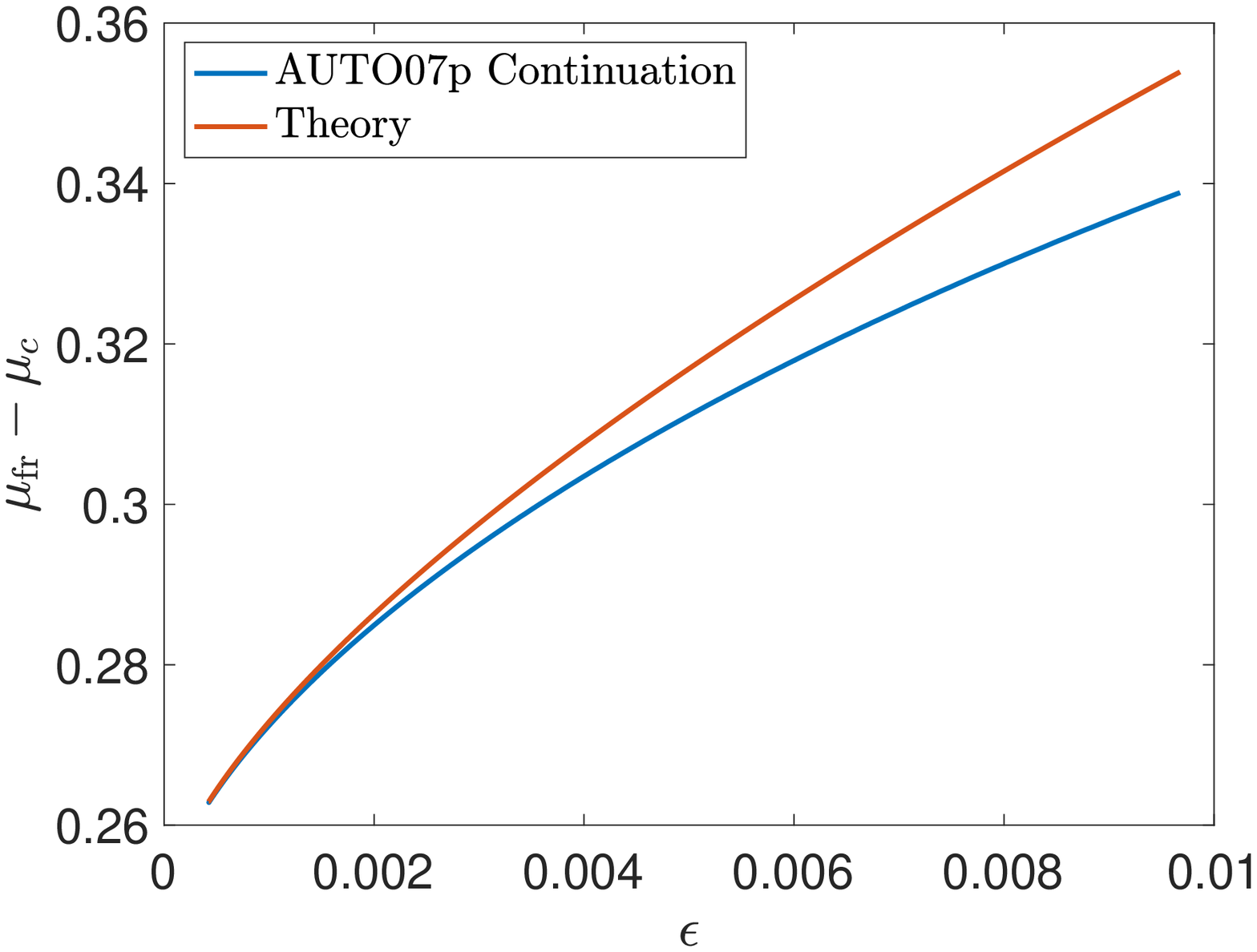}\hspace{-0.1in}
\includegraphics[trim = 0.0cm 0.0cm 0.0cm 0.0cm,clip,width=0.32\textwidth]{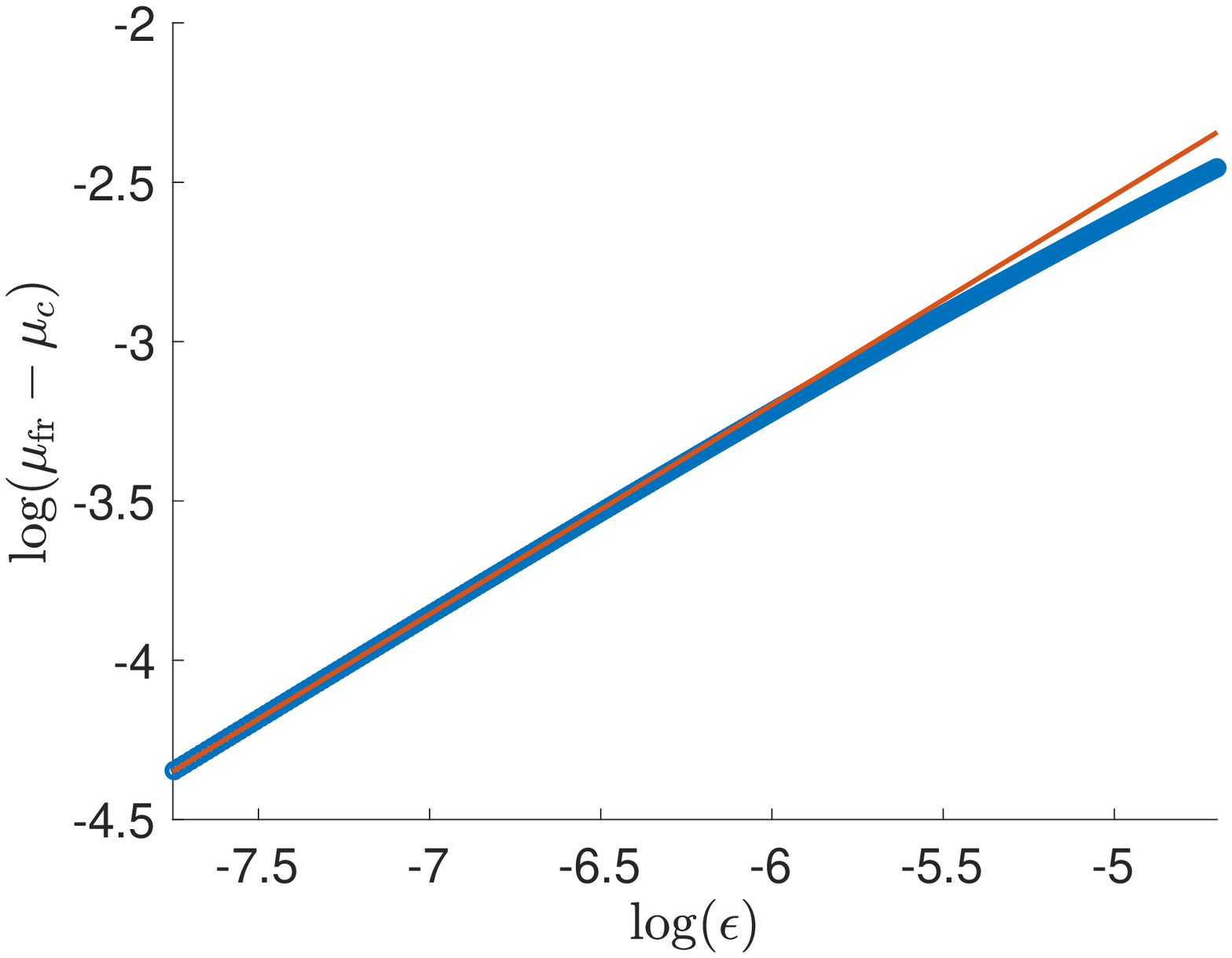}\hspace{-0.0in }
\caption{The front location $\mu_\mathrm{fr}$ obtained from numerical simulations and compared to the theoretical prediction $\mu_\mathrm{fr,pred}$. Left: plot of $\mu_\mathrm{fr}$ (blue circles), $\mu_c$ (orange line), and $\mu_\mathrm{fr,pred}$ (yellow line) given by the expansion \eqref{e:mufr} from Theorem \ref{t:0} for a range of $c$ values with $\epsilon = 0.0025$ fixed. Center: plot of $\mu_\mathrm{fr} - \mu_c$ comparing numerics (blue) and theoretical prediction \eqref{e:mufr} (orange) for $c = 1.2$. Right: Log-Log plot of $\mu_\mathrm{fr} - \mu_c$ against $\epsilon$ (blue), with linear fit (orange) of 10 left-most data points, also for $c = 1.2$. The measured slope is $0.650$. }\label{f:profs2}
\end{figure}

\paragraph{Absolute instability and the leading-order front interface} The leading order spatial delay in growth in the front interface behind the quenching threshold $\mu = 0$ is controlled by the transition between convective and absolute instability of the trivial state as $\mu$ increases towards 1 for decreasing $\xi$. We note this behavior was also observed and non-rigorously studied in the work \cite{chomaz99}. 
  To understand this in heuristic terms, consider the PDE \eqref{e:ac}, posed in the co-moving frame $\xi = x - ct$, with initial condition close to the trivial state except for a small localized perturbation centered around some $\xi>0$. As time increases, the perturbation will decay and be convected leftward until it reaches $\xi<0$ when it will start to weakly grow while still being convected leftward. Thus, at each fixed small $\xi<0$ the perturbation will decay pointwise.  This behavior will continue for more negative $\xi$ until $\mu$ is sufficiently large to induce pointwise growth, after which the front will grow to the non-trivial nonlinear state. This transition in growth type is known as the transition between convective and absolute instability. 

To further understand this, we briefly digress to summarize the concepts of absolute and convective instability of an equilibrium state. For more detailed discussions see \cite{sandstede2000absolute}. Consider the homogeneous Allen-Cahn equation, with $\mu$ a fixed constant, linearized around the trivial equilibrium $u = 0$, and posed in the co-moving frame with speed $c$,
\begin{align}\label{e:L}
v_t = v_{\xi\xi} + cv_\xi + \mu v =: L(\mu,c) v
\end{align}  
The trivial state is \emph{convectively unstable} if, for given $\mu,c>0$, a localized perturbation grows but is convected into the bulk at $\xi = -\infty,$ or in other words, the trivial state is unstable in the $L^2$-norm while locally at each point small perturbations decay over time. The state is \emph{absolutely unstable} if localized perturbations grow both in the $L^2$-norm and pointwise.  This transition can be located by studying the associated linear dispersion relation, obtained by inserting the ansatz $\re^{\lambda t + \nu \xi}$ into \eqref{e:L},
\begin{equation}\label{e:ldsp}
0 = d(\lambda, \nu,c) := \nu^2 + c\nu + \mu - \lambda.
\end{equation}
In the case of the Allen-Cahn equation, the transition between different types of instability is then obtained by finding $(\mu,c)$ values for which the branch point $(\lambda_\mathrm{br},\nu_\mathrm{br})$ of \eqref{e:ldsp} is marginally stable. That is, a $(\lambda,\nu)$-pair which solves
$$
0 = d(\lambda,\nu,c),\; 0 = d_\nu(\lambda,\nu,c),
$$
and satisfies $\mathrm{Re}\, \lambda_\mathrm{br} = 0$. Calculation gives 
$$
\lambda_\mathrm{br} = -c^2/4 + \mu ,\; \nu_\mathrm{br} = -c/2,
$$
so that the boundary between instabilities is given by the curve 
$$
\mu_c := c^2/4.
$$

Returning to the inhomogeneous system, and posing the time-dependent equation in the co-moving frame, one expects perturbations of the trivial state located near $\xi = 0$ to grow but be convected leftwards until reaching a $\xi$ value where $\mu(\xi)\geq\mu_c$. Here, they will also grow pointwise until being saturated at the level $u = \sqrt{\mu}$ through the nonlinear term. Thus, we define $\xi_c$ to be the value such that $\mu(\xi_c) = c^2/4.$

Further analyzing the numerical results depicted in Figure \ref{f:profs}, we find the slow-variation of the parameter ramp induces a secondary delay of instability and in the growth of the front, so that the front location, which we denote as $\xi_\mathrm{fr}$, is less than $\xi_c$ and the corresponding $\mu$-value, which we denote as $\mu_\mathrm{fr},$ is larger than $\mu_c.$ The numerics indicate the $\mu$-delay of the front interface varies like
$$
\mu_\mathrm{fr} - \mu_c \sim \epsilon^{2/3},
$$
consistent with our theoretical results below. Since $\mu \approx - \epsilon \xi$ for $\mu$ near 0, one would then expect the spatial delay to go like
$$
\xi_c - \xi_\mathrm{fr} \sim \epsilon^{-1/3},
$$
leaving a large plateau region where the front lies close to the now absolutely unstable trivial state. We discuss the implications of this delay on the stability of this front in Section \ref{ss:stab}.

For $c>2$, the trivial state is absolutely stable for all $\mu\leq 1$, hence small perturbations of the trivial state will be convected to negative infinity, and hence no front solution with this speed will exist.  In the original PDE, we expect such perturbations to grow and spread through the domain with asymptotic speed 2. It is of interest how the slowly varying quench alters the convergence of the front speed to this asymptotic rate. We briefly discuss this in Section \ref{s:homq}.

\subsection{Main existence result for dynamic fronts $c\in (0,2)$ \label{ss:thm1}}
As discussed above, we seek traveling wave solutions to the system of ODEs \eqref{e:tw0a}-\eqref{e:tw0b}
for $\mathcal{O}(1)$ values of $c \in (0,2)$.  
To simplify the setting, we reverse the spatial direction and consider solutions in $\zeta := -\xi$. We obtain the following traveling wave equation with asymptotic boundary conditions
\begin{align}
0=& u_{\zeta\zeta} - cu_\zeta + \mu u - u^3,\label{e:tw1a}\\
\mu_\zeta =& \epsilon (1-\mu^2),\label{e:tw1b}\\
\lim_{\zeta\rightarrow-\infty} u(\zeta) &= 0 
\hskip0.15truein {\rm and}
\hskip0.15truein 
\lim_{\zeta\rightarrow+\infty} (u(\zeta) - 1)  = 0.
\label{e:tw1c}
\end{align}
Note that now $\mu$ increases from $-1$ to $1$ as $\zeta$ increases. Further, we remark that all figures below depicting various aspects of the phase-portrait have direction of time governed by $\zeta$.
The desired solutions of this system 
are heteroclinic orbits between the equilibria 
$(u_0,v_0,\mu_-) = (0,0,-1)$ 
and 
$(u_+,v_+,\mu_+) = (1,0,1)$ 
in the following first-order system:
\begin{align}
u_\zeta &= v\label{e:tw2a} \\
v_\zeta &= cv - \mu u + u^3\label{e:tw2b}\\
\mu_\zeta &= \epsilon(1 - \mu^2).\label{e:tw2c}
\end{align}

These heteroclinic orbits will be found in the intersection of the unstable manifold, $W^\mathrm{u}(0,0,-1)$, of the former equilibrium and the  stable manifold, $W^\mathrm{s}(1,0,1)$, of the latter equilibrium. Since both of these are two-dimensional and lie in a three-dimensional ambient phase space, we expect a one-dimensional intersection and hence a locally isolated heteroclinic trajectory for each $\epsilon$ small.  Our result establishes the existence of such fronts and locates where their interface, or take off from the origin, is located. As observed in Figure \ref{f:profs}, the front has a fast jump from the trivial state up to local value of $\sqrt{\mu}$ when $\mu$ is near $\mu_c = c^2/4.$ Hence, to account for the $c$- and $\epsilon$-dependence of the front, we define the $\mu$-location of the front interface as
\begin{equation}\label{e:mfr}
\mu_\mathrm{fr} = \inf \{\mu\,:\, u>\sqrt{\mu_c}/2 \}.
\end{equation}
Since $\mu$ is one-to-one, we can then define $\zeta_\mathrm{fr}$ so that $\mu(\zeta_\mathrm{fr}) = \mu_\mathrm{fr}.$
Our main result is stated below.  See Figure \ref{f:uvpp} for a schematic of the phase portrait,
with insets in blue depicting local phase portraits for the singular system $\epsilon = 0$ near the origin.

\begin{Theorem}\label{t:0}
There exists an $\eps_0>0$ sufficiently small such that,
  for $0< \eps \le  \eps_0$ and for any $\mathcal{O}(1)$ value of
  $c \in (0,2)$, system \eqref{e:tw2a}-\eqref{e:tw2c} has a heteroclinic orbit
  $\Gamma_\eps$ which lies in the transverse intersection of
  $W^u(0,0,-1)$ and $W^s(1,0,1)$. Furthermore, $\Gamma_\eps$ is
  monotone increasing in $u$, and there exists a small
  $\tilde{\delta} > 0$ independent of $\eps$ such that $\Gamma_\eps$ is
  close to $(u,v)=(0,0)$ for $\mu \in [-1,\frac{c^2}{4} -\tilde{\delta})$
  and $\Gamma_\eps$ is close to $(u,v)=(\sqrt{\mu},0)$ for
  $\mu \in (\frac{c^2}{4} +\tilde{\delta}, 1]$.  The front location
  is given by
  \begin{equation}\label{e:mufr}
  \mu_{\rm fr}=\frac{c^2}{4}+\Omega_0\left(1-\frac{c^4}{16}\right)^{\frac{2}{3}}\epsilon^{\frac{2}{3}}
  +\mathcal{O}(\eps \ln(\eps)).
  \end{equation}
  Here, $\Omega_0$ is the smallest positive zero of the following
  combination of Bessel functions of the first kind,
  $J_{-1/3}(2 z^{3/2}/3) + J_{1/3}(2 z^{3/2}/3).$
\end{Theorem}\label{r:zetafr}

This theorem establishes the main result about quenched fronts for all
$\mathcal{O}(1)$ values of $c \in (0,2)$, 
showing that the monotone invasion fronts
 have interfaces located at $\mu = c^2/4$ to leading order, and not at $\mu=0$,
{\it i.e.}, not where the instantaneous pitchfork bifurcation occurs in
which $u=0$ becomes an unstable solution of the PDE. 
We observe that $\mu=c^2/4$ is where the 
unstable node, which is created at $\mu=0$ in the instantaneous
pitchfork bifurcation, becomes an unstable improper node,
on its way to transitioning to being an unstable spiral.
Hence, there is a substantial delay 
in the loss of stability of the $u=0$ in the PDE. The leading-order term gives an $\mathcal{O}(1)$ delay in $\mu$ which corresponds to an $\mathcal{O}(\epsilon^{-1})$ delay in $\zeta$. The next order term gives a further delay, where $\mu>\mu_c$ and the system is absolutely unstable, which is $\mathcal{O}(\epsilon^{2/3})$ in $\mu$ and thus $\mathcal{O}(\epsilon^{-1/3})$ in $\zeta$.
Moreover, in the
proof of the theorem, 
we use a projectivized coordinate
to track smoothly through $\mu=0$ and all the way up through $\mu=c^2/4$.
It turns out 
that there is a further delay in the loss of stability
({\it i.e.}, in $\mu_{\rm fr}$)
beyond $c^2/4$, 
which is of $\mathcal{O}(\eps^{2/3})$ duration, and 
this arises due to a slow-passage through a fold bifurcation 
in the projectivized system. %expressed in the projectivized coordinate.

\begin{figure}[h!]
\centering
\includegraphics[trim = 0.0cm 0.05cm 0.05cm 0.05cm,clip,width=0.7\textwidth]{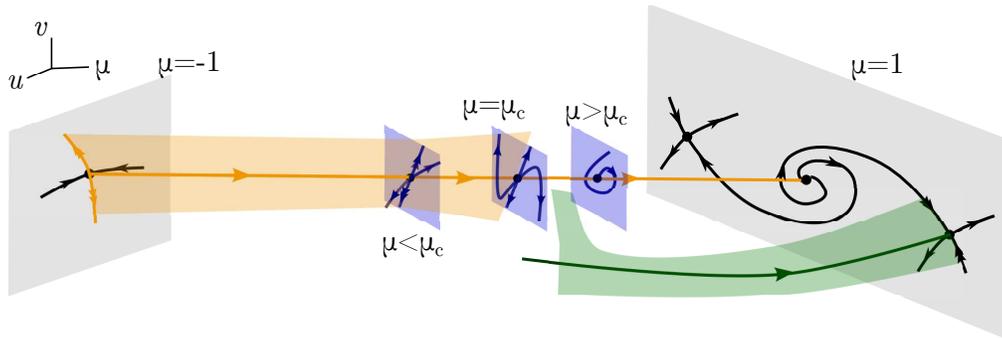}\hspace{-0.2in}
\caption{Schematic phase portrait for system \eqref{e:tw2a}-\eqref{e:tw2c} for $0<\epsilon\ll1$. Invariant planes $\{\mu = \pm 1\}$ depicted in grey, unstable manifold $W^\mathrm{u}(0,0,-1)$ in orange, stable manifold $W^\mathrm{s}(1,0,1)$ in green. Overlayed are $\mu = \text{constant}$ planes depicted in blue which are invariant for $\epsilon = 0$. }\label{f:uvpp}
\end{figure}

We use geometric singular perturbation theory to construct these heteroclinic orbits for each $\mathcal{O}(1)$ value of $c \in (0,2)$ in the singular limit $0<\epsilon \ll1$, first constructing the relevant manifolds for $\epsilon = 0$, where each plane $\{\mu = \text{constant}\}$ is invariant under the flow of \eqref{e:tw2a}-\eqref{e:tw2c}. We use a projective blow-up near the line $\{(0,0,\mu)\,;\, \mu\in[-1,1]\}$ to track the manifolds $W^\mathrm{u}(0,0,-1)$ and $W^\mathrm{s}(1,0,1)$ to a neighborhood of the point $(0,0,\mu_c)$ where an intersection can be constructed.  We use the projective coordinate $z = v/u,$ in combination with $u$ to track the evolution of linear subspaces near the origin as $\mu$ slowly varies. The eigenspaces of the $\epsilon = 0$ linear system are equilibria in the projective dynamics and collide in a fold bifurcation at $\mu_c$. For larger $\mu$, the corresponding eigenspaces become complex and hence the projective dynamics become oscillatory. This winding allows for subspaces to traverse more of the phase space, increasing the likelihood of an intersection. For $0<\epsilon\ll1$ these curves of equilibria perturb to normally hyperbolic invariant slow manifolds, with one-dimensional strong unstable fibers outside a neighborhood of $\mu_c$. To get around the loss of normal hyperbolicity near $\mu_c$, we use blow-up techniques to track the attracting slow-manifold and its unstable fibers around the fold where it can intersect $W^\mathrm{s}(1,0,1)$.

\subsection{Fronts created by a stationary quench ($c=0$): Phenomena and numerics}\label{ss:hm}

A stationary quench
is modeled by the PDE \eqref{e:ac}
with $c=0$.
Physically, the state $u=0$ is linearly unstable on the negative half of the domain 
and stable on the positive half. 
For small non-negative initial data, a stationary front forms, and its
profile is governed by the following spatial ODE:
\begin{equation}
\label{e:c0a-secondorder}
u_{\xi\xi} = 
-\mu u + u^3, \qquad
\mu_{\xi} = -\eps(1-\mu^2), \qquad \mu(0)=0,
\end{equation}
where $\xi=x-ct$ reduces to 
$\xi=x$. The front interface 
is controlled by the slow spatial ramping
through the pitchfork bifurcation,
which occurs at $\xi=0$, where $\mu = 0$. 
Indeed, 
in the three-dimensional $(u,v=u_\xi,\mu)$ phase space,
system \eqref{e:c0a-secondorder}
with $\eps=0$
has a pair of saddles at $(\pm\sqrt{\mu},0,\mu)$
and a center at $(0,0,\mu)$
for each $\mu>0$,
and these merge in a pitchfork bifurcation at $\mu=0$,
so that there is only a saddle fixed point 
at the origin for each $\mu<0$.
Then, for $\epsilon>0$ and small,
solution profiles of \eqref{e:c0a-secondorder}
are depicted in Figure~\ref{f:c0}.
The solutions 
lie near the curve $u=\sqrt{\mu}$ 
for large negative $\xi$, 
and near $u=0$
for large positive $\xi$.
In between, in a neighborhood of $\xi = 0$, 
the solutions slowly drop below $\sqrt{\mu}$ 
but then quickly rise above it,
with the exponentially decaying tail of the front being
located slightly ahead of the instantaneous bifurcation point $\mu=0$. 
Hence, the front interface appears to lie ahead of $\mu = 0$. 
From a PDE perspective, 
this advance of the front 
is caused by the lack of a drift term 
so that diffusion connects the front through a decaying tail across $\xi = 0$. 

\begin{figure}[h!]
\centering
\hspace{-0.2in}
\includegraphics[trim = 1.5cm 0.0cm 1.5cm 0.05cm,clip,width=1\textwidth]{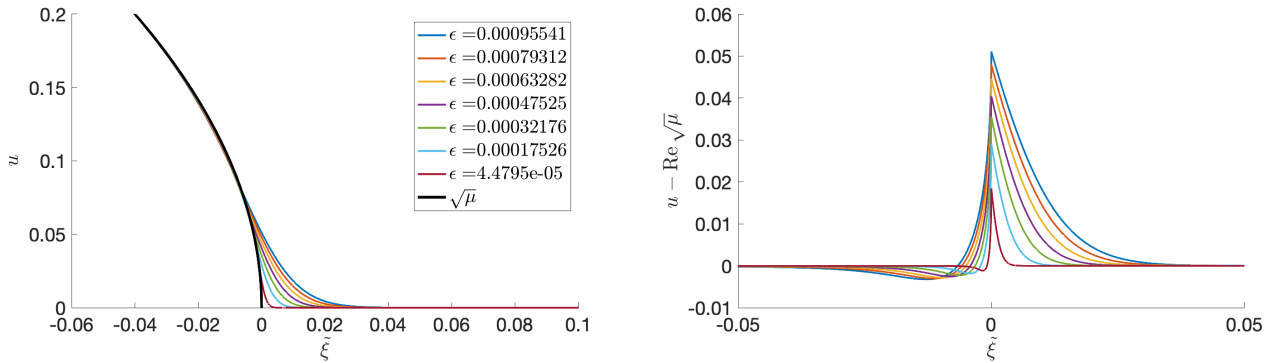}\hspace{-0.2in}
\caption{Left: the solutions of \eqref{e:c0a-secondorder},
i.e., the system with $c = 0$, for a range of small $\epsilon$ values, 
compared with $\sqrt{\mu}$ and zoomed in near $\xi = 0$; 
Right: the curves $u - \mathrm{Re}\, \sqrt{\mu}$ 
for the same range of $\epsilon$ values as in the left plot. 
Here $\tilde{\xi}= \eps\xi$.
}\label{f:c0}
\end{figure}

It turns out that the second Painlev\'{e} equation \cite{baik08,deift95,DLMF}
lies at the heart of system \eqref{e:c0a-secondorder}. 
This may be seen informally
by deriving the leading order asymptotics
for $0<\eps\ll 1$ as follows. 
Substitute the closed form expression 
$\mu(\xi)=-\tanh (\eps \xi)$ 
into \eqref{e:c0a-secondorder} to find
$u_{\xi\xi} = \tanh(\eps\xi) u + u^3.$
Next, scale $\eta=\eps^{1/3}\xi$ and
$u=\sqrt{2}\eps^{1/3}{\tilde u}$,
which corresponds to the significant degeneration
of the equation in the neighborhood of $\xi=0$ and $u=0$
where the instantaneous pitchfork bifurcation occurs.
The equation becomes
${\tilde u}_{\eta\eta} 
= \eps^{-2/3} \tanh(\eps^{2/3}\eta) {\tilde u} + 2 {\tilde u}^3,$
where the factor of $\sqrt{2}$ in the scaling of $u$
has put the coefficient on the cubic term into standard form.
Finally, Taylor expanding,
one obtains
\beq\label{e:pert-of-PII-intro}
{\tilde u}_{\eta\eta} 
= (\eta + \mathcal{O}(\eps^{4/3}\eta^3)) {\tilde u} + 2 {\tilde u}^3.
\eeq
Therefore,
we see that,
for any finite interval of values of $\eta$,
the parameter $\eps$ can be taken to be small enough so that
the equation is a perturbation 
of the second Painlev\'{e} equation (${\rm P}_{\rm II}$),
\begin{equation}
\label{e:PII-intro}
w_{\eta\eta} = \eta w + 2 w^3.
\end{equation}

The key solution of interest here is the Hastings-McLeod solution, $w_{\rm HM}(\eta)$ \cite{hastings80}, which is the unique positive, monotone solution of \eqref{e:PII-intro} which decays as $\eta\rightarrow+\infty$ and satisfies $w_{\rm HM}(\eta) \sim \sqrt{-\eta/2}$ as $\eta\rightarrow-\infty$. In more detail, it has the following asymptotics
\begin{align}
w_{\rm HM}(\eta) &\sim \mathrm{Ai}(\eta) \quad {\rm as} \ \eta \to \infty, \label{e:hma} \\
w_{\rm HM}(\eta) &\sim \sqrt{-\eta/2} \quad {\rm as} \ \eta \to -\infty, \label{e:hmb}\\
\frac{dw_{\rm HM}}{dx}(\eta) &< 0 \hskip0.2truein {\rm for} \hskip0.1truein {\rm all} 
\hskip0.1truein \eta.\label{e:hmc}
\end{align}
Here, $\mathrm{Ai}(\eta)$ denotes the Airy function. We note that all solutions of \eqref{e:PII-intro} which decay to zero as $\eta\rightarrow+\infty$ satisfy $w(\eta)\sim k\mathrm{Ai}(\eta)$ as $\eta\rightarrow+\infty$ for some $k\in \R$.  Parameterizing this family, $w_k(\eta)$, by $k\in\R$, we note that $w_{\rm HM} = w_1(\eta)$ partitions this family into two distinct classes. For $|k|>1$, the solution  $w_k(\eta)$ decays in oscillatory fashion as $\eta\rightarrow-\infty$. For $|k|<1$, the solution $w_k(\eta)$ has a pole at some finite point $\eta = c_0(k)<0$. That is $w_k(\eta) \sim \mathrm{sign}(k)/(\eta-c_0(k))$ as $\eta \to c_0(k)^+$, where we note that $c_0(k) \to -\infty$ as $\vert k \vert \to 1^+$. Proofs of these results can be found in \cite{hastings80}; see also Chapter 32 of the Digital Library of Mathematical Functions \cite[\S 32.11(ii)]{DLMF}, as well as \cite{clarkson2003painleve,Cleri_2020}. Also note, by symmetry, the solution with $k=-1$ is the other separatrix, with asymptotics $w_{-1}(\eta) \sim -\sqrt{-\eta/2}$ as $\eta\to -\infty$.

We note the solution $w_{HM}$ perturbs to a solution $\tilde{u}_{\rm HM}(\eta)$ of \eqref{e:pert-of-PII-intro}, which is the unique one satisfying the same asymptotic boundary conditions \eqref{e:hma} - \eqref{e:hmc}. Translating back to the original variables, we define 
\begin{align}
u_{HM}(\xi) = \sqrt{2}\epsilon^{1/3} w_{HM}(\epsilon^{1/3} \xi),
\end{align}
 which, for each $\epsilon$ sufficiently small,  formally gives the front of \eqref{e:c0a-secondorder} to leading order on any finite interval about $\xi = 0$. The numerically obtained solutions of the full system are compared to this rescaled Hastings-McLeod solution in Figure~\ref{f:c0p}. %We also note that additional intuition supporting the existence of a heteroclinic orbit can be obtained by analyzing the slowly varying Hamiltonian of the system. See Appendix \ref{a:2} for more detail. 

\begin{figure}[h!]
\centering
\hspace{-0.2in}
\includegraphics[trim = 0.0cm 0.0cm 0.0cm 0.0cm,clip,width=0.35\textwidth]{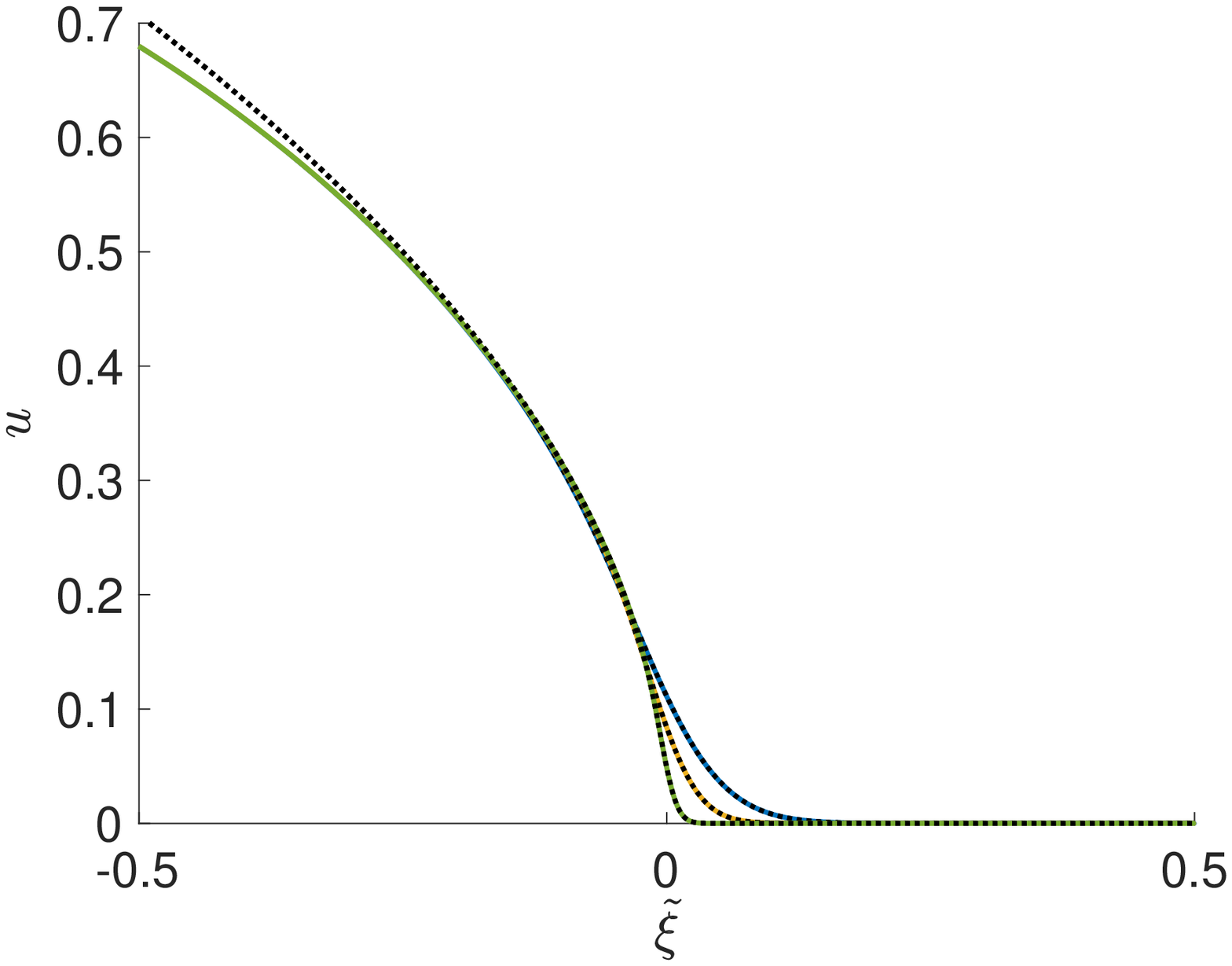}\hspace{-0.1in}\quad
\includegraphics[trim = 0.0cm 0.0cm 0.0cm 0.0cm,clip,width=0.35\textwidth]{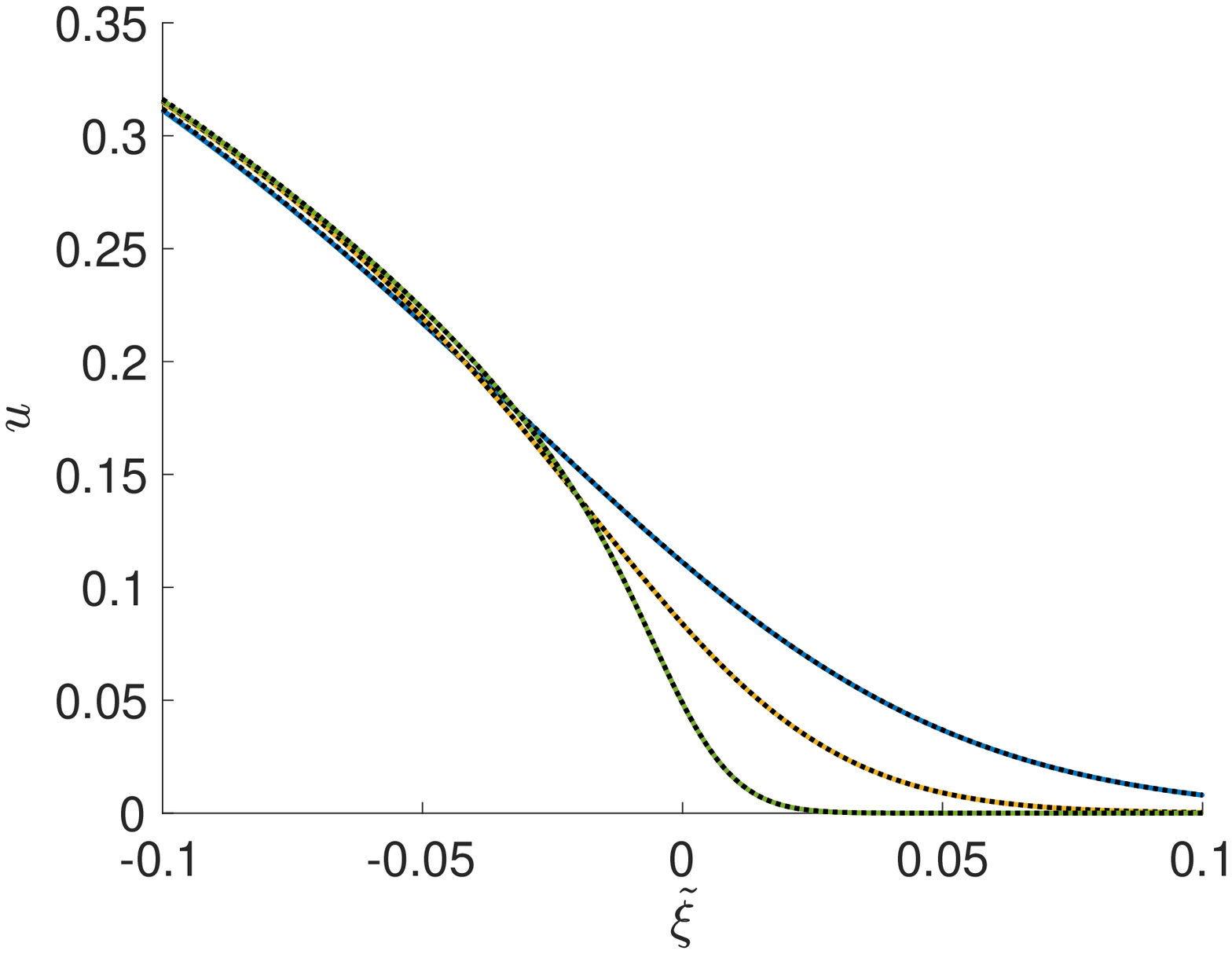}\hspace{-0.2in}\quad
\caption{Left: front solutions as functions of $\tilde \xi = \epsilon \xi$ 
from numerical continuation 
for  $\epsilon =9.81*10^{-3}$ (blue solid),
$\epsilon = 4.21*10^{-3}$(yellow solid), and $\epsilon = 8.27*10^{-4}$(green solid)  
with the rescaled connecting solution $u_{HM}$ of \eqref{e:PII-intro}  (black dashed). 
The blue, yellow, and green curves lie on top of each other for much of this plot. 
Right: zoom in of same solution profiles showing good agreement with the prediction 
$u_{HM}$. 
The numerical solution of \eqref{e:PII-intro} 
was obtained using the Matlab Chebfun package \cite{Driscoll2014}. }\label{f:c0p}
\end{figure}

%Additional intuition 
%to support the existence of the heteroclinics
%for $\eps>0$
%in \eqref{e:c0a-secondorder}
%comes from the $\mu$-dependent Hamiltonian,
%$H(u,v,\mu)= \frac{v^2}{2} + \frac{\mu u^2}{2} - \frac{u^4}{4},$
%in which the quadratic term in the potential
%has a slowly varying coefficient.
%Along solutions,
%the rate of change of $H$ in the time ($\xi$) is
%$\frac{dH}{d\xi} = \frac{\mu_\xi u^2}{2}.$
%For the existence of a heteroclinic 
%connecting $(u_+,v_+,\mu_+)$
%at which $H(1,0,1)= 1/4$
%to $(u_0,v_0,\mu_-)$
%at which $H(0,0,-1)= 0$,
%one needs
%$\Delta H = \int_{-\infty}^\infty \frac{dH}{d\xi} d\xi = - \frac{1}{4}.$
%A calculation using the system \eqref{e:c0a-secondorder}
%shows that, for all $\eps\ge 0$, there are orbits which
%satisfy this condition:
%\begin{equation}
%\int_{-\infty}^\infty \frac{dH}{d\xi} d\xi
%= \int_{-\infty}^\infty \frac{\mu_\xi u^2}{2} d\xi
%= \frac{\mu u^2}{2} \vert^\infty_{-\infty} 
%- \int_{-\infty}^\infty \mu u v d\xi 
%= \left(0 - \frac{1}{2} \right) 
%+ \int_{-\infty}^\infty v \left( \frac{dv}{d\xi} - u^3 \right) d\xi
%= - \frac{1}{2} 
%+\frac{v^2}{2} \vert_{-\infty}^\infty
%- \int_{1}^0 u^3 du 
%= - \frac{1}{4}.
%\end{equation}
%(Here, we note that the improper integrals exist,
%since $(u_+,v_+,\mu_+)$ and $(u_0,v_0,\mu_-)$
%are saddle points, so that the integrands
%decay exponentially.)
%Hence, this energy consideration suggests that
%the system should possess solutions
%which connect the two states
%also for $\eps > 0$.

\subsection{Existence result for stationary fronts with {\bf $c=0$} \label{ss:thm2}}

With the above intuition in mind, 
we state the main result for $c=0$.
The equation 
\eqref{e:c0a-secondorder}
may be written as a third-order autonomous system,
\begin{align}
\label{e:c0a-intro}
u_{\xi} &= v, \\
\label{e:c0b-intro}
v_{\xi} &= -\mu u + u^3, \\
\label{e:c0c-intro}
\mu_{\xi} &= -\eps(1-\mu^2), \qquad \mu(0)=0.
\end{align}
%There are fixed points at
%$(u_+,v_+,\mu_+)=(1,0,1)$ and 
%$(u_0,v_0,\mu_-)=(0,0,-1)$.
The front of \eqref{e:c0a-intro}-\eqref{e:c0c-intro}
is a heteroclinic orbit connecting the fixed points
$(u_+,v_+,\mu_+)=(1,0,1)$ to $(u_0,v_0,\mu_-)=(0,0,-1)$, 
and it lies in the transverse intersection
of the unstable and stable manifolds
of these fixed points, respectively.

\begin{Theorem}\label{t:2}
For each $\eps>0$ sufficiently small, 
system \eqref{e:c0a-intro}-\eqref{e:c0c-intro} 
has a heteroclinic orbit $(u^*,v^*,\mu^*)$ in the transverse intersection of $W^u(1,0,1)$ and $W^s(0,0,-1)$. Furthermore, there exists a small $\rho>0$, independent of $\epsilon$, such that the front satisfies
\beq\label{e:uest}
%|u^*(\xi) - u_{HM}(\xi)| < \rho \epsilon^{2/3}.
u^*(\xi) = u_{HM}(\xi) + \mathcal{O}(\epsilon^{2/3}),\quad \text{for all $|\xi|\leq \rho \epsilon^{-1/3}.$ }
\eeq
where $u_{HM}(\xi) = \sqrt{2}\epsilon^{1/3} w_{HM}(\epsilon^{1/3}\xi)$ and $w_{HM}$ is the unique Hastings-McLeod solution of \eqref{e:PII-intro}.

\end{Theorem}
The estimate \eqref{e:uest} implies that the scaled Hastings-McLeod solution gives the leading order inner solution in the region $|\mu|\leq \rho\epsilon^{2/3}$. A comparison of the leading-order inner solution with the numerically-obtained front solutions for two small $\epsilon$ values is given in Figure~\ref{f:c0p}. We add that, to make it easier to compare solutions for different $\epsilon$, the solutions are plotted against the variable $\tilde \xi = \epsilon \xi$, and the leading order asymptotics take the form
\begin{align}\label{e:hm}
u^*(\tilde \xi/\epsilon) 
\sim \epsilon ^{1/3} w_{HM}(\epsilon^{-2/3} \tilde \xi ), %\tilde u_{HM}(\epsilon^{-2/3} \tilde \xi ), 
\qquad |\tilde\xi|\lesssim \epsilon^{2/3}.
\end{align}

This theorem is proven using geometric desingularization, 
or ``blow-up" of system \eqref{e:c0a-intro}-\eqref{e:c0c-intro}, near the point $(u,v,\mu) = 0$, where the critical manifolds for $\epsilon = 0$ lose hyperbolicity. Here this point is blown-up into a 3-sphere whose singular dynamics are controlled by Painlev\'{e}'s second equation at leading order. We use inclination lemmas to track the desired invariant manifolds into a neighborhood of the sphere. We then use exponential trichotomies to lift the transversality of the Hastings-McLeod solution on the singular sphere and track $W^u(1,0,1)$ and $W^s(0,0,-1)$ across the sphere and show they also intersect transversely.

As part of the
analysis here of the fronts created by a stationary quench, we show that
the Hastings-McLeod solution lies in the transverse intersection of
invariant stable and unstable manifolds of \eqref{e:PII-intro}. In the extended
phase of \eqref{e:PII-intro}, these manifolds consist of solutions which satisfy
exponential growth and decay conditions as $\eta  \to \infty$ and $w \to 0$
and of solutions satisfying exponential growth and decay conditions as $\eta
\to -\infty$ and $w \to \sqrt{-\eta/2}$. 
As discussed above, the Hastings-McLeod solution is the
unique solution of the Painlev\'{e} II equation which separates two different
types of solutions. Namely, among all solutions that decay asymptotically
proportionally to an Airy function as $z \to +\infty$, it separates those
which undergo oscillatory decay as $z \to -\infty$ from those which have a
simple pole at some negative value of $z$. These two different classes of
solutions lie on different sides of the transverse intersection of the
stable and unstable manifolds. Moreover, establishing this transverse
intersection for \eqref{e:PII-intro} is also a natural building step for showing
that the stationary front of the PDE \eqref{e:ac} lies in
the transverse intersection of invariant manifolds.

Physically, the Allen-Cahn type PDE studied here may also be
viewed as a prototype system for studying more general problems in which
there is a slowly-varying parameter ramp in space. Such situations arise
for example in Taylor vortex flow when there is a time-independent
parameter ramp which varies slowly in space \cite{riecke1986pattern,rehberg1987forced}. The governing equations are much more complex there, but
experimental results and asymptotic analysis shows that the slowly-varying
spatial ramp can induce the selection of a unique pattern \cite{riecke1986pattern,riecke1987perfect}.

\begin{Remark}\label{r:ks}
The $c=0$ case is an important example of slow passage through a
super-critical pitchfork bifurcation in multiple time scale dynamical
systems, with two fast variables and one slow variable. Indeed, this is a
natural next step in the use of geometric desingularization to analyze
dynamic pitchfork bifurcations, which has previously only been done for
systems with one fast variable and one slow variable, see
\cite{Krupa_2001}. The geometry induced by the two fast variables
requires the tracking of additional hyperbolic directions.
\end{Remark}

\begin{Remark}\label{r:hm}
Earlier analyses of slow passage through pitchfork bifurcations
have involved the case of a generic center equilibrium undergoing a
slow dynamic pitchfork bifurcation 
in which the center becomes a saddle and two new
centers emerge. In Hamiltonian mechanics, this corresponds to a single
well potential slowly changing into a double well. These earlier analyses
\cite{haberman79,maree96} were carried out using singular perturbation theory
and matched asymptotic expansions. In contrast, because the pitchfork
bifurcation encountered here is of the opposite type, with a saddle point
becoming a center and giving birth to two saddles
(and as a result the full Allen-Cahn PDE 
transitions from one stable state to another),
a rigorous analysis is
possible by exploiting the hyperbolicity on both sides of $\mu=0$
and by using geometric
desingularization to study the loss of hyperbolicity 
in a neighborhood of $\mu=0$.
Also, in principle, one could use a complex time variable,
obtain the formal asymptotic results here
from the the earlier works \cite{haberman79, maree96}.
\end{Remark}

%\begin{Remark}
%In \cite{Krupa_2001}, the method of geometric desingularization is used to study slow passage through a generic pitchfork bifurcation
%in systems with one slow variable and one fast variable. Here, there are two fast variables, so that the slow passage through the pitchfork bifurcation occurs in a second-order differential equation, requiring the tracking of additional hyperbolic directions.
%\end{Remark}

\subsection{Outline}
The analysis of PDE \eqref{e:ac} 
in the case of $c \in (0,2)$
and the proof of Theorem \ref{t:0}
are presented in Sections \ref{s:set}-\ref{s:ori}.
In particular, in Section \ref{s:set}, 
we set up our theoretical approach, 
define the projective coordinates,
and describe the singular system with $\epsilon = 0.$ 
In Section \ref{sec:eps.gt.0}, 
we use Fenichel theory and geometric blow-up 
to unfold the dynamics and track the relevant invariant manifolds 
for $0<\epsilon\ll1.$ 
Then, in Section \ref{s:ori}, 
the desired heteroclinic intersection is established
in a neighborhood of the dynamic fold,
hence completing the proof of Theorem \ref{t:0}.
Next, the analysis of PDE \eqref{e:ac} 
in the case of $c=0$
and the proof of Theorem \ref{t:2}
are presented in Sections \ref{s:c0}-\ref{s:incl}.
In Section \ref{s:c0},
we begin the study of stationary fronts in the $c = 0$ case, 
using a geometric blow-up 
of a neighborhood 
of the instantaneous pitchfork bifurcation point.
Then, Section \ref{s:singhet}
establishes that 
the Hastings-McLeod solution of \eqref{e:PII-intro}
exists in the transverse intersection
of invariant manifolds, and then that 
the singular heteroclinic 
representing the stationary front created by the quench
also exists in the transverse intersection
of invariant manifolds of the full system.
The proof of Theorem 2 is completed
in Section \ref{s:incl},
by establishing the inclination properties of invariant manifolds,
and showing that the transverse intersection
exists for all $0< \epsilon \ll 1$.
In Section \ref{s:disc}, 
we complement the proofs of Theorems \ref{t:0} and \ref{t:2}
by giving an argument showing the fronts of Theorem \ref{t:0} 
are nonlinearly asymptotically stable, 
discussing the existence of other, non-monotonic front solutions 
possible in the wake of the quench for $c \in (0,2)$, 
and discussing parameter regimes not covered by our result, 
such as the $c,\epsilon \sim0$ regime. We provide additional numerical results to motivate future studies, as well as discuss other slowly-varying heterogeneities which we expect to induce novel front invasion behavior.

\section{Setup for traveling-waves with $c \in (0,2)$}\label{s:set}
In this section, and in Sections 3-4, we consider $\mathcal{O}(1)$ values of
the speed $c \in (0,2)$. We linearize system \eqref{e:tw2a} - \eqref{e:tw2c} about the equilibria $(u_0,v_0,\mu_-)$ and $(u_+,v_+,\mu_+)$.
The Jacobian at $(u_0,v_0,\mu_-)$ has eigenvalues
$$
\nu_{-,\epsilon} = 2\epsilon,\qquad \nu_{-,\pm}=\frac{c}{2} \pm \sqrt{\frac{c^2}{4}-\mu_-} = \frac{c}{2} \pm \sqrt{\frac{c^2}{4}+1}.
$$
Thus, it is a hyperbolic saddle with two-dimensional unstable manifold $W^\mathrm{u}(0,0,-1)$, whose tangent space is spanned by the vector $(1,\nu_{-,+},0)$ in the $\mu\equiv-1$ plane and by the vector $(0,0,1)^T$ in the direction of the $\mu-$axis. 
Then, the Jacobian 
at $(u_+,v_+,\mu_+)$ has eigenvalues
$$
\nu_{+,\epsilon} = -2\epsilon, \quad \nu_{+,\pm}=\frac{c}{2} \pm \sqrt{\frac{c^2}{4}-\mu_++3u_+^2} =\frac{c}{2} \pm \sqrt{\frac{c^2}{4}+2} .
$$
Thus, it is a hyperbolic saddle, with two-dimensional stable manifold, $W^\mathrm{s}(1,0,1)$, whose tangent space is spanned 
by the vector $(1,\nu_{+,-},0)^T$ in the $\mu\equiv1$ plane and by the vector $(0,0,1)^T$ in the direction of the $\mu-$axis. 
As mentioned above, we wish to locate intersections $ W^\mathrm{u}(0,0,-1)\cap W^\mathrm{s}(1,0,1)$, which consists of a pair of two-dimensional manifolds in three dimensional space, indicating we generically expect a one-dimensional intersection of these manifolds and hence a locally unique trajectory for each $0<\epsilon\ll1$.

\subsection{Projective coordinates/blow-up}
For $\epsilon = 0$, each $\mu = \text{constant}$ plane is invariant with equilibria $(0,0,\mu)$ for all $\mu$ and $(\pm\sqrt{\mu},0,\mu)$ for $\mu\in[0,1]$.  The latter are saddles for all $\mu\in (0,1]$. The former is a hyperbolic saddle for $\mu<0$, degenerate unstable node for $\mu = 0$, unstable node for $\mu\in (0,c^2/4)$. It is a degenerate source for $\mu = c^2/4$ with two-dimensional Jordan block, and is an unstable spiral for $\mu\in(c^2/4,1]$. We remark that the algebraically-double eigenvalue found at $\mu = c^2/4$ is also located using the double-root calculation given in Section \ref{ss:phen} above. In order to unfold the dynamics near $(u,v) = (0,0)$ for $\mu\in[-1, 1]$ and $0<\epsilon\ll1$, we perform a directional blowup in the variables
\beq\label{e:pc}
\tilde z = v/u,\quad u.
\eeq
See \cite{Holzer_2012} for a recent work using a similar approach in a different context.
These coordinates allow one 
to track the manifold $W^\mathrm{u}(0,0,-1)$ 
from $\mu = -1$ through the change in linear stability at $\mu =0$ and through the Jordan block at $\mu = c^2/4$.

In the coordinates \eqref{e:pc}, the system \eqref{e:tw2a}--\eqref{e:tw2c} becomes
\begin{align}
\tilde z_\zeta &= -\tilde z^2 + c\tilde z -(\theta + c^2/4) + u^2,\label{e:tildez-sysa}\\
u_\zeta &= \tilde z u,\label{e:tildez-sysb}\\
\theta_\zeta &= \epsilon (1 - (\theta + c^2/4)^2), \label{e:tildez-sysc}
\end{align}
where we have also set $\theta  := \mu - c^2/4$ to translate the point $\mu = c^2/4$ to the origin.
Here, $\mu_+ = 1$ corresponds to $\theta_+ := 1-c^2/4$ and $\mu_- = -1$ to $\theta_- := -1 -c^2/4.$  

\begin{Remark}
In order to unfold the dynamics in the region near the origin, one generally would blow up the line of equilibria $(0,0,\mu)$ into a cylinder via a polar coordinate blow up $u = r\cos \phi,\,\, v = r \sin \phi.$
%\begin{equation}\label{e:pol}
%u = r\cos \phi,\,\, v = r \sin \phi.
%\end{equation}
 Such a coordinate change, while elucidating the small amplitude dynamics, would push the non-trivial equilibria $(u,v) = (\sqrt{\mu},0)$ away to infinity in the limit $r\to0$, requiring multiple coordinate charts to construct the intersection.  Hence, we instead perform a directional blow-up, projecting the dynamics on different charts of the cylinder using blow-up in both the $u$ and $v$ directions,
$\tilde z = v/u, \, u$ and $ \tilde w = u/v, \, v,$  respectively.
We find that only the first chart is required to construct the monotonic front given in Theorem \ref{t:0}. We also note that both charts, or the aforementioned cylindrical blow-up, would be needed to construct non-monotonic fronts with oscillatory tails. 
See Section \ref{ss:nonmon} and Figure \ref{f:cyl} for more discussion on the non-monotonic fronts.
\end{Remark}

There are several key features  of system \eqref{e:tildez-sysa}-\eqref{e:tildez-sysc}. A central feature is that the plane
\begin{equation}
U_0 = \{ u = 0 \}
\end{equation}
is invariant for all $\epsilon\geq0$.  With $\epsilon = 0$, $\theta$ is a constant, and $U_0$ contains the equilibria of \eqref{e:tildez-sysa} -\eqref{e:tildez-sysc}, which are at
$(\tilde z,u,\theta) = (\tilde z_\pm(\theta), 0,\theta)$ for each $\theta \in [-1-c^2/4,0]$.
Here, $\tilde z_\pm$ satisfies
$$
-\tilde z^2 + c\tilde z -(\theta + c^2/4) =0, 
\quad \mathrm{Re} \, \tilde z_+ \geq c/2.
$$
These equilibria collide 
in a saddle-node bifurcation at $\theta = 0$ (that is $\mu = c^2/4$) 
and $\tilde z = c/2$. 
Also, at $\theta = -c^2/4$ (that is $\mu = 0$)
there is a pitchfork bifurcation
from the point $(\tilde z_-,0,\theta)$ 
in which a branch of equilibria emerges
$$
(\tilde z_*, u_*, \theta) = (0, \sqrt{\theta+c^2/4},\theta),
\qquad \theta\in(-c^2/4,1-c^2/4].
$$
These lie out of the plane $U_0$ and correspond to the non-trivial state $(u,v) = (\sqrt{\mu},0)$. Due to reversibility, there is also a branch of equilibria $(0, -\sqrt{\theta+c^2/4},\theta)$ for the same interval of fixed $\theta$ values, which correspond to the other non-trivial state $(u,v) = (-\sqrt{\mu},0)$ 
that also bifurcates at $\theta = -c^2/4$.

For $\epsilon>0$, 
only the points 
$(\tilde z_\pm,0, \theta_-)$ and $(\tilde z_*,u_*,\theta_+)$ persist as equilibria,
and only the planes $\theta = \theta_\pm$ remain invariant. 
Moreover, on the invariant plane $U_0$, 
the flow of \eqref{e:tildez-sysa} - \eqref{e:tildez-sysc} with $0<\eps\ll 1$
is governed by an algebraic Ricatti-equation,
which tracks the evolution of 1-D subspaces 
of the $(u,v)$-linearized dynamics
and which can be put into the normal form for slow-passage through a fold.

\begin{Remark}
The dynamics on the invariant plane $U_0$ correspond to the dynamics on the blown-up cylinder 
induced by the linear flow, and the reduced $\tilde z,\theta$ system tracks the dynamics of one-dimensional subspaces in the Grassmanian $Gr(2,1)$ under the linearized flow.  Here, when $\epsilon = 0$, equilibria of the projectivized flow, determined by the $\tilde z$-equation, are given by spatial eigenvalues $\nu$  of the $(u,v)$-linearization about the origin determined by the linear dispersion relation \eqref{e:ldsp}.
\end{Remark}

In order to study the dynamics on $U_0$ and those of the full system \eqref{e:tildez-sysa} - \eqref{e:tildez-sysc}, we make one further simplifying step. In particular, we complete the square $\tilde z = z+ c/2$, obtaining
\begin{align}
z_\zeta &= - z^2 -\theta + u^2,\label{e:zut1}\\
u_\zeta &=  (z + c/2) u,\label{e:zut2}\\
\theta_\zeta &= \epsilon (1 - (\theta + c^2/4)^2)\label{e:zut3}.
\end{align}
We shall work with this system in Sections \ref{s:set} - \ref{s:ori} to establish the main results for the heteroclinic orbit $\Gamma_\eps$,
and prove Theorem \ref{t:0}.
In the next subsection, we first study the $\eps=0$ system. Then, in the subsequent subsections, we will analyze the dynamics for $0< \eps \ll 1$ and show that there is a transverse intersection of the unstable manifold of $(z_+,0,\theta_-)$ and the stable manifold of $(z_*,u_*,\theta_+)$, for sufficiently small $\epsilon>0$. The heteroclinic $\Gamma_\eps$ will lie in that intersection,
see Figure \ref{f:eps1} for a depiction.

\begin{figure}[h!]
\centering
\includegraphics[trim = 0.0cm 0.0cm 0.0cm 0.05cm,clip,width=0.9\textwidth]{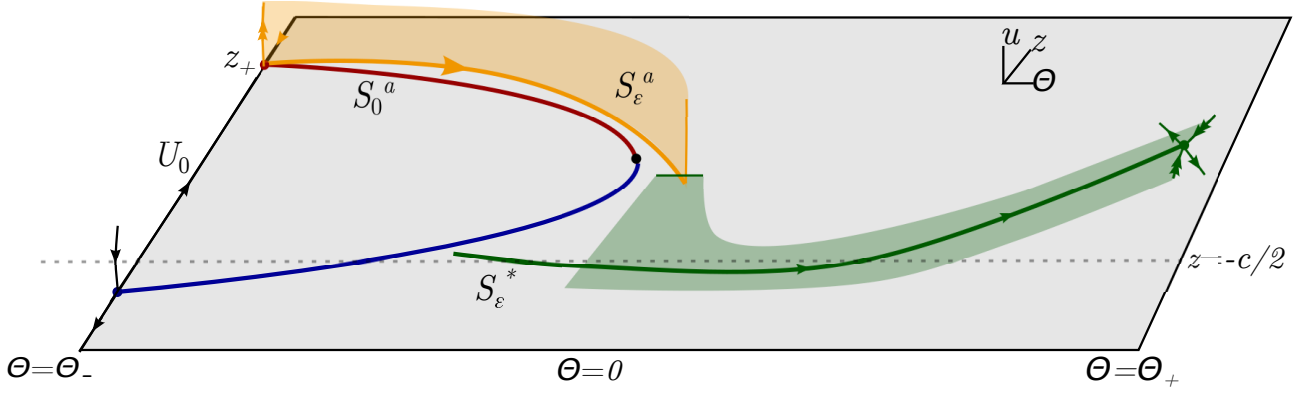}\hspace{-0.2in}
\caption{ Phase portrait for \eqref{e:zut1}-\eqref{e:zut3} for $0<\epsilon\ll1$. Red and blue curves give the critical attracting and repelling sets $S^{a/r}_0$, contained in the invariant plane $\{ u = 0\}$ (black), which make up the fold curve for $\epsilon = 0$. $S^a_\eps$, depicted in orange, gives the perturbed slow manifold for $\epsilon>0$. Green curve gives the perturbed slow manifold for the $\epsilon = 0$ equilibrium curve $(z_*,u_*,\theta)$.   Unstable manifold $W^\mathrm{u}(z_+,0,\theta_-)$ in orange foliated over $S^a_\eps$, stable manifold $W^\mathrm{s}(z_*,u_*,\theta_+)$, each two-dimensional with one slow dimension and one fast dimension.  }\label{f:eps1}
\end{figure}

\subsection{The $\epsilon = 0$ dynamics}
We next study the $\epsilon = 0$ limit of \eqref{e:zut1} - \eqref{e:zut3}. 
For $\epsilon = 0$, the planes $\{\theta = \mathrm{constant}\}$ are invariant.
The phase portraits on these invariant planes are depicted 
in Figure \ref{f:eps0}. 
The equilibria are now represented by
$$
(z_\pm(\theta),u) = (\pm\sqrt{-\theta},0), \quad\theta\leq0;\qquad\qquad (z_*,u_*(\theta)) = (-c/2,\sqrt{\theta+c^2/4}),\qquad \theta>-c^2/4. 
$$
The equilibria $(z_+,0,\theta)$ are stable 
in the $z$-direction for all $\theta<0$ 
and unstable in the $u$-direction for all $\theta\leq0.$ 
We let $\tilde W^\mathrm{u}(z_+,0,\theta)$ 
denote the 1-D unstable manifold of $(z_+,0,\theta)$. 
The equilibria $(z_-,0)$ are unstable 
in the $z$-direction for all $\theta<0$.
Then, in the $u$-direction, they are stable for $\theta<-c^2/4$ 
and unstable for $\theta\in (-c^2/4,0]$. 
Finally, the other equilibria 
$(z_*,u_*,\theta)$ 
of \eqref{e:zut1}-\eqref{e:zut3} with $\eps=0$
have one-dimensional stable manifolds, $\tilde W^\mathrm{s}(z_*,u_*,\theta)$. 
The bounded portions of these manifolds
converge in backward time for $\theta\leq0$ to the equilibrium $(z_-,0,\theta)$
(as may be seen from a null-cline analysis).

\begin{figure}[h!]
\centering
\includegraphics[trim = 0.0cm 0.0cm 0.0cm 0.0cm,clip,width=0.7\textwidth]{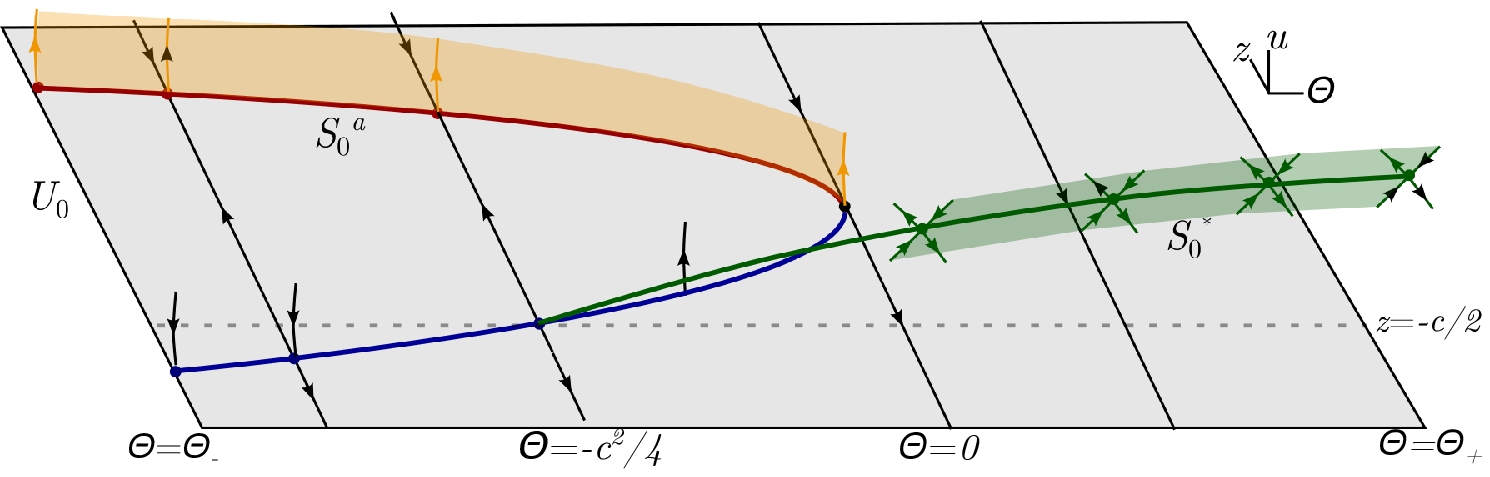}\hspace{-0.2in}\\
\vspace{0.1in}
\includegraphics[trim = 0.05cm 0.05cm 0.0cm 0.05cm,clip,width=0.95\textwidth]{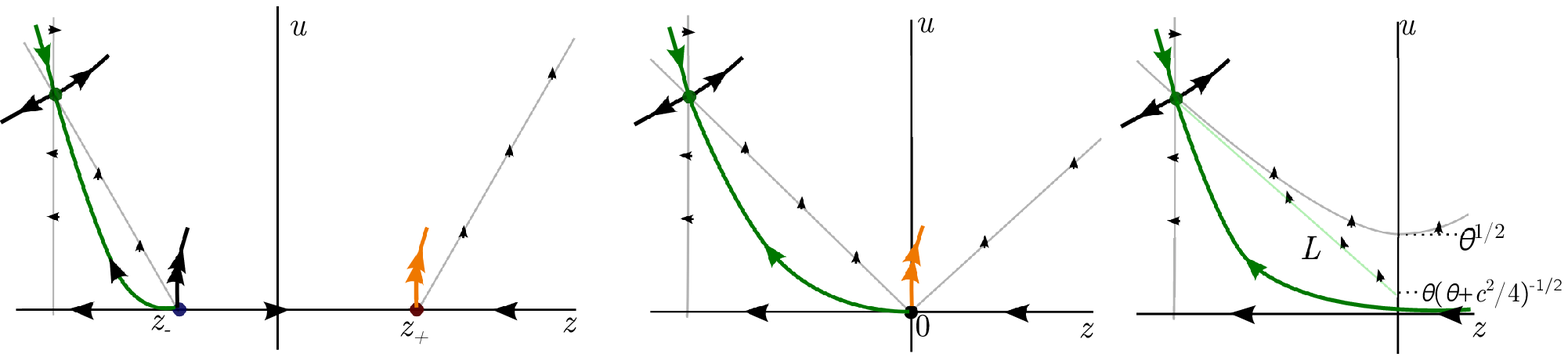}\hspace{-0.2in}\\
\caption{ Top: Phase portrait for \eqref{e:zut1}-\eqref{e:zut3} for $\epsilon = 0$.  Each point $(z_+,0,\theta)$ on $S_0^a$ has
one-dimensional unstable manifold $\tilde{W}^u(z_+,0,\theta)$ (orange fibers), while each point $(z_*,u_*,\theta)$ on $S_0^*$ has a one-dimensional stable manifold $\tilde{W}^s(z_*,u_*,\theta)$ (green fibers). Bottom: $(z,u)$ phase portraits for fixed $\theta$ and $\epsilon = 0$ in the cases $-c^2/4<\theta<0, \theta = 0,$ and $\theta>0$ from left to right. Light grey lines depict nullclines. The green curves denote $\tilde W^\mathrm{s}(z_*,u_*,\theta)$. In the bottom right figure, the trapping line $L$ defined in \eqref{e:ll} is depicted in light green.  }\label{f:eps0}
\end{figure}
%
%{\color{blue} ????? Add a figure showing a numerical representation of $W^\mathrm{s}(z_*,u_*,\theta)$ and possibly some data showing how $u_s(\theta)$ decreases with theta.}

With $\theta$ as a parameter, the $(z,u)$-vector-field has Jacobian
$$
\left(\begin{array}{cc}-2z & 2u \\u & z+c/2\end{array}\right).
$$
At the equilibrium $(z_+,0) = (\sqrt{-\theta},0)$,
the Jacobian has the following eigenvalue and eigenvector pairs:
$$
\nu = -2 z_+, V = (1,0)^T,\qquad \nu = z_+ + c/2, V = (0,1)^T.
$$
Hence, it is a saddle for each $\theta<0$, and the local unstable manifold is given as 
\begin{align}
\tilde W^\mathrm{u}(z_+,0,\theta) &:= \{(z,u)\,:\, z = h^u(u;\theta)\}\\
&h^u(u;\theta) = z_+ + \frac{u^2}{4 z_+ + c} - \frac{3 u^4}{2(3z_++c)(4z_++c)^2} + \mc{O}(|u|^6),
\end{align}
while its stable manifold is simply a subset of the $z$-axis. At the equilibrium $(z_-,0) = (-\sqrt{-\theta},0)$, 
the Jacobian has eigenvalue and eigenvector pairs
$$
\nu = -2 z_-, V = (1,0)^T,\qquad \nu = (z_-+c/2), V = (0,1)^T.
$$
Hence, the equilibrium is a saddle for $\theta<-c^2/4$ 
and a source for $-c^2/4<\theta<0$.  We remark that for $ -c^2/36<\theta<0$ the direction $(1,0)^T$ is the weak unstable direction and $(0,1)^T$ is the strong unstable direction, while these roles are reversed for $\theta<-c^2/36.$ In the former case, we can conclude that in backwards time, $\tilde W^\mathrm{s}(z_*,u_*,\theta)$ approaches the equilibrium $(z_-,0)$ tangentially along the $z$-axis (see Fig. \ref{f:eps0}, bottom left frame). Similar analysis can be done to obtain the expansion for the strong unstable manifold in the $u$-direction but, as it is not needed for this analysis, we omit it.

Finally, at the equilibrium $(z_*,u_*)$,
the Jacobian has the following eigenvalue and eigenvector pairs:
$$
\nu_{*,\pm} = c/2\pm \sqrt{3c^2/4 + 2\theta},\qquad V_\pm = \lp( \frac{\nu_{*,\pm}}{\sqrt{\theta+c^2/4}},1\rp)^T.
$$
Hence, it is a saddle. 
The local stable manifold is given by a graph over the $z$-coordinate as
\begin{align}
\tilde W^\mathrm{s}(z_*,u_*,\theta) &:= \{(z,u)\,:\, u = h^\mathrm{s}(z;\theta)\}\\
&h^s(z;\theta) = u_* + b_1 (z-z_*) + b_2 (z - z_*)^2 + b_3 (z - z_*)^3 +\mc{O}(|z-z_*|^4)\notag\\
&b_1 = \frac{-c-(3c^2+8\theta)^{1/2}}{2\sqrt{c^2 + 4\theta}},\,\, b_2 = \frac{(c+(3c^2+8\theta)^{1/2})(c(3c^2+8\theta)^{1/2}- 2(c^2+6\theta))}{2(c^2+4\theta)^{3/2}(c-3(3c^2+8\theta)^{1/2}},\notag\\
&b_3 = \frac{c(c+(3c^2+8\theta)^{1/2})\lp(-2(c^2+6\theta)+c(3c^2+8\theta)^{1/2}\rp)\lp( 4(2c^2+7\theta) + c(3c^2+8\theta)^{1/2}\rp)}{(c^2+4\theta)^{5/2}(c-3(3c^2+8\theta)^{1/2})^2(c-2(3c^2+8\theta)^{1/2})}.\notag\\
\end{align}
%Here, the coefficients were determined by hand and checked with a mathematica script (more coefficients can be obtained in this way).

Using these facts with a standard nullcline analysis, one obtains the phase portraits in Figure \ref{f:eps0}. From this analysis and a trapping region argument, one can directly see that for each $\theta >0$ small the stable manifold $\tilde W^\mathrm{s}(z_*,u_*,\theta)$ intersects the $u$-axis at a point $u$ with $0<u<\sqrt{\theta}$.  It turns out we can obtain better control of this intersection point. This is the subject of the following lemma:
\begin{Lemma}\label{l:ws0}
For each $\theta>0$ sufficiently small, $\tilde W^\mathrm{s}(z_*,u_*,\theta)$ intersects the set $\{ z = 0\}$ transversely at one point $(0,u_s(\theta))$ with
$$
0<u_s(\theta)\leq \frac{\theta}{\sqrt{\theta+c^2/4}}.
$$
%Furthermore, for $\theta$ in a small neighborhood of the origin there exists a fixed $C>0$ such that $0<u_s(\theta)<C\theta$ for all such $\theta.$
\end{Lemma}
\begin{proof}
We construct a trapping region for $\tilde{W}^\mathrm{s}(z_*,u_*,\theta)$, 
flowed backwards in $\zeta$. Let
\begin{equation}\label{e:ll}
L:= \{ (z,u)\,|\, u = u_* + m (z - z_*), z\in(z_*,0]\},\quad m := -\frac{c}{2\sqrt{\theta+c^2/4}},
\end{equation}
where $m$ is the slope of the $z$-nullcline at $(z_*,u_*)$. We find $L\cap\{z = 0, u\in(0,\sqrt{\theta})\}\neq\varnothing $ for all $\theta>0$ since
$$
u_*-mz_* = \sqrt{\theta+c^2/4} - \frac{c^2}{4\sqrt{\theta+c^2/4}} = \frac{\theta}{\sqrt{\theta+c^2/4}}>0
$$
Next, one can readily calculate that on $L$
$$
\frac{u_\zeta}{z_\zeta} - m = \frac{1}{\sqrt{1+4\theta/c^2}} + \frac{\sqrt{c^2+4\theta}(2c z - 4\theta)}{4\theta(c+2z)}<0,
$$
for all $z\in(-c/2,0)$ and any $\theta>0$ sufficiently small. 
Hence, the flow points ``outwards" along $L$ in forward time. This shows that the slope of the vector field along $L$ is more negative than that of the line $L$ itself, and hence that the flow points outward along $L$. Combining this with the facts that the flow also points outward along the $u$-nullcline at $z = -c/2$ and that the $u = 0$ line is invariant, we obtain that $\tilde W^\mathrm{s}(z_*,u_*,\theta)$ must intersect $I:=\{(0,u)\,|\, 0<u<\sqrt{\theta}\}$ with $u< \theta/\sqrt{\theta+c^2/4}.$ Finally, transversality follows by the properties of the vector field along the line $I$.
\end{proof}

Since $L$ defines a boundary of the trapping region in the above proof, we also have the following corollary:
\begin{Corollary}
Let $\delta>0$ be small, fixed, and independent of $\epsilon$. There exist a $\theta_0>0$ sufficiently small 
and a constant $C_\delta>0$, possibly dependent on $\delta$, 
such that the intersection point $(\delta,u_{s,\delta}(\theta)):= \tilde W^\mathrm{s}(z_*,u_*,\theta)\cap \{z = -\delta\}$ satisfies
\begin{equation}
0< u_{s,\delta}(\theta)\leq C_\delta \theta,
\end{equation}
uniformly for all $\theta\in(0,\theta_0)$.
\end{Corollary}

%{\color{blue} Add calculation for the range of $\epsilon$ the existence result applies to...}

\section{Invariant manifolds, foliations, and slow flow }\label{sec:eps.gt.0}
In this section,
we analyze the dynamics of system \eqref{e:zut1}--\eqref{e:zut3}
for $0 < \eps \ll 1$.
Our goal will be to use geometric singular perturbation theory \cite{fenichel,Jones1995} to view $W^\mathrm{u}(z_+,0,\theta_-)$ as a perturbation of the union of $\epsilon = 0$ manifolds $\cup_{\theta \in (-1-c^2/4,0]}\tilde W^\mathrm{u}(z_+,0,\theta) $ and $W^\mathrm{s}(z_*,u_*,\theta_+)$ as a perturbation of the union of $\epsilon = 0$ manifolds $\cup_{\theta\in(-c^2/4,1-c^2/4]}\tilde W^\mathrm{s}(z_*,u_*,\theta)$.  

\paragraph{Slow passage through a fold} 

Let us begin with $W^\mathrm{u}(z_+,0,\theta_-)$. 
First, for $\epsilon = 0$ the curve of equilibria
$$
S_0^a:=\{(z,u,\theta)\,:\, z = z_+(\theta),  u = 0, \theta\in [-1-c^2/4, 0)\}
$$
is a normally hyperbolic invariant manifold with expanding direction in the $u$ direction and attracting direction in the $z$ direction for all $\theta\leq-b$, for some $b>0$ fixed, small, and independent of $\epsilon$. Note that this family collides with a repelling curve of equilibria $S_0^r:=\{ z = z_-(\theta), u = 0, \theta\in[-1-c^2/4,0)\}$ in a generic fold bifurcation at $\theta = 0$, and hence loses normally hyperbolicity at $\theta = 0$.

Applying Fenichel theory to the dynamics on the invariant set $U_0 = \{u = 0\}$,
that is to the fast-slow subsystem on the invariant $(z,\theta)$-plane, we see that the critical manifold $S_0^a$ perturbs smoothly in $0<\epsilon\ll1$ to a 1-D invariant slow manifold $S_\epsilon^a\subset U_0$ for $\theta<-b<0$. Also note that $S_\epsilon^a$ makes up the weak unstable manifold of the left equilibrium $(z_+,0,\theta_-)$. Since $\theta_\zeta\approx \epsilon$ near $\theta = 0$, Theorem 2.1 of \cite{krupaszmolyan01} allows one to track $S_\epsilon^a$ forward in $\theta\geq-b$ past the fold point at the origin. Further, one can rigorously calculate the bifurcation delay in $\theta>0$. In particular, setting 
$$
\tilde \Sigma_\delta:=\{(z,\theta)\,:\, z = -\delta, \theta\in (0,\delta)\},
$$
one can adapt
Theorem 2.1 \cite{krupaszmolyan01} to obtain the following result for the fast-slow subsystem on $U_0$:
\begin{Proposition}\label{p:delay}
Let $\delta>0$ be fixed small. 
There exists an $\epsilon_0>0$ such that, for all $0<\epsilon\leq \epsilon_0$,
the slow invariant manifold $S_\epsilon^a$ passes through the section 
$\tilde\Sigma_\delta$ at a point 
$(z,\theta) = (-\delta,\theta_a(\epsilon))$ with
\beq
\theta_a(\epsilon) 
= \Omega_0 \left(1 - \frac{c^4}{16}\right)^{2/3} \epsilon^{2/3} 
+ \mathcal{O}(\epsilon \ln(\eps)),\label{e:slpass}
\eeq
where $\Omega_0$ is the smallest positive zero of 
$J_{-1/3}(2z^{3/2}/3) + J_{1/3}(2z^{3/2}/3)$ 
and $J_{\pm1/3}$ are Bessel functions of the first kind. (Note $z$ is a generic complex variable here, distinct from $z$ introduced in \eqref{e:zut1} -\eqref{e:zut3}, and also $\Omega_0  = 2.338107...$)
\end{Proposition}
\begin{proof}
Define the following change of coordinates
$$
z = -(1-c^4/16)^{1/3}\tilde x,
\qquad 
\theta = -\left(1-\frac{c^4}{16}\right)^{2/3} \tilde y,
\qquad 
\zeta= \tau \left(1-\frac{c^4}{16}\right)^{-1/3}. 
$$
On the invariant set $U_0$, the system \eqref{e:zut1}--\eqref{e:zut3} then takes the form
\begin{align}
\frac{d\tilde x}{d\tau} &= \tilde x^2 - \tilde y,\\
\frac{d\tilde y}{d\tau} &= \epsilon\left(-1-\frac{c^2}{2(1-\frac{c^4}{16})^{1/3}} \tilde y + (1-\frac{c^4}{16})^{1/3}\tilde y^2\right).
\nonumber
\end{align}
This system is equivalent to equation (2.5) in \cite{krupaszmolyan01} with their $g$ defined as $g(\tilde x,\tilde y,\epsilon) = \left(-1-\frac{c^2}{2(1-\frac{c^4}{16})^{1/3}} \tilde y + (1-\frac{c^4}{16})^{1/3}\tilde y^2)\right).$ 
Hence, Theorem 2.1 in \cite{krupaszmolyan01}
shows that 
$\tilde{y} = -\Omega_0 \eps^{2/3} + \mathcal{O}(\eps \ln(\eps))$
on $\tilde \Sigma_\delta$.
Translating this back, one obtains $\theta_a(\eps)$,
and the result is established.
\end{proof}

Next, notice that the subset 
$$
U_0^r:= \{(z,u,\theta)\,:\, z>-c/2, u = 0\},
$$
of the invariant plane $U_0$, is a normally hyperbolic (repelling) invariant manifold for all $\epsilon\geq0$ (for completeness we also notice that the corresponding subset $U_0^a\subset U_0$  with $z<-c/2$ is normally attracting). The dynamics in the normal direction to $U_0^r$ are exponentially repelling, while the dynamics in the tangential directions along $U_0^r$ are exponentially attracting in a neighborhood of $S_0^a$. Hence, the dynamics in a tubular neighborhood of $S_0^a$ are smoothly foliated by 1-D unstable fibers which we denote by $\mc{F}_{(z,\theta)}^\mathrm{uu}$. The Fenichel theory \cite{fenichel,Jones1995} guarantees that these fibers can be written as a graph over the normal direction
$$
\mc{F}^\mathrm{uu}_{(z,\theta)}:= \{ (z,u,\theta)\,|\, (z,\theta) = h^\mathrm{uu}(u;z,\theta), |u|\leq \gamma \},$$
for some $\gamma>0$ small and independent of $\eps$.
Here, $h^\mathrm{uu}$ is $C^r$-smooth in $u$, $C^{r-1}$-smooth in the base-point $(z,\theta)$ for any $r\in\N$, 
and satisfies
$$
h^\mathrm{uu}(0;z,\theta) = (z,\theta), \quad  \frac{d}{du}h^\mathrm{uu}(0;z,\theta) = 0.
$$
This foliation satisfies the invariance condition
$$
\Phi_\zeta(\mc{F}^\mathrm{uu}_{(z,\theta)})\subset \mc{F}^\mathrm{uu}_{\phi_\zeta(z,\theta)},
$$
where $\Phi_\zeta$ is the flow of the full 3-D system, and $\phi_\zeta$ is the flow on the invariant set $U_0$. 
For the base points on $S_0^a$ in particular, these fibers are given by the unstable manifolds $\tilde W^\mathrm{u}(z_+,0,\theta)$. This foliation persists smoothly for $0<\epsilon\ll1$, but we suppress the $\epsilon$-dependence to simplify notation. 

%{\color{blue} Double check $\epsilon$ dependence of fibers, and that it doesn't need to be added into the notation!}

For base points on the perturbed slow manifold $S_\epsilon^a$, the union of fibers gives a local representation of the unstable manifold of the point $(z_+,0,\theta_-)$, and $\mc{F}^\mathrm{uu}_{(z_+,\theta_-)}$ gives its local strong unstable manifold,
$$
W^\mathrm{u}(z_+,0,\theta_-)\cap\{|u|\leq\gamma\} = \bigcup_{(z,\theta)\in S_\epsilon^a} \mc{F}^\mathrm{uu}_{(z,\theta)},
$$
for some $\gamma>0$ sufficiently small. See Figure \ref{f:eps1} for a depiction. In addition, such a smooth foliation also holds in a neighborhood of the origin $(z,u,\theta) = (0,0,0)$, since the dynamics in $z$ are weakly exponential for $-1\ll\theta<0$ and algebraic for $\theta\geq0$.

As we are interested in how the manifold $W^\mathrm{u}(z_+,0,\theta_-)$ behaves in a neighborhood of the origin, we extend the section $\tilde \Sigma_\delta$ into the $u$-direction, defining for $\delta,\eta, \gamma>0$ fixed small,
$$
\Sigma_\delta:=\{(z,\theta,u)\,: \, z = -\delta, \theta\in(-\eta,\eta), u\in[0,\gamma)\}.
$$
We can now use the strong-unstable fibers over $S_\epsilon^a$ to describe the intersection of $W^\mathrm{u}(z_+,0,\theta_-)$ with $\Sigma_\delta$.
\begin{Lemma}\label{l:wuep}
Fix $\delta,\eta,\gamma>0$ small. Then there exists an $\epsilon_0$ such that for all $\epsilon\in[0,\epsilon_0)$ the unstable manifold $W^\mathrm{u}(z_+,0,\theta_-)$ intersects $\Sigma_\delta$ transversely and is a graph in $\theta$ of a smooth function $g^u:\R\rightarrow\R$ over the $u$-coordinate:
$$
W^\mathrm{u}(z_+,0,\theta_-)\cap \Sigma_\delta = \{(-\delta,u,g^u(u;\epsilon))\,,\, u\in[0,\gamma)\}.
$$
\end{Lemma}
\begin{proof} %{(sketch)\color{blue}}
This follows by the transverse intersection of $S^a_\eps$ with
$\tilde \Sigma_\delta$, the fact that the fibers $\mc{F}^\mathrm{uu}_{(z,\theta)}$ are vertical at
leading order in $u$, and the smoothness of the fibers $\mc{F}^\mathrm{uu}_{(z,\theta)}$.
%This follows by the transverse intersection of $S_\epsilon^a$ with $\tilde \Sigma_\delta$, the smoothness of the fibers $\mc{F}^\mathrm{uu}_{(z,\theta)}$ and the fact that they are vertical at leading order in $u$.
\end{proof}

Next, we use Fenichel theory to conclude that, for $0<\epsilon\ll1$, the manifold $W^\mathrm{s}(z_*,u_*,\theta_+)$ is a smooth perturbation of the union of stable manifolds $\cup_{\theta>-c^2/4} \tilde W^\mathrm{s}(z_*,u_*,\theta)$ for $\epsilon = 0$. Indeed the saddle curve
$$S_0^*:=\{(z,u,\theta)=(z_*,u_*,\theta)\,:\,\theta \in(-c^2/4, 1- c^2/4)\},$$
 depicted in green in Figure \ref{f:eps0}, persists for $0<\epsilon \ll 1$ as a 1-D  normally hyperbolic invariant slow manifold $S_\epsilon^*$. The asymptotic expansion of $S_\epsilon^*$ is given by
\begin{align}
z &= -\frac{c}{2}
+ \epsilon\frac{1-\left(\theta+\frac{c^2}{4}\right)^2}{2\left(\theta+\frac{c^2}{4}\right)}
+ \mathcal{O}(\epsilon^2),
\qquad
u = \sqrt{\theta + \frac{c^2}{4}}
- \epsilon \frac{c}{4}\frac{1-\left(\theta+\frac{c^2}{4}\right)^2}{\left(\theta+\frac{c^2}{4}\right)^{3/2}}
+ \mathcal{O}(\epsilon^2).
\end{align}

We have
\begin{Lemma}\label{l:wsep}
Fix $\delta,\eta,\gamma>0$ small. There exists an $\epsilon_0>0$ such that,
for all $\epsilon\in(0,\epsilon_0)$, the invariant manifold $W^\mathrm{s}(z_*,u_*,\theta_+)$ intersects the section $\Sigma_\delta$ transversely in a curve which is described as the graph $g^\mathrm{s}:\R\rightarrow\R$ over the $\theta$ coordinate:
\begin{equation}
W^\mathrm{s}(z_*,u_*,\theta_+)\cap\Sigma_\delta = \{(-\delta,g^\mathrm{s}(\theta;\epsilon),\theta)\},\qquad g^\mathrm{s}(\theta;\epsilon) = \mathcal{O}(\epsilon+\theta).
\end{equation}
\end{Lemma}
\begin{proof} %(sketch)
For $\epsilon = 0$, existence, transversality, as well the bound $|g^\mathrm{s}(\theta)|\leq C\theta$ for some $C$ independent of $\eps$, follow by Lemma \ref{l:ws0} and smooth dependence of $\tilde W^\mathrm{s}(z_*,u_*,\theta)$ on $\theta$. Then, for $0<\eps\ll1$, Fenichel theory implies that the curve $S_0^*$ of saddle equilibria $(z_*,u_*,\theta)$ for $\epsilon = 0$ perturbs to a slow, normally hyperbolic invariant manifold for $0<\epsilon\ll1$ which forms the weak stable manifold of $(z_*,u_*,\theta_+)$. Also by the Fenichel theory, the manifolds $\tilde W^\mathrm{s}(z_*,u_*,\theta)$ perturb to the strong-stable fibers of $W^\mathrm{s}(z_*,u_*,\theta_+)$. The result then follows by smooth dependence on $\epsilon$.
\end{proof}

From these two results, since the curve $W^\mathrm{u}(z_+,0,\theta_-)\cap \Sigma_\delta$ is a graph over $u$ and is vertical at leading order and the curve $W^\mathrm{s}(z_*,u_*,\theta_+)\cap\Sigma_\delta$ is a graph over $\theta$, one generically expects the desired intersection to exist for sufficiently small $\epsilon$; see Figure \ref{f:ori}. We demonstrate this in the next section. Furthermore, the bifurcation delay prediction for $\theta$ can then be translated to a $\mu$-prediction for the delay 
$$
\mu_\mathrm{fr} \approx c^2/4+\theta_a(\eps),%+C_0\epsilon^{2/3}.
$$
recall \eqref{e:slpass}, which will then establish \eqref{e:mufr}.

%To summarize the rest of our argument, we use the precise control over the stable manifolds $\tilde W^\mathrm{s}(z_*,u_*,\theta)$ for $\epsilon = 0$, to perturb in $\epsilon$ and track $W^\mathrm{s}(z_*,u_*,\theta_-)$ back to a neighborhood of the origin $(z,u,\theta) = (0,0,0)$. There we find it intersects a section $\Sigma_\delta = \{z = -\delta\}$ in a line which can be written as a graph over the $\theta$ variable. 

\begin{figure}[h!]
\centering
\includegraphics[trim = 0.0cm 0.0cm 0.0cm 0.0cm,clip,width=0.7\textwidth]{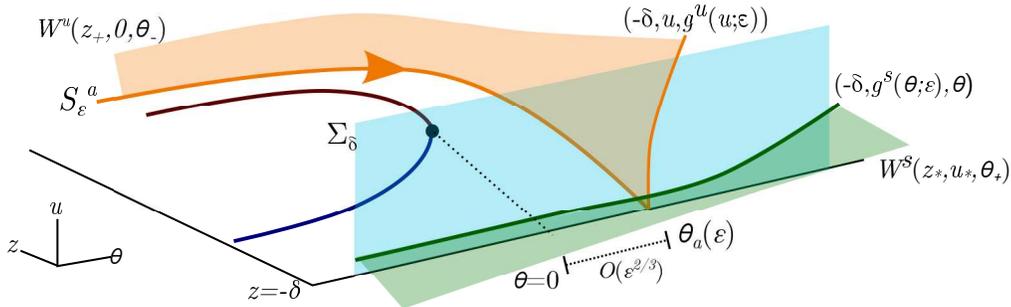}\hspace{-0.2in}
\caption{Dynamics near the origin, depicting the intersection of $W^\mathrm{u}(z_+,0,\theta_-)$  (orange) and $W^\mathrm{s}(z_*,u_*,\theta_+)$ (green) with the section $\Sigma_\delta$ (light blue) and how the transverse intersection is obtained. Intersections with $\Sigma_\delta$ are depicted as solid lines (orange and green respectively) without arrows. Red and blue curves once again depict the $\epsilon = 0$ fold curve $S_0^a\cup S_0^r$.  }\label{f:ori}
\end{figure}

\begin{Remark}
 For each fixed value of $c \in (0,2)$, Proposition \ref{p:delay} is an
asymptotic result valid for sufficiently small values of $\epsilon$. Here, we
observe that the opposite limit in which $\epsilon>0$ is fixed and $c \to 2^-$
is a different singular limit. First, with $\epsilon>0$ fixed, there is no
asymptotic time scale separation in the system (1.13)-(1.15) on $\{ u = 0
\}$ for the variables $z$ and $\theta$. More importantly, with $c=2$, the
system has a fixed point at $(z,\theta)=(0,0)$, and solutions with initial
data in that region of the fourth quadrant between the parabola $\theta =
-z^2$ and the positive $z$-axis approach that fixed point, {\it i.e.,
$\theta(\zeta) \le 0$ for all $\zeta$}. In contrast, for any value of
$c<2$, no matter how close to 2, the origin is no longer a fixed point,
and solutions with initial data in the same region approach the invariant
line $\{ \theta = 1 - \frac{c^2}{4} \}$ with $z \to -\infty$. Hence, the
limit $c \to 2^-$ is a different singular limit. In a manner similar to \cite{goh2014triggered}, we expect the absolute spectrum to once again play a role in determining the value of $\theta$ on exit from a neighborhood of the origin, and hence the location of the front interface for fixed $\epsilon>0$. We do not address this here, since
our interest in the quenching problem is for small $\epsilon$.
%
% pinning of the front interface for fixed $\epsilon>0$, the value of
%$\theta$ on exit from a neighborhood of the origin (namely,
%$\theta_a(c-2)$) scales as $c \to 2^-$. We do not address that here, since
%our interest in the quenching problem is for small $\epsilon$.
\end{Remark}

%\begin{Remark}
%We also mention the work \cite{carter2018unpeeling} which finds Airy points along the repelling slow manifold of the Fitzhugh-Nagumo system, at which the local linear stability type of the point in the fast subsystem changes from being an unstable node to an unstable spiral.
%\end{Remark}

\section{Dynamics near origin: completing the proof of Theorem \ref{t:0}}\label{s:ori}

To construct the desired intersection, we use the foliation graph $h^\mathrm{uu}$ of $U_0$ to straighten the fibers, and decouple the $(z,\theta)$-dynamics from the $u$ dynamics. In these new coordinates, the unstable manifold $W^\mathrm{u}(z_+,0,\theta_-)$ is vertical, while $W^\mathrm{s}(z_*,u_*,\theta_+)$ still intersects $\Sigma_\delta$ in a graph over $\theta$. %This allows us to readily construct an intersection.
To begin, we use the hyperbolic dynamics normal to $U_0^r = \{ u = 0,\, z>-c/2\}$ to straighten the fibers in a neighborhood of $U_0^r$ so that the $z$ and $\theta$ equations become independent of $u$. In particular, the function $h^\mathrm{uu}$, which defines the strong-unstable foliation of $U_0^r$, 
defines a smooth coordinate change 
$$
(z_1,\theta_1) = h^\mathrm{uu}(u;z,\theta),\,\, u_1 = u.
$$
 Here, we have that $h^\mathrm{uu}(u;z,\theta) = I_2 + \tilde h^\mathrm{uu}(u;z,\theta)$ with $\tilde h^\mathrm{uu}(u;z,\theta) = \mathcal{O}(u^2)$ uniformly in $z,\theta,$ and $\epsilon$, and hence is locally invertible for $|u|<\gamma$.

By substituting $(z,\theta) = (h^\mathrm{uu})^{-1}(z_1,\theta_1)$ and $u = u_1$ 
into \eqref{e:zut1}--\eqref{e:zut3} and using the invariance property
, 
we obtain the following system:
\begin{align}
z_{1,\zeta} &= - z_1^2 -\theta_1,\label{e:z1}\\
u_{1,\zeta} &=  u_1 f_1(z_1,\theta_1,u_1;\epsilon),\qquad\qquad \label{e:u1}\\
\theta_{1,\zeta} &= \epsilon f_2(z_1,\theta_1,u_1;\epsilon)\label{e:t1}%\epsilon (1 - (\theta_1 + c^2/4)^2).
\end{align}
for smooth functions $f_1, f_2$ with 
$$
f_1(z_1,\theta_1,0;\epsilon) = z_1+c/2,\qquad f_2(z_1,\theta_1,0;\epsilon) = 1-(\theta_1+c^2/4)^2,
$$
and $f_i(z_1,\theta_1,u;\eps) - f_i(z_1,\theta_1,0;\eps) 
= O(u^2)$ as $u\rightarrow0$.
Note the dynamics on $U_0^r$ are left unchanged. In a neighborhood of $U_0^r$, the manifolds $W^\mathrm{u}(z_{1,+},0,\theta_{1,-})$ and $W^\mathrm{s}(z_{1,*},u_{1,*},\theta_{1,+})$ can be described 
by the dynamics of the base points in $U_0^r$ of the fibers which they intersect.

Next, in view of Lemma \ref{l:wuep}, the unstable manifold is now vertical, 
$$
W^\mathrm{u}(z_+,0,\theta_-) = \bigcup_{(z_1,\theta_1)\in S_\epsilon^a} \{(z_1,u_1,\theta_1)\,:\,  \,\, |u_1|\leq \gamma\}.
$$
Thus, in \eqref{e:z1}-\eqref{e:t1},
\beq\label{e:wu_g}
W^\mathrm{u}(z_+,0,\theta_-)\cap\Sigma_\delta = \{ (-\delta,u,\theta_a(\epsilon))\,:\, |u|\leq \gamma\},
\eeq
where we recall that $\theta_a(\epsilon)$ is the intersection of $S_\epsilon^a$ with $z_1 = -\delta$ defined in \eqref{e:slpass} (and which is unchanged in these new coordinates since $u = 0$).

Furthermore, in view of Lemma \ref{l:wsep}, we can also conclude that in the new coordinates
\beq\label{e:ws_g}
W^\mathrm{s}(z_*,u_*,\theta_+)\cap\Sigma_\delta  = \{(-\delta,\tl g^\mathrm{s}(\theta;\epsilon),\theta)\},
\eeq
with $|\tl g^\mathrm{s}(\theta;\epsilon)|\leq C(\theta+\epsilon)$, since the fibers $\mc{F}^\mathrm{uu}$ vary quadratically in $u$.  
Hence, we seek intersections of the curves described in \eqref{e:wu_g} and \eqref{e:ws_g}. Equating the two curves, we obtain the matching equations
\begin{align}
u &= \tl g^\mathrm{s}(\theta;\epsilon),\\
\theta_a(\epsilon) &= \theta,
\end{align}
where $u$ and $\theta$ are free in $(-\gamma,\gamma)$ and $[0,\eta)$, respectively. Hence, for any $\epsilon\in(0,\epsilon_0)$, we choose $\theta = \theta_a(\epsilon)$ and  $u = \tl g^\mathrm{s}(\theta_a(\epsilon);\epsilon)$, to conclude the desired intersection. We note that at the intersection location $\theta_a(\epsilon)$ the $u$ coordinate is $\mc{O}(\epsilon^{2/3})$. A standard finite-time argument shows that the additional delay in $\theta$ needed for $u(\zeta) = \sqrt{\mu_c}/2 = c/4$ is then $\mathrm{o}(\epsilon^{2/3})$ and thus higher-order.  This completes the proof of the theorem.

\section{Stationary fronts: geometric desingularization analysis}\label{s:c0}
In this section,
we begin the proof of Theorem \ref{t:2}. That is we study fronts created by a stationary quench,
which solve 
\eqref{e:tw0a}-\eqref{e:tw0b}
with $c=0$,
\begin{align}
\label{e:c0a}
u_{\xi} &= v, \\
\label{e:c0b}
v_{\xi} &= -\mu u + u^3, \\
\label{e:c0c}
\mu_{\xi} &= -\eps(1-\mu^2), \qquad \mu(0)=0.
\end{align}
Here, $\xi=x-ct$ reduces to 
$\xi=x$. We first note system \eqref{e:c0a}-\eqref{e:c0c}
is invariant under the reflection 
$(u,v,\mu)$ to $(-u,-v,\mu)$. We then note that for $\eps=0$, 
the system \eqref{e:c0a}-\eqref{e:c0c} has normally hyperbolic manifolds
which are curves of saddle equilibria 
\begin{align}
S_0^{\pm} &= \{ (u,v,\mu) = (\pm\sqrt{\mu},0,\mu), \mu>\tilde{\eta} \} \\
S_0^0 &= \{ (u,v,\mu) = (0,0,\mu), \mu<-\tilde{\eta} \},
\end{align}
where $\tilde{\eta}>0$ is small
and independent of $\eps$.
We examine these critical manifolds 
for $\eps=0$, as well as the perturbed slow manifolds
which exist for $0<\eps\ll 1$ by Fenichel theory,
in the four-dimensional extended system
\begin{align}
u_\xi &= v\label{e:uvmea}\\
v_\xi&= -\mu u + u^3\label{e:uvmeb}\\
\mu_\xi&= -\epsilon (1 - \mu^2)\label{e:uvmec}\\
\epsilon_\xi&=0. \label{e:uvmed}
\end{align}
We denote the family of such perturbed slow manifolds as $S^\pm_\epsilon$ and $S^0_\epsilon$ and the union of them for $\epsilon\geq0$ small as $M^{\pm}$ and $M^0$. These correspond to center-like manifolds in the extended system.
As mentioned above, for each $\epsilon$-slice, $S^\pm_\epsilon$ forms part of the unstable manifold of the equilibria $(u,v,\mu) = (\pm 1,0,1)$ while $S^0_\epsilon$ forms part of the unstable manifold $W^\ru(0,0,1)$ for $\mu>\tl\eta$ and part of the stable manifold $W^\rs(0,0,-1)$ for $\mu<-\tl\eta$. These manifolds give the base points of fibers which foliate the manifolds they live in. For example $S_\epsilon^+$ serves as the base points of strong stable/unstable fibers which foliate $W^{\rs/\ru}(1,0,1)$. Hence, we wish to use the slow manifolds to track the containing invariant manifolds and construct the desired heteroclinic intersection.

We wish to track the perturbed slow manifolds
through a neighborhood 
of $(u,v,\mu) = (0,0,0)$,
where they lose normal hyperbolicity,
using the quasi-homogeneous geometric blow up
\begin{equation}\label{e:barcoor}
u = r \overline{u}, \quad v = r^2 \overline{v}, \quad \mu = r^2 \overline{\mu},\quad \epsilon = r^3 \overline{\epsilon}.
\end{equation}
These coordinates blow up 
the origin $(0,0,0,0)$ 
into a 3-sphere 
$ S_3 = \{r = 0, \overline{u}^2+\overline{v}^2+\overline{\mu}^2+\overline{\eps}^2 = 1\}$,
which is invariant under the induced flow.
In particular, 
it is natural 
to study the dynamics on and near the sphere
using the following three charts defined by ${\bar \mu}=1, {\bar \eps}=1, $ and ${\bar \mu}=-1$,
respectively:
\begin{align}
{\rm Entry} \, {\rm chart} \, &K_1: \quad
u=r_1 u_1,
v=r_1^2v_1,
\mu=r_1^2,
\eps=r_1^3 \eps_1 \\
{\rm Rescaling} \, {\rm chart} \, &K_2: \quad
u=r_2 u_2,
v=r_2^2v_2,
\mu=r_2^2\mu_2,
\eps=r_2^3 \\
{\rm Exit} \, {\rm chart} \, &K_3: \quad
u=r_3 u_3,
v=r_3^2v_3,
\mu=-r_3^2,
\eps=r_3^3 \eps_3.
\end{align}
Here, $x_i$ denotes the variable
${\bar x} \in \{ {\bar u}, {\bar v}, {\bar \mu}, {\bar \eps} \}$
in chart $K_i$. The change of coordinate map $\kappa_{12}$ between the charts $K_1$ and $K_2$, as well the map $\kappa_{23}$ between $K_2$ to $K_3$ are given as 
\begin{align}
&\kappa_{12}\,\,:\,\, u_2 = \epsilon_1^{-1/3}u_1,\quad v_2 = \epsilon_1^{-2/3}v_1,\quad \mu_2 = \epsilon_1^{-2/3},\quad r_2 = \epsilon_1^{1/3}r_1, \qquad\qquad \epsilon_1>0\\
&\kappa_{23}\,\,:\,\, u_3 = (-\mu_2)^{-1/2}u_2,\quad v_3 = (-\mu_2)^{-1} v_2,\quad \epsilon_3 = (-\mu_2)^{-3/2},\quad r_3 = r_2 (-\mu_2)^{1/2},\qquad\qquad \mu_2 < 0.
\end{align}
We remark that the second mapping above, $\kappa_{23}$, maps into the exit chart where $\bar\mu<0$.  We next collect information about the phase portrait near the sphere $\{r = 0\}$ in each coordinate chart, first describing the entry and exit charts $K_1,K_3$ and then the re-scaling chart $K_2$.  

\paragraph{Entry chart $K_1$ phase portrait} 
In chart $K_1$,
the governing equations are
\begin{align}
u_1' &= v_1 + \frac{1}{2} \eps_1 u_1 (1-r_1^4) \label{e:k1a}\\
v_1' &= -u_1 + u_1^3 + \eps_1 v_1 (1-r_1^4) \label{e:k1b}\\
\eps_1' &= \frac{3}{2}\eps_1^2 (1-r_1^4) \label{e:k1c}\\
r_1' &=-\frac{1}{2}r_1\eps_1 (1-r_1^4). \label{e:k1d}
\end{align}
Here, we recall that $K_1$ is defined by $\mu_1=1$,
and we have introduced the new time variable
$\xi_1=r_1\xi$
to desingularize the vector field,
with the prime now denoting the derivative with respect to $\xi_1$. We note that the system is autonomous so that the reparametrization of solutions
leaves the trajectories in phase space intact.
The system \eqref{e:k1a}-\eqref{e:k1d} has fixed points
at $p^-=(-1,0,0,r_1)$,
$p^0=(0,0,0,r_1)$,
and $p^+=(1,0,0,r_1)$
for each $r_1 \ge 0$.
These are exactly the points
at which the invariant manifolds
$S_0^-$,
$S_0^0$,
and $S_0^+$,
respectively,
enter the neighborhood of the blown-up singularity. The equilibrium $p_+$, and indeed each equilibrium in $S^+_0$, has one-dimensional stable and unstable eigenspaces contained in the $(u_1,v_1)$ plane and two center directions, one in the $r_1$ direction, tangential along $S_0^+$, and the other given by the generalized eigenvector $(0,1,-2/(1-r_1^4),0)$ for $r_1>0$ and the eigenvector $(0,1,-2,0)$ for $r_1 = 0$ (note the former center direction corresponds to the family of equilibria formed by $S^+_0$).  Thus, $S_0^+$ lies inside of a two-dimensional center manifold $\mathcal{M}^{c,+}$ which, in the original extended system \eqref{e:uvmea} - \eqref{e:uvmed}, corresponds to the family $M^+$ of slow manifolds for $\epsilon$ small. Due to the strong stable and unstable directions in the $(u_1,v_1)$ directions, $\mc{M}^{c,+}$ is normally hyperbolic with strong stable and unstable foliations. The union of the strong unstable fibers forms a center-unstable manifold $\mathcal{M}^{\rcu,+}$ which corresponds to $W^\rcu(1,0,1)$ in the original coordinates. In $K_1$, $\mathcal{M}^{\rcu,+}$  contains the set of all bounded solutions as $\xi_1\rightarrow-\infty$.

Furthermore,
the hyperplane $\{ r_1=0 \}$ is an invariant set,
and on it the dynamics 
reduce to
\begin{align}
u_1' &= v_1 + \frac{1}{2} \eps_1 u_1 \nonumber\\
v_1'&= -u_1 + u_1^3 + \eps_1 v_1 \nonumber \\
\eps_1'&= \frac{3}{2}\eps_1^2 \nonumber \\
r_1'&= 0. \nonumber
\end{align}
Hence, standard center manifold theory directly implies that, when restricted to $\{r_1 = 0\}$,
$p^\pm$ have one-dimensional normally hyperbolic center manifolds, $N_1^\pm$. Moreover, these are not unique due to the presence of both hyperbolic repelling and attracting dynamics in the $(u_1,v_1)$ plane. Note also that $p^0$ has a one-dimensional normally elliptic center manifold given by $(u_1,v_1,\eps_1)=(0,0,\eps_1)$.

We focus on $p^+$ and $N_1^+$
for the heteroclinic here.
By standard center manifold theory,
$N_1^+$ is tangent at $p^+$
to the center eigendirection spanned by
$(0,-1,2)$.
Asymptotically, it is represented by
\begin{align}
u_1 &= 1 - \frac{\eps_1^2}{8} - \frac{73}{128} \eps_1^4 + \mathcal{O}(\eps_1^6) \nonumber \\
v_1 &= - \frac{\eps_1}{2} - \frac{5}{16} \eps_1^3 
- \frac{803}{256}\eps_1^5 + \mathcal{O}(\eps_1^7).\label{e:n1p}
\end{align}
This follows from applying the invariance condition,
and we recall that 
all center manifolds in the family
have the same expansion in powers of small $\eps_1$.
See Figure \ref{f:c0bu}.

\paragraph{Exit Chart $K_3$}
The phase portrait in $K_3$ near $r = 0$ can be derived in a similar way. The governing equations 
are
\begin{align}
u_3' &= v_3 - \frac{1}{2} \eps_3 u_3 (1-r_3^4) \label{e:k3a}\\
v_3' &= u_3 + u_3^3 - \eps_3 v_3 (1-r_3^4) \label{e:k3b}\\
\eps_3' &= -\frac{3}{2}\eps_3^2 (1-r_3^4) \label{e:k3c}\\
r_3' &=\frac{1}{2}r_3\eps_3 (1-r_3^4). \label{e:k3d}
\end{align}
This system has a curve of equilibria $S_0^0 = \{(0,0,0,r_3),\, r_3\geq0\}$, each of which have strong stable/unstable directions in the $(u_3,v_3)$ plane. $S_0^0$ also lies inside a two-dimensional center manifold $\mathcal{M}^{\mathrm{c},0}$ tangent to the $(\epsilon_3,r_3)$ plane. Here one such center manifold is given by the plane $\{(0,0,\epsilon_3,r_3)\,:\, \epsilon_3,r_3\geq0\}$.  This manifold is once again normally hyperbolic with one-dimensional strong stable and unstable fibers. The union of stable fibers gives a local description of a center-stable manifold $\mathcal{M}^{\mathrm{cs},0}$ which corresponds locally to $W^\mathrm{cs}(0,0,-1)$. Similarly to $K_1$, the $r_3  = 0$ plane is invariant with the reduced system
\begin{align}
u_3' &= v_3 - \frac{1}{2} \eps_3 u_3 \nonumber\\
v_3'&= u_3 + u_3^3 - \eps_3 v_3 \nonumber \\
\eps_3'&= -\frac{3}{2}\eps_3^2 \nonumber \\
r_3'&= 0. \nonumber
\end{align}
Here we find the trivial center manifold $N_3^0$ given by $(0,0,\epsilon_3)$ for $\epsilon_3\geq0$.  

Hence, by tracking manifolds across the rescaling chart, we wish show that the three-dimensional manifolds $\mathcal{M}^{\mathrm{cs},0}$ and $\mathcal{M}^{\rcu,+}$ have a two-dimensional intersection, with one direction corresponding to variation in $\epsilon$ and the other the direction of the flow.
\paragraph{Rescaling Chart $K_2$}
Finally, we work in the rescaling chart $K_2$ 
to identify the geometrically unique solution
that represents the desired heteroclinic 
in the blown-up vector field.
In $K_2$,
system \eqref{e:uvmea}-\eqref{e:uvmed} becomes
\begin{align}
u_2' &= v_2\label{e:k2a}\\
v_2'&= -\mu_2 u_2 + u_2^3\label{e:k2b}\\
\mu_2'&= -1 + r_2^4\mu_2^2\label{e:k2c}\\
r_2'&=0. \label{e:k2d}
\end{align}
Here, the prime denotes 
the derivative with respect to 
the new time variable
$\xi_2 = r_2 \xi$. 
%(by which the vector field is desingularized),
%and we note that the system is autonomous
%so that the reparametrization of solutions
%leaves the trajectories in phase space intact.
We focus on the dynamics of this system
on the invariant set $\{ r_2=0\}$,
where the system reduces to
\begin{align}
u_2' &= v_2\label{e:k2ar0}\\
v_2'&= -\mu_2 u_2 + u_2^3\label{e:k2br0}\\
\mu_2'&= -1 \label{e:k2cr0}\\
r_2'&=0. \label{e:k2dr0}
\end{align}
Then, by converting
the $(u_2,v_2)$ subsystem into a second-order scalar equation,
scaling $u_2 = \sqrt{2}\tilde{u}_2$,
and recalling that $\mu_\xi = -\epsilon(1-\mu^2)$ so that $\mu_2=-\xi_2$
 on $\{ r_2 = 0 \}$,
we find that the governing equation
on $\{ r_2=0\}$ is 
\begin{equation}\label{e:pii}
\tilde{u}_2'' = \xi_2 \tilde{u}_2 + 2\tilde{u}_2^3.
\end{equation}
This is precisely the second Painlev\'e equation
(${\rm P}_{\rm II}$), 
recall \eqref{e:PII-intro}.
Note that the scaling used here 
to derive $\tilde{u}_2(\xi_2)$
is the same as that used in Section \ref{ss:hm}
for $\tilde{u}(\eta)$ and $w(\eta)$
since $r_2=\eps^{1/3}$ in $K_2$.

Now, as previewed above
while deriving the formal asymptotics,
the key solution of \eqref{e:pii}
that is of interest here 
is the Hastings and McLeod solution,
$w_{\rm HM}$ of \eqref{e:PII-intro},
which we denote here
by $\tilde{u}_2^*(\xi_2)$.
It is the unique solution 
which satisfies 
the asymptotic boundary conditions
$$
\tilde{u}_2^*(\xi_2)\sim\sqrt{-\xi_2/2}, \,\,\xi_2\rightarrow-\infty,\qquad
\tilde{u}_2^*(\xi_2)\sim \mathrm{Ai}(\xi_2),\,\, \xi_2\rightarrow +\infty,
$$
and which decays strictly monotonically.
Finally, scaling back to $u_2$, 
this yields the unique 
monotonically decaying solution 
$u_2^*(\xi_2)$
of \eqref{e:k2ar0}-\eqref{e:k2dr0}
with the asymptotics 
\begin{align}
u_2^*(\xi_2) &\sim \sqrt{-\xi_2} =\sqrt{\mu_2}, \,\,\xi_2\rightarrow-\infty, 
\label{e:u2stara}\\
u_2^*(\xi_2) &\sim \sqrt{2}\mathrm{Ai}(\xi_2),\,\, \xi_2\rightarrow +\infty.
\label{e:u2starb}
\end{align}

\section{Singular heteroclinic connection on the sphere, transversality}\label{s:singhet}
On the blow-up sphere, $\{r = 0\}$, the Hastings-Mcleod solution $u_2^*$ represents a heteroclinic solution connecting the equilibria $p^+$ on the $\bar\mu>0$ hemisphere to the equilibria $p^0$ in the $\bar \mu<0$ hemisphere. Below, we find that in the charts $K_1$ and $K_3$ this unique connecting orbit gives a 1-D center manifold in the $r_j = 0$ invariant subspaces in chart $K_j$ for both $j = 1,3$.  We thus use this heteroclinic orbit to transport the center unstable manifold $\mc{M}^{cu,+}$ from $K_1$ across the sphere to locate an intersection with the center-stable manifold $\mc{M}^{cs,0}$ in $K_3$. To address the non-uniqueness of the center manifolds in $K_1$ and $K_3$, we first construct an intersection between the local 3-D center unstable manifold of the equilibrium $p^+$ which contains the 1-D center manifold $\kappa_{21}^{-1}u_2^*$ in $K_1$ and the local 3-D center stable manifold of the equilibrium $p^0$ which contains the 1-D center manifold $\kappa_{23}u_2^*$. We do this in order to flow these invariant manifolds globally across $K_2$ using the variational dynamics around $u_2^*$.  We then use inclination properties of the flow in each chart to conclude the same transversality and intersection properties for the center unstable/stable manifolds $\mc{M}^{cu,+}$, $\mc{M}^{cs,0}$.

  Using the inverse coordinate change $\kappa_{12}^{-1}:K_2\rightarrow K_1$, given by $u_1=u_2 \mu_2^{-1/2}, v_1=v_2\mu_2^{-1}, \eps_1=\mu_2^{-3/2},$ and $r_1=r_2 \mu_2^{1/2}$, we can translate the asymptotics of $u_2^*$ into the variables of chart $K_1$. We find that, when flowed back through the entry chart coordinates $K_1$, the solution $\kappa_{12}^{-1}u_2^*(\xi_2)$ asymptotically approaches $p^+ = (1,0,0,0)$ as $\xi_1\to -\infty$, and it lies on a center manifold, $N_1^+$, of this equilibrium.  In fact, the higher order terms in the asymptotic expansion of the Hastings-McLeod solution as $\xi_2 \to -\infty$, given by
\begin{align}
u_2^*(\xi_2) = \sqrt{-\xi_2}\left(1+ \frac{1}{8\xi_2^3} - \frac{73}{128\xi_2^6} + \frac{10219}{1024 \xi_2^9} +\mathcal{O}(\xi_2^{-12})\right),
\end{align}
(see for example \cite{deift95,baik08}, and also \cite{Cleri_2020} for the full trans-series
asymptotics) also agree with the higher-order terms in the expansion of $N_1^+$; recall \eqref{e:n1p}.

In a similar manner, using the coordinate change $\kappa_{23}$, we can translate the asymptotics of $u_2^*$ as $\xi_2\rightarrow+\infty$, given in \eqref{e:u2starb}, into the $K_3$ variables. We find the set $\kappa_{23}u_2^*$ is a 1-D center manifold of the equilibrium $p^0$ in the $\{ r_3 = 0\}$ invariant subspace. Indeed, using the coordinate transform $\xi_2 = -\mu_2 = (\epsilon_3^{2/3})$ the leading order expansion for the Airy functions
\begin{align}
\mathrm{Ai}(\xi) &= \frac{\re^{-\frac{2}{3}\xi^{3/2}}}{2\sqrt{\pi}\xi^{1/4}}\left( 1 + \mathcal{O}(\xi^{-3/2}) \right),\\
\mathrm{Ai}'(\xi) &= -\frac{\xi^{1/4}\re^{-\frac{2}{3}\xi^{3/2}}}{2\sqrt{\pi}}\left( 1 + \mathcal{O}(\xi^{-3/2})\right),
\end{align}
we have the following asymptotic description of the trajectory in $K_3$
\begin{align}
u_3 &= \frac{\exp(-\frac{2}{3} \epsilon_3^{-1}) }{\sqrt{2\pi}}\left( \epsilon_3^{1/2}  + \mathcal{O}(\epsilon_3^{3/2}) \right),\label{e:u3wc}\\
v_3 &=  -\frac{\sqrt{2}\exp(-\frac{2}{3} \epsilon_3^{-1}) }{\sqrt{2\pi}}\left(\epsilon_3^{1/2}  + \mathcal{O}(\epsilon_3^{3/2}) \right)\label{e:v3wc}.
\end{align}
Thus, this trajectory approaches $p^0$ tangentially along the center direction formed by the $\epsilon_3$-axis.

As described above, the equilibria $p^+$ and $p^0$ each have 1-D strong stable and strong unstable subspaces, along with 2-D center spaces. We let $W^{\rcu}_1(p^+)$ denote the 3-D local center-unstable manifold of $p^+$ in $K_1$ which contains $\kappa_{12}^{-1}u_2^*$ and let $W^{\rcs}_3(p^0)$ be the 3-D local center-stable manifold of $p^0$ in $K_3$ which contains $\kappa_{23}u_2^*$. Furthermore, we let $W^{\rcu}_2(p^+)$ and  $W^{\rcs}_2(p^0)$ denote the above manifolds in the $K_2$ coordinates.  These manifolds can be continued along a neighborhood of the connecting solution $u_2^*$ using the flow of the $K_2$ dynamics.

\begin{figure}[h!]
\centering
\hspace{-0.2in}
\includegraphics[trim = 0.0cm 0.0cm 0.05cm 0.0cm,clip,width=0.7\textwidth]{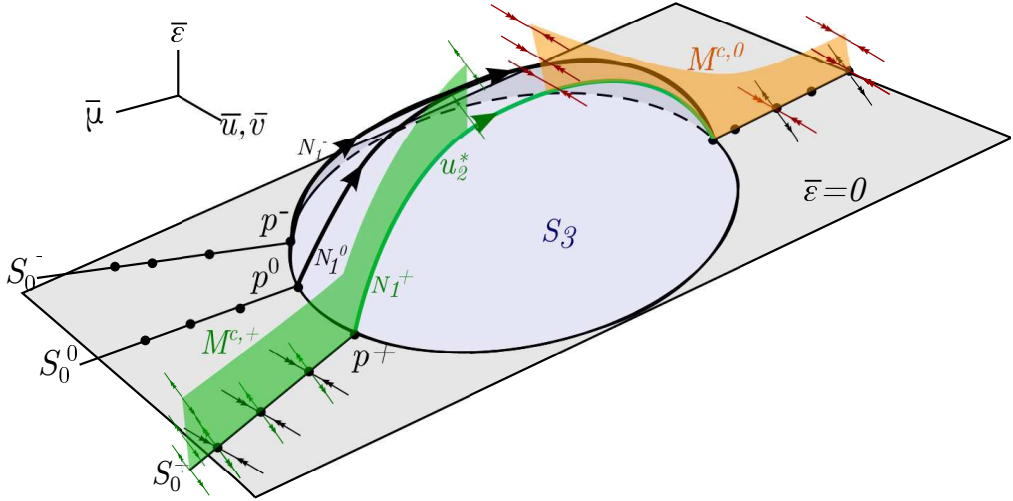}\hspace{-0.2in}
\caption{Schematic depiction of 4-D blown-up phase space near the blow-up sphere $S_3$ (light blue) in the coordinates \eqref{e:barcoor}. The singular heteroclinic $u_2^*$ on $S_3$ connecting $p_+$ in chart $K_1$ to $p^0$ in chart $K_3$  is depicted in green.  Near the equilibria $p_+$ this curve also gives the center manifold $N_1^+$ described in \eqref{e:n1p}.  Critical equilibria curves $S_0^\pm,S_0^0$ for $r\geq0$ lying inside the $\bar\epsilon = 0$ plane (grey) are given by black lines with dots. The green and orange surfaces respectively denote the 2-D center manifolds $\mathcal{M}^\mathrm{c,+}$ and $\mathcal{M}^\mathrm{c,0}$ described in Section \ref{s:c0}, and the double-arrowed green and red curves denote 1-D strong unstable and stable fibers. The desired intersection is given for $0<\epsilon \ll1$ by the intersection of the union of $\mathcal{M}^\mathrm{c,+}$ and its strong unstable fibers with $\mathcal{M}^\mathrm{c,0}$ and its strong stable fibers.   }\label{f:c0bu}
\end{figure}

We wish to show that the invariant manifolds $W^{\rcu}_1(p^+)$ and $W^{\rcs}_3(p^0)$, globally continued across the sphere intersect transversely with two dimensional intersection containing $u_2^*$. To do this, we track them both in the rescaling chart $K_2$ in a neighborhood of $u_2^*$ using the associated variational equation. In particular, letting $U_2^* = (u_2^*,v_2^*,\mu_2^*,0)^T$, and $F(U)$ denote the 4-D vector-field defined in \eqref{e:k2a} - \eqref{e:k2d}, we insert the solution decomposition $U = U_2^* + W,\quad W = (w_1,w_2,w_3,w_4)^T\in \R^4$, into the nonlinear system, obtaining
\begin{align}\label{e:nlvar}
W' &= A_2(\xi_2) W + G(\xi_2,W),\\
&A_2(\xi_2) = DF(U_2^*(\xi_2)),\qquad G(\xi_2,W) = F(U_2^*(\xi_2) + W) - F(U_2^*(\xi_2)) - DF(U_2^*(\xi_2)) W.\notag
\end{align}
Here $A_2$ takes the 2x2 block form
$$
A_2(\xi_2) = \left(\begin{array}{cc}A_{2,0}(\xi_2) & A_{2,1}(\xi_2) \\0_2 & 0_2\end{array}\right), \quad 
A_{2,0}(\xi_2) = \left(\begin{array}{cc}0 & 1 \\\xi_2+3(u_2^*)^2 & 0\end{array}\right),\quad
A_{2,1}(\xi_2) = \left(\begin{array}{cc}0 & 0 \\-u_2^* & 0\end{array}\right),\quad
$$
where $0_2$ denotes the 2x2 zero matrix.

We study the evolution of the tangent spaces of the desired invariant manifolds along $u_2^*$ using the linear variational equation 
\begin{align}\label{e:lv}
W' = A_2(\xi_2) W. \quad 
\end{align}
Such tangent spaces can be studied using exponential trichotomies \cite{sell2002dynamics} to track not only hyperbolic, but also center dynamics about $u_2^*.$ We readily observe that the $w_3$ and $w_4$ directions, corresponding to the $\epsilon_2$ and $r_2$ directions, are constant. Due to the upper diagonal element coupling $w_3$ to $w_2$, the subspace spanned by these directions is not invariant. We do note that the $w_4$ direction is invariant, and spans one dimension of the center bundle. Also, the $w_3 = w_4 = 0$ subspace is invariant and contains the hyperbolic dynamics on both $\R_\pm$. Using the asymptotics of $u_2^*$, one can obtain the following result
\begin{Lemma}\label{l:exp3}
The system \eqref{e:lv} possesses exponential trichotomies  $\R^4 = E^{\rs,\pm}(\xi_2)\oplus E^{\ru,\pm}(\xi_2)\oplus E^{c,\pm}(\xi_2)$ on both $\R_\pm$, with $E^{\rs/\ru,\pm}(\xi_2)$ contained in the $(w_1,w_2)$ subspace and $(0,0,0,1)^T\in  E^{c,\pm}(\xi_2)$ for all $\xi_2$. 
\end{Lemma}
\begin{proof}
On $\R_+$, the asymptotics of $u_2^*$ given in \eqref{e:u2starb} imply that $A_2(\xi_2)$ is a localized perturbation of $ \left(\begin{array}{cc}A_{2,Ai}(\xi_2) & 0_2  \\0_2 & 0_2\end{array}\right)$, where $A_{2,Ai}(\xi_2)  = \left(\begin{array}{cc}0 & 1 \\\xi_2 & 0\end{array}\right) $. As the subsystem $W_h' = A_{2,Ai}(\xi_2)W_h,\,\, W_h = (w_1,w_2)^T$ is the first-order system formulation of a rescaled Airy equation $w_1''-\xi_2 w_1 = 0$, it has an exponential dichotomy on $\R_+$ whose stable and unstable subspaces are spanned by the linearly independent functions $w_1 = \mathrm{Ai}(\xi_2),\,\, \mathrm{Bi}(\xi_2)$. Standard roughness results then give the existence of an exponential dichotomy on $\R_+$ of the hyperbolic subsystem $W_h' = A_{2,0}(\xi_2) W_h$ of \eqref{e:lv}. Thus, since the coupling term $w_3 u_2^*$ vanishes exponentially fast for $\xi_2\rightarrow+\infty$, such roughness results also give the existence of an exponential trichotomy also for the full system. %mention convergence of the stable/unstable bundle to the spans of these two solutions?

On $\R_-$, the hyperbolic subsystem $W_h' = A_{2,0}(\xi_2)W_h$ is an algebraically localized perturbation of another scaled Airy system. In particular, since $u_2^*(\xi_2)\sim \sqrt{-\xi_2}$, we have $A_{2,0}(\xi_2)\sim \left(\begin{array}{cc}0 & 1 \\\ -2\xi_2 & 0\end{array}\right)$, so that the corresponding system is approximated by the first order formulation of $w_1''+2\xi_2 w_1 = 0$ for $\xi_2<0$, which has two linearly independent solutions $w_1 = \mathrm{Ai}(-2^{1/3}\xi_2),\,\, \mathrm{Bi}(-2^{1/3}\xi_2)$ that again give the asymptotic stable and unstable space respectively. Roughness once again gives the existence of a dichotomy for the hyperbolic subspaces. The existence of center subspace is obtained by using the fact that $w_3$ and $w_4$ are constant, and applying a variation of constants argument to solve the following initial value problem for each $w_3$-value,
$$
W_h' = A_{2,0}(\xi_2)W_h + \left(\begin{array}{c} 0 \\ u_2^*(\xi_2) w_3\end{array}\right), \qquad W_h(0) = 0,\quad \xi_2\in \R_-.
$$
\end{proof}

\begin{Proposition}\label{p:transv}
The 1-D unstable and stable subspaces $E^{\ru,-}_2(0)$ and $E^{\rs,_+}_2(0)$ intersect transversely. That is, $\R^2 = E_2^{\ru,-}(0)\oplus E_2^{\rs,+}(0)$.
\end{Proposition}
\begin{proof}
It suffices to consider the 2-D hyperbolic subsystem $W_h' = A_{2,0}(\xi_2) W_h$. First, we note this system is the first-order formulation of the linearized Painlev\'{e}-II equation
\begin{align}
0 = L_0w_1 := w_1'' - (\xi_2+3(u_2^*)^2)w_1.
\end{align}
so that exponentially localized eigenfunctions of the latter correspond to solutions of the former lying in the intersection $E_2^{\ru,-}(\xi_2)\cap E_2^{\rs,+}(\xi_2)$. Here $L_0$ is a $L^2$ self-adjoint operator, with closed densely-defined domain. This operator takes the form of the often-studied Schr\"{o}dinger operator $\p_{\xi_2}^2+V(\xi_2)$ with potential $V(\xi_2) = - (\xi_2 + 3(u_2^*)^2)$. Since the potential satisfies $|V(\xi_2)|\rightarrow+\infty$ as $|\xi_2|\rightarrow +\infty$, standard results \cite[Thm. XIII.47]{ReedSimonIV} give that $L_0$ has no essential spectrum and the discrete spectrum $\{\lambda_{j}\}$ satisfies $\lambda_0\geq\lambda_1\geq\lambda_2\geq...$ with $\lambda_j\rightarrow - \infty$. Such results also give for such operators that if $V(\xi_2)<0$, the ``ground-state`` eigenfunction $\phi_0$ of the eigenvalue is strictly positive $\phi_0>0$.  The asymptotics of $u_2^*$ as $|\xi_2|\rightarrow\infty$ imply that if our potential has $V(\xi_2)>0$ at some point then it has $V(\xi_2)>0$ at most on bounded interval in $\R_-$ and hence,  since $u_2^*$ is smooth, that $m = \max_{\xi_2} V(\xi_2)$ is finite. Hence, the potential of the shifted operator $L_0-(m+\delta)$, for $\delta>0$ small, is strictly negative. Thus, the ground state eigenfunction is strictly positive. 

Now, to obtain a contradiction, assume that the ground-state eigenvalue has $\lambda_0\geq0$. Then, differentiating the Painlev\'{e}-II equation $u_2''+(-\xi_2)u_2 - u_2^3 = 0$ in $\xi_2$, we obtain that
$$
L_0 \p_\xi u_2^* = u_2^*.
$$
Also, we recall that $\p_\xi u_2^*<0$. We then calculate
\begin{align}
\lambda_0 \la \phi_0, \p_\xi u_2^* \ra_{L^2} &= \la \phi_0, L_0\p_\xi u_2^* \ra_{L^2}%\notag\\
= \la \phi_0,u_2^* \ra_{L^2}>0,
\end{align}
which is a contradiction because $\p_\xi u_2^* \cdot \phi_0<0$ so that $\la \phi_0, \p_\xi u_2^* \ra_{L^2}<0$.  Hence we have that $\lambda_0<0$ and therefore that the hyperbolic subsystem $W_h' = A_{2,0}(\xi_2) W_h$ has no exponentially localized solution and hence the two subspaces in question must intersect trivially.
\end{proof}

\begin{Remark}
We also note that Appendix \ref{a:1} gives a rigorous proof of the negativity of the potential, $V(\xi_2) = -(\xi_2+3(u_2^*)^2) < 0$, for all $\xi_2$. This implies that the shift of the operator and results from Schr\"{o}dinger operators is not needed above. One actually need only study the numerical range $\lambda_0 \| \phi_0\|_{L^2}^2 = \la L_0 \phi_0, \phi_0 \ra_{L^2} = -\int_{R} (\p_\xi\phi_0)^2 d\xi + \int_{R} V(\xi) \phi_0^2 d\xi <0$ to infer the negativity of the ground-state eigenvalue.
\end{Remark}

%\begin{Remark}
%We remark that all numerical computations and asymptotic results we've obtained imply that $V(\xi) = -(\xi+3(u_2^*)^2) < 0$ for all $\xi$, but we do not know of a rigorous result. This would imply that the shift of the operator is not needed above. Due to the asymptotics of the system, such an inequality is only in question in a bounded neighborhood of the origin. Furthermore, such an inequality would allow a more straight-forward proof of the above result where one studies the numerical range $\lambda_0 \| \phi_0\|_{L^2}^2 = \la L_0 \phi_0, \phi_0 \ra_{L^2} = -\int_{R} (\p_\xi\phi_0)^2 d\xi + \int_{R} V(\xi) \phi_0^2 d\xi <0$ implying the negativity of the ground-state eigenvalue. See Appendix \ref{a:1} for 
%\end{Remark}

Given the results of Lemma \ref{l:exp3} and Proposition \ref{p:transv} about the linear dynamics around $u_2^*$, we then can conclude the desired intersection properties of the center-unstable and center-stable manifolds around $u_2^*$.

\begin{Proposition}\label{p:cscu}
In a tubular neighborhood of $u_2^*$, the invariant manifolds $W^{\rcu}_1(p^+)$ and $W^{\rcs}_3(p^0)$ intersect transversely with two dimensional intersection containing $u_2^*$.
\end{Proposition}
\begin{proof}
First, we observe that the variational equation \eqref{e:nlvar} and the exponential trichotomies on $\R_\pm$ can be used to construct and continue the manifolds $W^{\rcu}_1(p^+)$ and $W^{\rcs,}_3(p^0)$ in a neighborhood of $u_2^*$ for all $\R_+$ and $\R_-$ respectively.  Furthermore, the tangent spaces of these manifolds along $u_2^*$ are given by the three-dimensional spaces $E^{\rcu,-}_2(\xi_2):= E^{\ru,-}_2(\xi_2)\oplus E^{c,-}_2(\xi_2)$ and $E^{\rcs,+}_2(\xi_2):= E^{\rs,+}_2(\xi_2)\oplus E^{c,+}_2(\xi_2)$. Restricting to a three-dimensional transverse section $\tilde \Sigma_2 = \{\mu_2  = 0\}$, we wish to construct a 1-D family of intersections in $\tilde \Sigma_2$ by writing the invariant manifolds locally as graphs over the relevant tangent bundles and constructing a set of matching equations.

In more detail, the transversality given in Proposition \ref{p:transv} gives a coordinate basis of $\tilde \Sigma_2$ as $\tilde \Sigma_2 = E^{\rs,+}_2(0) \oplus E^{\ru,-}_2(0) \oplus \mathrm{span}\{ e_4\},$  where $e_4$ points in one of the center directions, while the other center direction points along the flow, transverse to $\tilde \Sigma_2.$  We let $(w_s,w_u,w_4)$ denote the corresponding coordinates and also note that as these coordinates arise from the nonlinear variation equation, we have that $u_2^*(0)\cap \tilde \Sigma$ corresponds to $(w_s,w_u,w_4) = 0.$ In these coordinates, we can write the invariant manifolds as graphs
\begin{align}
W^{\rcu}_1(p^+)\cap\tilde\Sigma_2 &= \{(h_-(w_u,w_4),w_u,w_4)\,: |w_u|,|w_4|\leq \delta\},  h_-:\R^2\rightarrow \R,\\
W^{\rcs}_3(p^0)\cap\tilde\Sigma_2 &= \{(w_s,h_+(w_s,w_4),w_4)\,: |w_s|,|w_4|\leq \delta\},  h_+:\R^2\rightarrow \R
\end{align}
for some  $0<\delta\ll1$, for smooth functions $h_\pm$ with tangency conditions $h_-(0,0) = D_{w_u,w_4} h_-(0,0) =0$,  and $h_+(0,0) = D_{w_s,w_4} h_+(0,0) =0$.  Intersections of the two invariant manifolds can then be obtained via the following matching equations
\begin{align}
h_-(w_u,w_4) &= w_s,\\
w_u &= h_+(w_s,w_4).
\end{align}
Note we have equated the $w_4$ component of each graph description. Rearranging these equations, intersections are then given as zeros of the following set of equations $\mc{H}(w_s,w_u;w_4) := (w_s - h_-(w_u,w_4) , w_u - h_+(w_s,w_4) )^T.$ The properties of the graphs then imply
$$
\mc{H}(0,0;0) =(0, 0),\qquad D_{w_s,w_u} \mc{H}(0,0;0) = I_2,
$$
so that, by the Implicit Function theorem, one can solve for $(w_s,w_u)$ as a function of $w_4$ near $(0,0,0)$, giving a one-parameter family of solutions parametrized by the $w_4$ variable, that is $r_2$, which corresponds to $\bar\epsilon.$ 
\end{proof}

Having constructed the heteroclinic between equilibria on the singular sphere, we now use inclination lemmas to also conclude an intersection between the desired invariant manifolds $\mc{M}^{\rcu,+}$, $\mc{M}^{\rcs,0}$.  We state the argument in detail for the dynamics near $K_1$, and outline the argument for $K_3$ as it follows in a similar manner.

\section{Inclination properties and completion of the proof of Theorem \ref{t:2}}\label{s:incl}
\subsection{Inclination properties in chart $K_1$}
\paragraph{Straightening the foliations}
We wish to track how $\mc{M}^{\rcu,+}$ passes through a neighborhood of the equilibrium $p^+$. We note that by the properties of the linearization about $S^+_0$, the manifolds $\mc{M}^{\rcu,+}$ and $W^{\rcu}_1(p^+)$ are both tangent to the collection of center-unstable eigenspaces of $S^+_0$. While they may not coincide due to the non-uniqueness of center manifolds, we find that they leave a neighborhood of $p^+$ exponentially close to each other. 

As the vector-field in $K_1$ coordinates is $C^s$ smooth for all $s\in\N$, the local center-stable and center-unstable manifolds possess the same regularity properties. Hence, classic results \cite{deng1990homoclinic} give that there exists a $C^{s-2}$ change of coordinates to $(w_s,w_u;w_{c,1},w_{c,2})^T$, with $0\in\R^4$ corresponding to $p^+$, which flattens the center manifold of $p^+$ along with its strong-stable and unstable foliations. For simplicity, we let $w_c = (w_{c,1},w_{c,2})^T$. In such coordinates, the system takes the form
\begin{align}
w_s' = \lambda_s w_s + g_s(w_s,w_u;w_c)w_s,\\
w_u' = \lambda_u w_u + g_u(w_s,w_u;w_c)w_u,\\
w_c' = h_c(w_c)+g_c(w_s,w_u;w_c),
\end{align}
where $\lambda_{u/s} = \pm \sqrt{2}$, $h_c:\R^2\rightarrow\R^2$ gives the vector-field on the 2-D center manifold $\mathcal{M}^{\mathrm{c}}$, and the nonlinearities satisfy
\begin{align}
Dg_s(0,0;0) = 0 = Dg_u(0,0;0),\quad g_c(0,w_u;w_c) = g_c(w_s,0;w_c) = 0.
\end{align}
We remark that the coordinates $w_j$ used here  are different than those used in the proof of Proposition \ref{p:cscu}.
Hence the center stable and unstable manifolds are given as the invariant foliations of straight fibers
$$
W^{\rcu}_1(p^+) = \bigcup_{|w_{c,0}|\leq \delta}\{w_c = w_{c,0},\, w_s = 0,\, |w_u|\leq \delta \},\qquad
W^{\rcs}_1(p^+) = \bigcup_{|w_{c,0}|\leq \delta}\{w_c = w_{c,0},\, w_u = 0,\, |w_s|\leq \delta  \}.
$$ 

We then define in and out sections, transverse to the flow of the system, which track how $\mc{M}^{\rcu,+}$ enters and leaves a neighborhood of $p^+$ locally near the sphere. We set,  
\begin{align}
\Sigma_1^\mathrm{in} &= \{(w_s,w_u,w_{c,1},\rho)\,\,:\, |w_s|\leq \alpha, |w_u|\leq \beta, 0\leq w_{c,1}\leq \delta \},\notag\\
\Sigma_1^\mathrm{out} &=\{(w_s,w_u,\Delta,w_{c,2})\,\,:\, |w_s|\leq \tilde\alpha, |w_u|\leq \tilde\beta, 0\leq w_{c,2}\leq \tilde \rho \},\quad\notag
\end{align} for some small positive constants $\alpha,\beta,\delta,\rho, \Delta, \tilde\alpha,\tilde\beta, \tilde \rho.$ 

\paragraph{Dynamics on the center-manifold}
We find that the vector-field $h_c$ is unchanged in these straightened coordinates and the dynamics are governed by
\begin{align}
w_{c,1}' &= \frac{3w_{c,1}^2}{2} (1 - w_{c,2}^4),\\
w_{c,2}' &= -\frac{w_{c,1}w_{c,2}}{2}(1-w_{c,2}^4).
\end{align}
Using a change of coordinates $' = \dot \,\,\,(1-r_1^4)^{-1}$, which preserves the direction of the flow for small enough
values of $r_1$, one can obtain the partially decoupled system
\begin{align}
\dot{w}_{c,1} &= \frac{3w_{c,1}^2}{2} ,\label{e:ctr1a}\\
\dot{w}_{c,2} &= -\frac{w_{c,1}w_{c,2}}{2}\label{e:ctr1b},
\end{align}
where $\cdot$ denotes differentiation with respect to the new variable $\tilde \xi_2$.
This system can be explicitly solved to find that the solution with initial data $w_c(0) = (\epsilon_0,\rho),\,\, 0<\epsilon_0<\Delta$ lying in the section $\Sigma_1^\mathrm{in}$ has the form
$$
(w_{c,1},w_{c,2})(\tilde\xi_2) = \left(\frac{1}{\epsilon_0^{-1} - \frac{3}{2} \tilde\xi_2}\,,\, \rho (1 - \frac{3}{2}\epsilon_0\tilde \xi_2)^{1/3}\right),
$$
and thus intersects the out section $\Sigma_{1}^\mathrm{out}$ where $w_{c,1}(\tilde \xi_{2,out}) = \Delta$ at the time  
$
\tilde \xi_{2,out} = \frac{2}{3}(\epsilon_0^{-1}-\Delta^{-1} ).
$
We also note that the corresponding $w_{c,2}$-component of the solution satisfies
\begin{align}
w_{c,2}(\tilde\xi_{2,out}) = \rho \epsilon_0^{1/3}\Delta^{-1/3}.
\end{align}
Changing coordinates back to $\xi_2$-time, we obtain the transition time as 
$$
\xi_{2,out} = \frac{2}{3}(\epsilon_0^{-1} - \Delta^{-1}) \left(1-\mc{O}(\rho) \right).
$$

Furthermore, we find that $w_{c,1}$ blows up in finite time at $\xi_2 =2/(3\epsilon_0)$ while all initial conditions with $w_{c,1}>0$ satisfy $\lim_{\xi_2\rightarrow 2/(3\epsilon_0)} w_{c,2}(\xi_2) = 0$. Thus, we can define a transition map $\Pi_1:\Sigma_1^{in}\rightarrow \Sigma_1^{out}$ for all points with $w_{c,1}\neq0$, for constants $\alpha, \beta, \rho$ chosen suitably. (In particular, we require $w_{c,2}(\xi_{2,out})\approx \rho (\epsilon_0\Delta^{-1})^{1/3}< \tilde \rho$, $\epsilon_0<\Delta$, and $\beta$ sufficiently small so that $|w_u(\xi_{2,out})|\leq \tilde \beta$.) 

Next we wish to determine how $\mc{M}^{\rcu,+}$ intersects $\Sigma_1^{out}$. Since,  $\mc{M}^{\rcu,+}$ is tangent to $\{w_s = 0\}$ along $S_0^+$, it can be written as a graph over the center-unstable space. In particular, the intersection with the in-section is given as
\begin{align}
\Sigma_1^{in}\cap \mc{M}^{\rcu,+} = \{(h_{\rcu}^{in}(w_u,w_{c,1},\rho),w_u,w_{c,1},\rho) \,:\, |w_u|\leq \beta, 0\leq w_{c,1}\leq \delta \},\label{e:hcuin}
\end{align}
for a $C^r$ smooth function with $h_{\rcu}^{in}(0,0,\rho) = 0, \, \partial_{w_u} h_{\rcu}^{in}(0,0,\rho) = \partial_{w_{c,1}} h_{\rcu}^{in}(0,0,\rho) = 0.$. The Sil'nikov coordinates then allow one to readily track such initial conditions forward to $\Sigma_1^{out}$ using the straight foliation of the center manifold and an inclination result. In particular we find that $\Pi_1$ maps $\Sigma_1^{in}\cap\mc{M}^{\rcu,+}$ onto a set which is exponentially close to $W^{\rcu}_1(p^+) = \{w_s = 0\}.$

\begin{Proposition}\label{p:pi1}
For $0<\delta<\Delta$ and $\Delta,\tilde\beta,\tilde \rho>0$ sufficiently small, there exists a $C>0$, such that the image of $\Sigma_1^{in}\cap\mc{M}^{\rcu,+}$ under the transition map $\Pi_1$ in $\Sigma_{1}^{out}$ can be written as a graph 
$$
\Pi_1(\Sigma_1^{in}\cap \mc{M}^{\rcu,+}) = \{ (w_{s}^{out}, w_{u}^{out}, \Delta, w_{c,2}^{out})\,:\, w_{s}^{out} = h_\rcu^{out}(w_{u}^{out}, w_{c,2}^{out}),\,\, 0<w_{c,2}^{out}<\tilde \rho,
\,\, |w_u^{out}|<\tilde\beta  \}
$$ with $h_\rcu^{out}:\R^2\rightarrow \R$ $C^r$-smooth, satisfying
\begin{equation}
|h_\rcu(w_{u}^{out}, w_{c,2}^{out})|\leq C \re^{\frac{2\lambda_s}{3\Delta}\left((\rho/w_{c,2}^{out})^3 - 1 \right)}, \quad 0<w_{c,2}^{out}<\rho
\end{equation}
uniformly for $|w_u^{out}|<\tilde\beta$.
\end{Proposition}
\begin{proof}
We use a Sil'nikov boundary value formulation to write the $w_s$-coordinate of $\Pi_1(\Sigma_1^{in}\cap \mc{M}^{\rcu,+})$ as a graph over the $w_{u}^{out}$ and  $w_{c,2}^{out}$ coordinates. In other words, we can write solutions with initial data in $\Sigma_1^{in}\cap \mc{M}^{\rcu,+}$ solely in terms of the $\Sigma_1^{out}$ data. 

Using the straightened foliations of the strong stable and unstable dynamics, the results of \cite{deng1990homoclinic} imply there exists a unique solution $(w_s,w_u,w_{c,1},w_{c,2})(\xi;w_s^{in},w_u^{out},w_{c,1}^{in},\rho )$ of the Sil'nikov boundary value problem with boundary data $w_c(0) = (w_{c,1}^{in},\rho), w_s(0) = w_s^{in},\, w_u(\xi_{2,out}) = w_u^{out}$ for $|w_{c,1}^{in}|\leq \delta,\,\, |w_s^{in}|\leq \alpha$, and $|w^{out}_u|\leq \tilde \beta$.
Lemma 3.1 of \cite{deng1990homoclinic} also gives that there exists exponential expansions of the solution components. In more detail, if $w_c^0(\xi)$ denotes the solution on the center manifold with initial condition $w_c^0(0) = w_c(0) = (w_{c,1}^{in},\rho)$, then we have
\begin{align}
w_c(\xi) = w_c^0(\xi) + R(\xi,\xi_{2,out},w_s^{in}, w_u^{out}, w_{c,1}^{in},\rho)
\end{align}
for some $\R^2$ valued $C^{r-2}$-function with $R(0,\xi_{2,out},w_s^{in}, w_u^{out}, w_{c,1}^{in},\rho) = 0$. This perturbation, as well as the hyperbolic parts of the solution satisfy the following estimates for some $C>0$ independent of $\xi_{2,out}$ and the boundary data,
\begin{align}
|w_s(\xi)|&\leq C \re^{\lambda_s\xi},\\
|w_u(\xi)|&\leq C \re^{\lambda_u(\xi - \xi_{2,out})}\\
|R(\xi)|&\leq C \re^{\lambda_s \xi + \lambda_u(\xi - \xi_{2,out})}.
\end{align}

With these general estimates, for each pair $(w_{u}^{out},w_{c,1}^{in})$, we evaluate the $w_u$-component of the Sil'nikov solution at $\xi = 0$ and set $w_s^{in} = h_\rcu^{in}(w_u(0) ;w_{c,1}^{in},\rho)$ where $h_\rcu^{in}$ is the graph for the center-unstable manifold defined in \eqref{e:hcuin} above. Furthermore, we also use 
$$
w_{c,2}^{out} = w_{c,2}(\xi_{2,out}) = \rho (w_{c,1}^{in}/\Delta)^{1/3}\cdot(1+\mc{O}(\rho))
$$
 given above to  write $w_{c,1}^{in}$ in terms of $w_{c,2}^{out}$, obtaining $w_{c,1}^{in} = \Delta (w_{c,2}^{out}/\rho)^3\cdot(1+\mc{O}(\rho))$.  The graph $h_\rcu^{out}$ is then given as the function 
$$
h_\rcu^{out}(w_u^{out},w_{c,2}^{out}):= w_s(\xi_{2,out};w_s^{in}, w_u^{out}, w_{c,1}^{in},\rho),
$$
with the aforementioned substitutions for $w_s^{in}$ and $ w_{c,1}^{in}$. The estimates on $h_\rcu^{out}$ then follow from using the substitutions and the exponential estimate on $w_s(\xi)$ above as well as the expansion,
$$
\xi_{2,out} = \frac{2}{3}((w_{c,1}^{in})^{-1} - \Delta^{-1})(1+\mc{O}(\rho)) = \frac{2}{3\Delta}\left((\rho/w_{c,2})^3 - 1 \right)(1+\mc{O}(\rho)).
$$
\end{proof}

%Let $W_*(\xi) = (w_s,w_u,w_{c,1},w_{c,2})(\xi)$ be the solution with initial condition starting in $\Sigma_1^{in}\cap \mc{M}^{\rcu,+}$
%$$
%(w_s,w_u,w_{c,1},w_{c,2})(0) = (h_{\rcu}(w_u^{in},w_{c,1}^{in},\rho), w_u^{in},w_{c,1}^{in},\rho),
%$$
%with $|w_u^{in}|\leq \alpha, 0<w_{c,1}\leq \delta$. 
%The graph $h_\rcu$ is then given as the function $w_s(\xi_{2,out},w_s^{in}, w_u^{out}, w_{c,1}^{in},\rho)$. Since $w_{c,2}^{out} = \rho (w_{c,1}^{in}/\Delta)^{1/3}*(1+\mc{O}(\rho))$  we can write $w_{c,1}^{in}$ in terms of $w_{c,2}$ and thus obtain the desired 2-D graph.

%To obtain the estimate for the $w_s$ component we evaluate and write $w_s(\xi_{2,out})$ as a function of the boundary data $w_{c,2}^{out}$. Here we use $w_{c,2}^{out} = \rho (w_{c,1}^{in}/\Delta)^{1/3}*(1+\mc{O}(\rho))$ to find that, for $w_{c,2}^{out}<\rho$ that 
%$$
%\xi_{2,out} = \frac{2}{3}((w_{c,1}^{in})^{-1} - \Delta^{-1})(1+\mc{O}(\rho)) = \frac{2}{3\Delta}\left((\rho/w_{c,2})^3 - 1 \right)(1+\mc{O}(\rho)).
%$$

\subsection{Inclination properties on chart $K_3$}
One can also show that $ W^{\rcs}_3(p^0)$ is exponentially close to $\mc{M}^{\rcs,0}$ in a neighborhood of $p_0$ in the $K_3$ chart. The result follows in the same way as done in $K_1$ but one reverses time, flowing backwards from the ``out" chart to the ``in" chart. To this end one can once again change to coordinates $(w_s,w_u,w_{c,1},w_{c,2})$ which straighten the strong fibers so that $W^{\rcs}_3(p^0)$ is locally given by $\{w_u = 0\}$. Roughly $w_{c,1}$ corresponds to $\epsilon_3$ and $w_{c,2}$ to $r_3$. We recall that the linearization at $p^0$ has hyperbolic eigenvalues $\tilde\lambda_u = 1,\tilde\lambda_s =-1$. One then defines in and out sections
\begin{align}
\Sigma_{3}^{in} &= \{ (w_s,w_u,\Delta,w_{c,2} )\,:\, |w_s|\leq\tl\alpha, |w_u|\leq \tl\beta,  0\leq w_{c,2}<  \rho\},\notag\\
\Sigma_{3}^{out} &= \{ (w_s,w_u,w_{c,1},\rho)\,:\, |w_s|\leq\alpha, |w_u|\leq \beta,  0\leq w_{c,2} < \delta\},\notag
\end{align}
along with a transition map $\Pi_3:\Sigma_3^{out}\rightarrow \Sigma_3^{in}$ formed by the time-reversed flow.
Using the center manifold dynamics given by
\begin{align}
w_{c,1}' &= -\frac{3w_{c,1}^2}{2} (1 - w_{c,2}^4),\\
w_{c,2}' &= \frac{w_{c,1}w_{c,2}}{2}(1-w_{c,2}^4),
\end{align}
and a Sil'nikov boundary value problem, we then have the following inclination result.
\begin{Proposition}\label{p:k3inc}
For $0<\delta<\Delta$ and $\Delta,\tilde\beta, \rho>0$ sufficiently small, there exists a $C>0$, such that the intersection of the image of $\Sigma_3^{out}\cap\mc{M}^{\rcs,0}$ under the transition map $\Pi_3$ in $\Sigma_{3}^{in}$ can be written as a graph 
$$
\Pi_3(\Sigma_3^{out}\cap\mc{M}^{\rcs,0}) = \{ (w_s^{in},w_u^{in},\Delta, w_{c,2}^{in}\,:\, 
w_u^{in} = h_{\rcs}^{in}(w_s^{in},w_{c,2}^{in}),\,\,0<w_{c,2}^{in}< \rho,
\,\, |w_s^{in}|<\tilde\alpha  \},
$$
with $h_{\rcs}^{in}:\R^2\rightarrow \R$ a $C^r$-smooth function satisfying
\begin{equation}
|h_\rcs(w_{s}^{in}, w_{c,2}^{in})|\leq C \re^{\frac{2\tilde\lambda_s}{3\Delta}\left((\rho/w_{c,2}^{in})^3 - 1 \right)}, \quad 0<w_{c,2}^{in}<\rho,
\end{equation}
uniformly for $|w_s^{in}|<\tilde\beta$.
\end{Proposition}

%Due to the fact that $\mc{M}^{\rcs,0}$ has $u_3 = v_3 = 0$, we do not need inclination properties to locate it in relation to the local manifold $W^{\rcs}_3(p^0)$. Indeed $W^{\mathrm{c}}_3(p^0)$ can be written as a graph over $\mc{M}^{c,0} = \{u_3 = v_3 = 0\}$, and $W^{\rcs}_3(p^0)$ can be written as a graph over the direct sum of $\mc{M}^{c,0}$ and the strong stable subspace. %and is tangent to $\mc{M}^{\rcs,0}$ along $S^0_0$. 
%Furthermore, we claim (??) that in a neighborhood of the origin, the two invariant manifolds are exponentially close. {\color{blue} (Have this for restriction to $r_3 = 0$, but what about for $r_3>0$?)}
%
%Hence, say by once again straightening the foliations of the center manifold $\mc{M}^{c,0}$, one can show that the intersections of these invariant manifolds with the section
%$$
%\Sigma_{3}^{in} = \{ (u_3,v_3,\epsilon_3, r_3)\,:\, |u_3|,|v_3|\leq \alpha, 0\leq r_3 < \tilde \rho\}
%$$
%are exponentially close in $r_3$ to one another.

\subsection{Tracking across the re-scaling chart and completion of proof of Theorem \ref{t:2}}

To complete the proof, one translates the intersections $\Sigma_1^{out}\cap \mc{M}^{\rcu,+}$ and $\Sigma_3^{in}\cap \mc{M}^{\rcs,0}$ into the chart $K_2$ using $\kappa_{12}$ and $\kappa_{23}^{-1}$ respectively, and then flows them forward and backward respectively to construct an intersection in the section $\tilde \Sigma_2 = \{\mu_2 = 0\}$. In short, the existence of a 2-D intersection follows from the exponential closeness of $\Sigma_1^{out}\cap \mc{M}^{\rcu,+}$ to $\Sigma_1^{out}\cap W_1^\rcu(p^+)$ and  $\Sigma_3^{in}\cap \mc{M}^{\rcs,0}$ to  $\Sigma_3^{in}\cap W_3^\rcs(p^0)$, and the inclination properties about the transverse heteroclinic $u_2^*\in  W_1^\rcu(p^+)\cap  W_3^\rcs(p^0) $ in $K_2$.%, and the transverssality of the intersection of $W^\rcu_1(p^+)$ and $W^\rcs_3(p^0)$.

For $(u_1,v_1,\epsilon_1,r_1)\in\Sigma_1^{out}$, recall we have $\epsilon_1 = \Delta$ and thus 
$$
(u_2,v_2,\mu_2,r_2) = \kappa_{12}(u_1,v_1,\Delta,r_1) = (\Delta^{-1/3} u_1,\Delta^{-2/3} v_1, \Delta^{-2/3},\Delta^{1/3} r_1).
$$
We then define the entry section in $K_2$ as $\Sigma_2^{in} = \kappa_{12}\Sigma_1^{out} =  \{\mu_2 = \Delta^{-2/3}\}$ and set $\xi_{2,in} = -\Delta_1^{-2/3}$ so that $u_2^*(\xi_{2,in})\in \Sigma_2^{in}$.   For $(u_3,v_3,\epsilon_3,r_1)\in \Sigma_3^{in}$ we similarly have 
$$
(u_2,v_2,\mu_2,r_2) = \kappa_{23}^{-1}(u_3,v_3, \Delta, r_3) = ( \Delta^{-1/3} u_3, \Delta^{-2/3} v_3, -\Delta^{-2/3}, \Delta^{1/3}r_3),
$$
and thus define $\Sigma_2^{out} = \kappa_{23}^{-1}\Sigma_3^{in} = \{\mu_2 = -\Delta^{-2/3}\}$ and $\xi_{2,out} = \Delta_1^{-2/3}$ so that $u_2^*(\xi_{2,out})\in \Sigma_2^{out}$.  

Next, using the structure of $\kappa_{12}$ and the result of Proposition \ref{p:pi1}, we have that $\mc{M}_{in}:=\kappa_{12}(\Sigma_1^{out}\cap \mc{M}^{\rcu,+})$ is $\mc{O}(\Delta^{-2/3} \re^{-C/\Delta})$ away from $\kappa_{12}(\Sigma_1^{out}\cap W_1^{\rcu}(p^+)$ for some constant $C>0$ for $\Delta$ sufficiently small and $r_1<\rho/2$. Thus, we observe that $\mc{M}_{in}$ is a 2-D manifold in $\Sigma_2^{in}$ which intersects the linear stable bundle $E^{\rs,-}(\xi_{2,in})$ transversely. % and thus with points lying in the strong-stable fibers normal to $W_1^\rcu(p^+)$.
Letting $\Phi_{\xi_2}$ denote the flow of \eqref{e:k2a} - \eqref{e:k2d} in $K_2$, the hyperbolic inclination properties about $u_2^*$ then imply that $\Phi_{\xi_2} (\mc{M}_{in})$ exponentially converges onto $W_1^\rcu(p^+)$ as $\xi_2$ increases and can be written as a graph over the center-unstable bundle $E^{\rcu,-}_2(\xi_2)$. Using the monotonicity properties of the $\mu_2$ flow for $0\leq r_1 \ll1$, the transition map $\Pi_{2,in}:\Sigma_{2}^{in}\rightarrow \tilde\Sigma_2$ defined by the flow $\Phi_{\xi_2}$ is well-defined, with time of flight $\xi_2 = -\Delta^{-2/3} + \mc{O}(r_2)$.  Therefore we conclude that $\Pi_{2,in}\mc{M}_{in}$ can be written as a graph over $E^{\rcu,-}(0)$ and is exponentially close to $\tilde \Sigma_2\cap W_1^{\rcu}(p^+)$ in a neighborhood of $u_2^*(0).$

In a similar manner, Proposition \ref{p:k3inc}, gives that $\mc{M}_{out}:= \kappa_{23}^{-1}( \Sigma_3^{in}\cap\mc{M}^{\rcs,0})$ is $\mc{O}(\Delta^{-2/3} \re^{-C/\Delta})$ away from $\kappa_{23}^{-1}(\Sigma_3^{in}\cap W_3^\rcs(p_0))$ for $r_3<\rho/2$.  Defining $\Pi_{2,out}:\Sigma_2^{out}\rightarrow \tilde \Sigma_2$ by using the backwards flow $\Phi_{\xi_2},\, \xi_2<0$, the inclination properties about $u_2^*$ imply that $\Pi_{2,out} \mc{M}_{out}$ can be written as a graph over $E^{\rcs,+}(0)$ and is exponentially close to $\tilde \Sigma_2\cap W_3^{\rcs}(p^0)$. Then using the transversality properties of the intersection $W_1^\rcu(p^+)\cap W_3^\rcs(p^0)$ we conclude the existence of the desired intersection, completing the existence result of Theorem \ref{t:2}. 

Estimate \eqref{e:uest} is obtained by putting the above results for charts $K_1$-$K_3$ together and translating back to the original coordinates. Here $\sqrt{2}w_{HM}$ is given by $u_2^*$ in the $K_2$ coordinates. We see that the heteroclinic, formed by $u_2^*$, obtained in the singular limit of the above geometric desingularization analysis is the leading-order approximation of the desired front solution in the region $|\mu|\lesssim \rho \epsilon^{2/3}$, where $\rho>0$ is a small, $\epsilon$ independent constant.  Moreover, given that $\mu\sim -\epsilon \xi$  in a neighborhood of the origin,  the leading order asymptotics hold for $|\xi|\leq \rho \epsilon^{-1/3}.$  Unwinding the scalings from the blow-up coordinates, % we have established the existence of the heteroclinic orbit in the singular limit, and 
 the desired heteroclinic front solution asymptotically satisfies
\begin{align}
|u^*(\xi) - \epsilon ^{1/3} u_2^*(\epsilon^{1/3}\xi)| &\leq \rho \epsilon^{2/3},  \qquad |\xi|\leq \rho \epsilon^{-1/3}.
\end{align}

\section{Discussion and future directions}\label{s:disc}
To conclude, we discuss several immediate consequences of our results and highlight several avenues for future research. We expect our $c>0$ results and the phenomenological mechanisms studied in this work to govern front dynamics for any scalar reaction-diffusion equation 
\beq\label{e:srd}
u_t = u_{xx} + f(x-ct,u), \quad u(x,t)\in\R,
\eeq 
where $f$ is smooth with slowly varying heterogeneity which moderates the stability of a homogeneous equilibrium state and undergoes a bifurcation to a stable equilibrium state as $\xi$ moves from $+\infty$ to $-\infty$. For example, we expect a result similar to Theorem \ref{t:0} to hold for \eqref{e:srd} for a slowly-varying Fisher-KPP type nonlinearity $f(\xi,u) = \mu(\xi) u - u^2$ with $\mu$ defined as above. Furthermore, we expect the underlying mechanisms studied here to govern the formation of front solutions in slowly-varying super-critical pattern-forming equations, such as the real and complex Ginzburg-Landau equations, the Swift-Hohenberg equation, and many relevant reaction-diffusion equations.

This work can be viewed as a new contribution to the nascent body of research studying dynamic bifurcation in spatially extended systems. It points to a new set of problems which are of interest both for applications and for mathematics. It also provides a novel application-motivated example of how techniques from geometric singular perturbation theory can be used to uncover and precisely characterize front dynamics in a slowly-varying environment.  From a technical perspective, it also provides a testbed to apply dynamic bifurcation techniques in a higher-order system with multiple additional hyperbolic directions as well as a control parameter (in our case $c$) which governs the specific type of dynamic bifurcation.

\subsection{Stability}\label{ss:stab}
%It is not difficult to see that the solutions constructed in \eqref{t:0} are asymptotically stable, that is, they attract all nearby initial conditions exponentially. In particular we have the following asymptotic stability result. We  therefore let $\bar{u}$ be the first component of $\Gamma_\eps$, suppressing the dependence on $\eps$ and $c$. 
%\begin{Proposition}
%Assuming the hypotheses of Theorem \ref{t:0}, the front solution $\bar u$ is exponentially asymptotically stable as a solution of the equation \eqref{e:ac} in the co-moving frame variables $(\zeta = -(x-ct),t)$ 
%\begin{equation}
%u_t = u_{\zeta\zeta} - cu_\zeta +\mu u - u^3
%\end{equation}
%in either of the spaces $BC^0(\R)$ or $L^2(\R)$.
%\end{Proposition}
%\begin{proof}

It is not difficult to see that the solutions constructed in Theorem \ref{t:0} and Theorem \ref{t:2} are asymptotically stable, that is, they attract all nearby initial conditions exponentially in the equation \eqref{e:ac}, posed in the co-moving frame  $\zeta = -(x - ct)$,
\begin{equation}
u_t = u_{\zeta\zeta} - cu_\zeta +\mu u - u^3
\end{equation}
%For this, it is sufficient to show that the spectrum of the linearization at such a solution has strictly negative real part. Write therefore $\bar{u}$ for the first component of $\Gamma_\eps$, suppressing the dependence on $\eps$ and $c$, and recall that $\bar{u}_\zeta>0$ and  $\mu_\xi>0$ so that . We then need to consider the spectrum of
%
Given standard results on asymptotic stability in semilinear PDE (see \cite[Ch. 5]{henry1981geometric} or \cite{kapitula2013spectral}), it is sufficient to show that the spectrum of the linearization at such a solution has strictly negative real part. Note that there is no spatial translation eigenvalue due to the heterogeneity. We write therefore $u^*$ for the first component of $\Gamma_\eps$, suppressing the dependence on $\eps$ and $c$, and recall that $u^*_\zeta>0$ and  $\mu_\zeta>0$. We then need to consider the spectrum of the linearization
\begin{equation}\label{e:L0}
 \mathcal{L}_0 u:=u_{\zeta\zeta} - c u_\zeta + (\mu-3(u^*)^2)u,
\end{equation}
considered as a closed and densely defined operator on, say $BC^0(\R)$. This operator is conjugate to a formally self-adjoint operator
\begin{equation}\label{e:Lc}
 \mathcal{L}_c  u:=(e^{-c\zeta/2} \mathcal{L}_0 e^{ c\zeta/2})u=   u_{\zeta\zeta} + (\mu-\frac{c^2}{4} -3(u^*)^2)u,
\end{equation}
Indeed, $ \mathcal{L}_c$ is clearly self-adjoint on $L^2(\R)$. A quick calculation shows that the essential spectra of $\mathcal{L}_0$ and $\mathcal{L}_c$ have strictly negative real part. Moreover, inspecting the decay of eigenfunctions, that is, to solutions of $\mathcal{L}_c u=\lambda u$ with $\Re\lambda\geq 0$, one quickly sees that the point spectra of $\mathcal{L}_0$ and $\mathcal{L}_c$ coincide. Similarly, point and essential spectra do not depend on the choice $BC^0$ versus $L^2$, so that we restrict ourselves to excluding eigenvalues $\lambda\geq 0$ to $\mathcal{L}_c$ in $L^2$.

To exclude such eigenvalues, we proceed as in Proposition \ref{p:transv} above. Assume that there is a maximal eigenvalue $\lambda_0\geq 0$ with eigenfunction $u_0(\zeta)$, which then has a sign and we assume $u_0(\zeta)>0$. Next, recall that $u^*_{\zeta\zeta}-cu^*_\zeta+\mu u^*-(u^*)^3=0$, so that, by differentiating with respect to $\zeta$, we find
\begin{equation}\label{e:der}
  \mathcal{L}_0 u^*_\zeta + \mu_\zeta u^*=0,
\end{equation}
or
\begin{equation}\label{e:der1}
  \mathcal{L}_c \left(e^{-c\zeta/2}u^*_\zeta\right) + \left(e^{-c\zeta/2} \mu_\zeta\right) u^*=0.
\end{equation}
One quickly verifies that $\left(e^{-c\zeta/2}u^*_\zeta\right)$ is exponentially localized, as is  $\left(e^{-c\zeta/2} \mu_\zeta\right) u^*$, and we shall exploit this property by testing the eigenvalue against these functions. We find from $\mathcal{L}_c u_0=\lambda u_0$ after integrating against $e^{-c/2} u^*_\zeta$, that, using first that $u^*_\zeta,u_0>0$, $\lambda_0\geq 0$, self-adjointness of $\mathcal{L}_c$, and \eqref{e:der1},
\begin{align*}
 \la \re^{-c\zeta/2}u^*_\zeta,\mathcal{L}_c u_0\ra_{L^2} &=\lambda_0\la \re^{-c\zeta}u^*_\zeta,u_0\ra_{L^2}\geq 0,\\
\la \mathcal{L}_c(e^{-c\zeta/2}u^*_\zeta), u_0\ra_{L^2}&\geq 0,\\
 \la -\re^{-c\zeta/2}\mu_\zeta u^*, u_0\ra_{L^2}&\geq 0,
\end{align*}
a contradiction to $\mu_\zeta,u^*,u_0>0$.

Non-monotone fronts, discussed next, are likely unstable with increasing Morse index. Using Maslov index arguments, for example, one would seek to establish the additional unstable eigenvalues for each node created in the solution.

\begin{Remark}
For slowly-ramped fronts in systems without a comparison principle and the monotonicity properties exploited above, the techniques of \cite{goh2021spectral} should be of use in locating spectrum and proving stability. In more detail, the front $u^*$ is exponentially close to the trivial state $u = 0$ in the $\mc{O}(\epsilon^{2/3})$-wide interval $\mu\in(\mu_\mathrm{fr},\mu_c)$. Here, the trivial state is absolutely unstable.  Since this region is $\mc{O}(\epsilon^{-1/3})$-wide in the spatial variable $\xi$, one expects all but finitely many of the point spectrum of $\mc{L}$ to lie close to the absolute spectrum of the trivial state. Following the aforementioned work, one would projectivize the eigenvalue problem $\mathcal{L}v = \lambda u$ and track the slow winding of the unstable subspace as $\xi$ passes from $ \xi_\mathrm{fr}$ to $\xi_c$. Since the winding frequency slows to zero as $\xi$ increases and $\mu\rightarrow \mu_c^-$, one does not expect intersections to exist for $\lambda\geq0$, and thus no unstable eigenvalues.
\end{Remark}

\subsection{Fronts with non-monotonic tails}\label{ss:nonmon}
As briefly mentioned in the introduction, our approach for dynamic quenching can readily be extended to prove the existence of fronts with oscillatory tails as well as for fronts with $\lim_{\zeta\rightarrow+\infty} u(\zeta) = -1$. Such fronts arise from the slow attracting manifold $S_\epsilon^a$ on the plane $U_0$, and its corresponding strong unstable foliation, winding all the way around the cylinder before intersecting the stable manifold. After $S_\epsilon^a$ passes around the fold point, the $z$-dynamics in \eqref{e:zut1} cause the manifold to blow-up to negative infinity in finite time. This corresponds to the trajectory moving to another chart of the cylinder. Dynamics on this chart can be coordinatized with a blow-up in the $v$-direction, $w = v/u$, where $w = 0$ roughly corresponds to $z = \infty$.  On the invariant cylinder the dynamics are governed by the equation
$$
w_\zeta = 1 - cw + (\theta+c^2/4) w^2
$$  
and thus consist of constant drift at leading order for $w\sim0$. After tracking the slow manifold through this chart one would then study the dynamics in the $-u$ blow-up with the coordinate $\hat z = -v/u$ and find intersections of the unstable manifold with the stable manifold of the $u = -1$ equilibrium. Further tracking it around the cylinder back to the original chart one could then find another intersection with the original stable manifold.  We once again remark that one could use a polar coordinate blowup of the dynamics near $(u,v) = (0,0)$ without the use of charts.  See Figure \ref{f:cyl} for a schematic depiction. These dynamics are similar those found in the work \cite{carter2018unpeeling} which finds Airy points along the repelling slow manifold of the Fitzhugh-Nagumo system, where the local linear stability type of the point in the fast subsystem changes from being an unstable node to an unstable spiral.
As we expect the fronts with non-trivial winding around the cylinder to be unstable, we do not rigorously pursue their existence here.   

Interestingly, non-monotonicity can also result from a small bias in the cubic, leading to a Painlev\'e II equation with an asymmetric cubic $\eta w + 2 w^3 + k$ for some $k>0$. We expect a variety of interesting applications and more complex results relating to the competition between pulled and pushed fronts, and refer to \cite{troy} for a discussion of applications and analysis of relevant, non-monotone, special solutions in the stationary case $c=0$.  

\begin{figure}[h!]
\centering
\includegraphics[trim = 0.0cm 0.0cm 0.00cm 0.05cm,clip,width=0.6\textwidth]{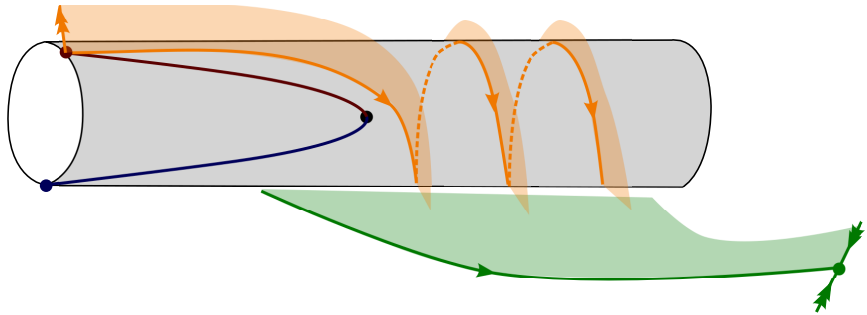}\hspace{-0.2in}
\caption{Schematic depiction of dynamics near the polar coordinate blow-up of the line $(0,0,\mu)$ into the cylinder $\{r = 0\}$ (grey). Colors correspond to objects depicted in previous figures. Winding of the unstable manifold $W^u(0,0,-1)$, which in the blow-up coordinates consists of the attractive slow manifold $S_\epsilon^a$ in the cylinder (orange trajectory) and its strong unstable foliation (orange sheet), allows for additional intersections between the stable manifold $W^s(1,0,1)$ (green). Red and blue curves denote the $\epsilon = 0$ curves of equilibria. Furthermore, this winding allows for connections with the stable manifold $W^s(-1,0,1)$ of the other equilibrium, $u\equiv1$, at $\mu = 1$ (not depicted).    }\label{f:cyl}
\end{figure}

\subsection{Fronts for asymptotically small speeds $0< c\ll1$}\label{e:csm}
We now discuss front solution behavior and asymptotics in the limit where the quenching speed $c$ is asymptotically small. First of all, numerical results in Figure \ref{f:smc} of the difference $\mu_\mathrm{fr} - \mu_c$, show that as $c$ decreases, the $\epsilon$ interval on which the $\epsilon^{2/3}$-delay is valid shrinks. In other words, we observe that as $c$ decreases, the value of $\epsilon_0$ given in Theorem \ref{t:0} goes to zero. Indeed for sufficiently small $c$, the front interface lies ahead of $\mu_c$ so that $\mu_\mathrm{fr} - \mu_c<0$, at least for the numerical range of values $\epsilon$ used in computation. Thus the front tail bleeds into the region where $\mu\leq0$. From a PDE perspective this advance of the front tail could be viewed as being caused by the comparatively large role diffusion plays when the quench is slow moving. Also, we find below that for such small speeds, the front profile resembles the unique connecting solution of Painlev\'e's second equation observed in the $c = 0$ case discussed in Section \ref{s:c0} above.

To understand this behavior one could alternatively seek to understand the limit $c\rightarrow0^+$ for $\epsilon$ fixed small. Such numerics are also depicted in the right plot of Figure \ref{f:smc}. We find, as $c$ decreases the front follows $\sqrt{\mu}$ for a larger range of $\xi$ but decays more slowly as $\xi$ increases past 0.   Furthermore, we can also track the change in front behavior by tracking the value $u(\xi_c)$, where $\xi_c$ is such that $\mu(\xi_c) = \mu_c$. This indicates the size of the front at the leading order take off point. Since there is an additional delay in the front interface for $\epsilon$ sufficiently small, we expect these values to remain exponentially small. In Figure \ref{f:uceps}, we indeed find that the interval of $\epsilon$ values where $u$ is exponentially small decreases as $c$ decreases.  In the limit $c = 0$, there is no such interval and a linear fit of the log-log data here indicates that $u(\xi_c = 0)$ scales like $\epsilon^{1/3}$. 

Further evidence that there is a transition at $c \sim \epsilon^{1/3}$ in
the dynamics of the fronts comes from some preliminary analysis. On the
one hand, for asymptotically small values of $c$ which satisfy $c \ll
\eps^{1/3}$, it turns out that the system is again a perturbation of the
Painlev\'e II equation, as is the case for $c=0$. Indeed, for $c>0$, one
starts with system \eqref{e:uvmea}-\eqref{e:uvmed} and adds the term $-cv$
to the second component. For asymptotically small values
$c=\epsilon^\sigma \tilde{c}$ where $\tilde{c}=\mathcal{O}(1)$ with
respect to $\epsilon$, one may use the same dynamically rescaled
coordinates \eqref{e:barcoor} and the same method of geometric
desingularization as used above. In particular, in the rescaling chart
$K_2$, one finds the same system \eqref{e:k2a}-\eqref{e:k2d}, as in the
analysis of the case $c=0$, but now with the term $-r_2^{3\sigma-1}
\tilde{c} v_2$ included in the second component, \eqref{e:k2b}. This term
is a small perturbation term for $0<r_2\ll 1$ as long as $\sigma > 1/3$.
Hence, for $c \ll \eps^{1/3}$, the structure of the full system is also
that of a small perturbation of the Painlev\'e II equation, as above in the
analysis for $c=0$.

On the other hand, for asymptotically small values of $c$ which satisfy $c
\gg \eps^{1/3}$, preliminary analysis suggests that one can extend the
method of proof of Theorem 1 down to $c \gg \eps^{1/3}$.
For asymptotically small values of $c$, the boundary of $U_0^r$ at $\{ z =
- c/2\}$ gets close to the axis, and with $c \gg \eps^{1/3}$, the method of Sections 2-4 can still be used to show that the
invariant manifolds intersect transversely. Moreover, the terms in the
asymptotic expansion \eqref{e:mufr} stay well ordered for $c \gg
\eps^{1/3}$.

\begin{figure}[h!]
\centering
\includegraphics[trim = 0.0cm 0.0cm 0.0cm 0.0cm,clip,width=0.38\textwidth]{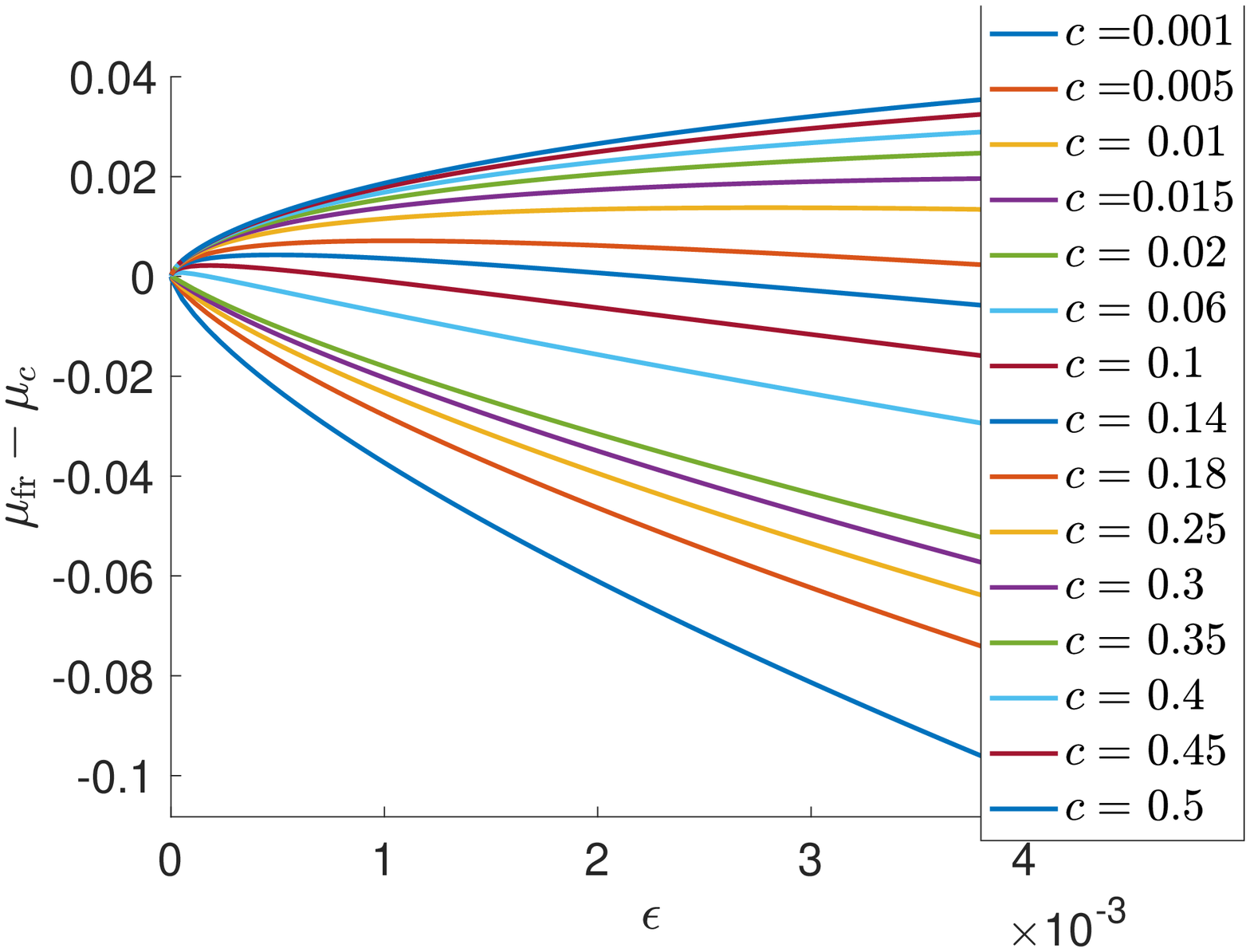}\hspace{-0.2in}\qquad
\includegraphics[trim = 0.0cm 0.0cm 0.05cm 0.05cm,clip,width=0.6\textwidth]{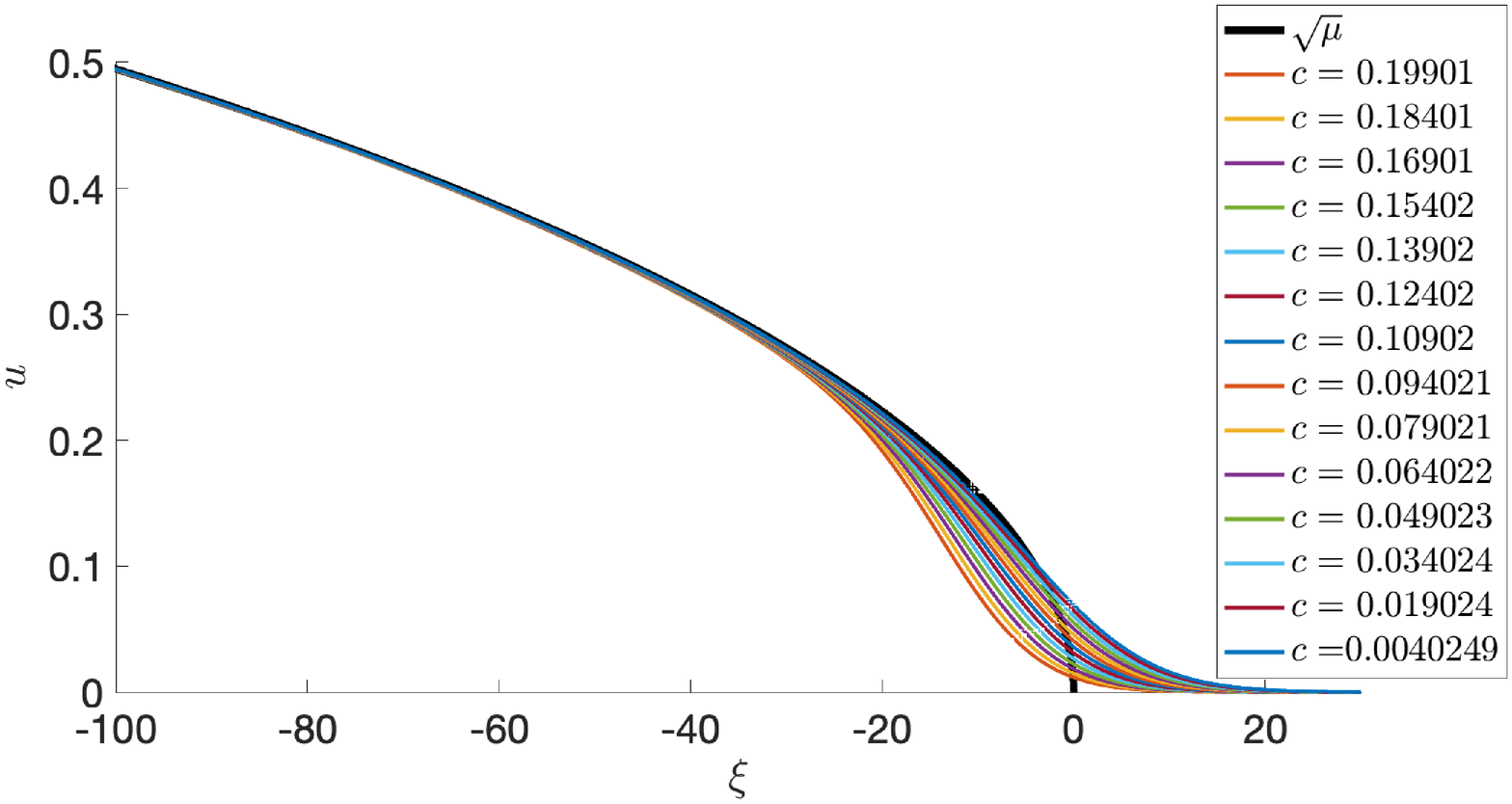}\hspace{-0.2in}
\caption{Left: plots of the numerically measured difference $\mu_\mathrm{fr} - \mu_\mathrm{c}$ against $\epsilon$, for a range of $c$ values (given in legend), curves increase as $c$ increases; Right: Plots of the front profile near $\mu = 0$ for a range of $c$-values (in legend), curves decrease as $c$ increases with $\epsilon =  0.0025$ fixed.    }\label{f:smc}
\end{figure}

\begin{figure}[h!]
\centering
\includegraphics[trim = 0.0cm 0.0cm 0.05cm 0.05cm,clip,width=0.45\textwidth]{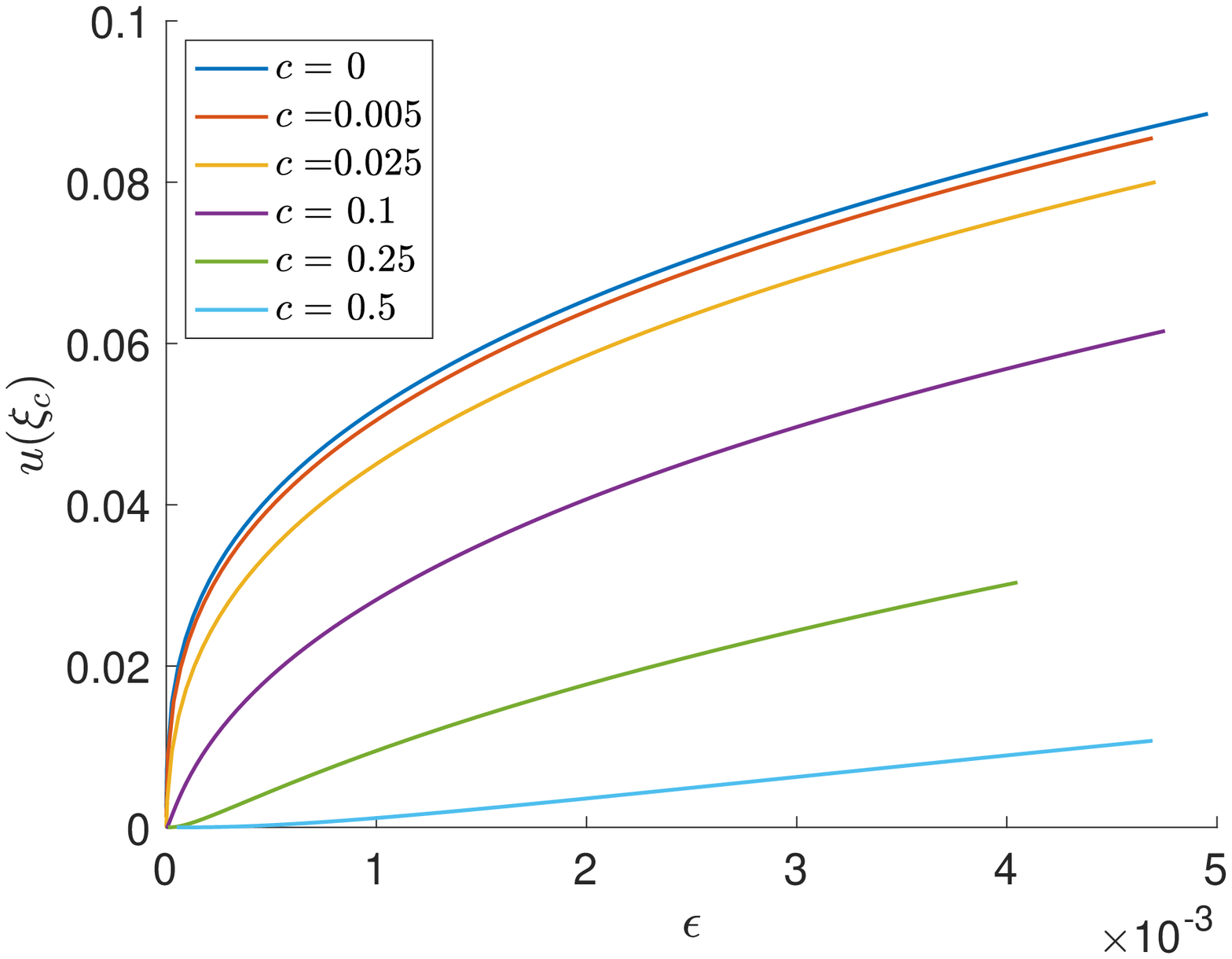}\hspace{-0.2in}\qquad
\includegraphics[trim = 0.0cm 0.0cm 0.05cm 0.05cm,clip,width=0.45\textwidth]{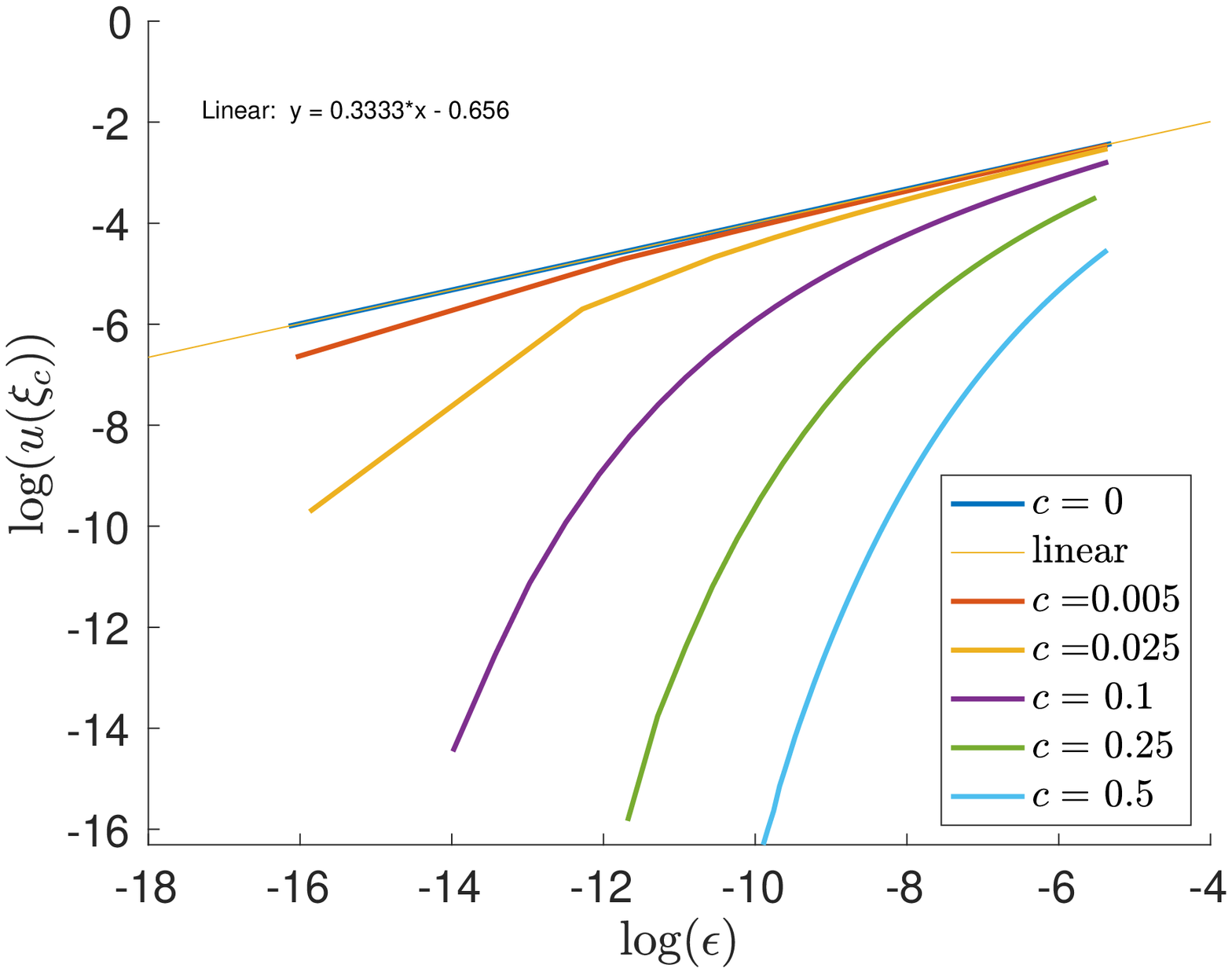}\hspace{-0.2in}
\caption{Left: plot of the values $u(\xi_c)$ against $\epsilon$ for a range of speeds $c$. Note the curves from convex to concave as $c$ decreases; Right: Log-log plot of $u(\xi_c)$ against $\epsilon$, showing that the front height becomes exponentially small in $\epsilon$ for moderate speeds $c$. Also included is a linear fit of the $c = 0$ curve (light yellow, with fit equation printed), indicating that $u(\xi_c)$ scales like $\epsilon^{1/3}$ in this case.  }\label{f:uceps}
\end{figure}

\subsection{$c>2$ and spatially homogeneous slow quenches}\label{s:homq}
As mentioned in the introduction, we expect no traveling wave solutions to exist for quenching speeds $c>2$. In the full dynamics of the PDE  \eqref{e:ac}, since $\mu$ approaches $1$ as $t\rightarrow+\infty$ for all points $x\in \R$,  we expect compactly supported perturbations to spread with asymptotic speed $2$.
% See for example the space-time plot of the solutions of \eqref{e:ac} with $c>2.$ {\color{blue} add this figure below?!!}. 
To characterize this regime, we introduce an altered parameterization of the quench
\begin{align}\label{e:alp}
u_t &= u_{xx} + \mu(\alpha x - t)u - u^3,\\
\mu(\eta) &= -\tanh(\epsilon \eta), \quad \mu(0) = 0.
\end{align}
with a new traveling wave variable $\eta = \alpha x- t$. Here the new parameter $\alpha\in\R$ gives the speed of the moving quench as $1/\alpha$, and thus the range $\alpha>1/2$ corresponds to the case $c\in(0,2)$ studied above, while the range $\alpha\in(0,1/2)$ corresponds to $c>2$, and $\alpha = 0$ to a spatially homogeneous quench which uniformly renders the trivial state unstable.  In the latter two cases, one immediate question of  interest is how the front interface moves and how its speed asymptotically approaches 2.  In the $\alpha = 0$ case, where $\mu$ slowly varies from $-1$ to $1$ as time evolves from $t = -\infty$ to $t = +\infty$, uniformly in $x$, a leading-order heuristic prediction can be obtained using a simple characteristic argument. The uniform growth of $\mu$ causes perturbations of the trivial state to accelerate their growth as $t>0$ increases. Since the growth is slow, one ``freezes coefficients" so that the predicted instantaneous invasion speed at each fixed $t>0$ is given as $s(t) = 2\sqrt{\mu(t)}.$ Hence, given a localized perturbation lying near the origin with support contained in $[-x_0,x_0]$ for some $x_0$, one predicts the front location $x_\mathrm{fr}(t)$ to satisfy the characteristic equation
$$
\frac{dx_\mathrm{fr}}{dt} = s(t),\qquad x_\mathrm{fr}(0) = x_0,
$$
and hence is given as 
\begin{equation}\label{e:xf}
x_\mathrm{fr,pred}(t) = x_0 + \int_0^t 2\sqrt{\mu(\sigma)}d\sigma. 
\end{equation}
For $\mu(\eta) = -\tanh(\epsilon\eta)$, or alternatively for a purely linear ramp $\mu(\eta) = \epsilon \eta,$ it is possible to obtain $x_\mathrm{fr,pred}(t)$ in closed form. See Figure \ref{f:homq} for a comparison of the numerically measured front location $x_\mathrm{fr,num}$ and this prediction. We find, after an initial transient where the front establishes itself, the front location moves slightly faster than the prediction.

A simple heuristic argument supporting this finding goes as follows. The linear spreading in the stationary frame of a perturbation of the trivial state with exponential decay $\sim \re^{\nu x}$ is determined by a quantity known as the \emph{envelope velocity}, defined as $s_\mathrm{env}(\nu) = -\mathrm{Re} \lambda(\nu) / \mathrm{Re}\, \nu $, where $\lambda(\nu)$ is a root of the linear dispersion relation \eqref{e:ldsp} with $c = 0$. For a given $\mu$ and $\nu\in \R$, we find $s_\mathrm{env}(\nu) = -\frac{\nu^2+\mu}{\nu}$.  The linear spreading speed discussed is related to the envelope speed through $s_\mathrm{lin} = \min_{\nu\in \R} s_\mathrm{env}(\nu) = 2\sqrt{\mu}$. In the stationary frame, the $\nu<0$ (corresponding to rightward spreading waves) which minimizes the envelope velocity is given as $\nu = -\sqrt{\mu}$. Now let us return back to the slowly-varying quench $\mu = \mu(t)$.  At a given fixed time $t_1>0$, the above prediction for the invasion speed $s(t_1) = 2\sqrt{\mu(t_1)}$ would have a front with spatial decay $\nu(t_1) = -\sqrt{\mu(t_1)}$. Now for a time $t_2$ just after $t_1$, where $\mu$ has increased further, the envelope speed of this tail $s_\mathrm{env}(\nu(t_1))$ is greater than the predicted instantaneous speed for $\mu(t_2)$.  Hence we expect the front to accelerate faster than predicted in calculation \eqref{e:xf}. We anticipate that one can obtain a more refined prediction, as well as rigorous existence and asymptotics, by explicitly solving the linearized equation $v_t = v_{xx} + \mu(t) v$ to understand spreading asymptotics of exponential tails and then construct fronts using comparison principle methods. %We leave this, as well as the study of the $\alpha\in(0,1/2)$ case, for a future work.

\begin{figure}[h!]
\centering
\includegraphics[trim = 0.05cm 0.05cm 0.05cm 0.05cm,clip,width=0.45\textwidth]{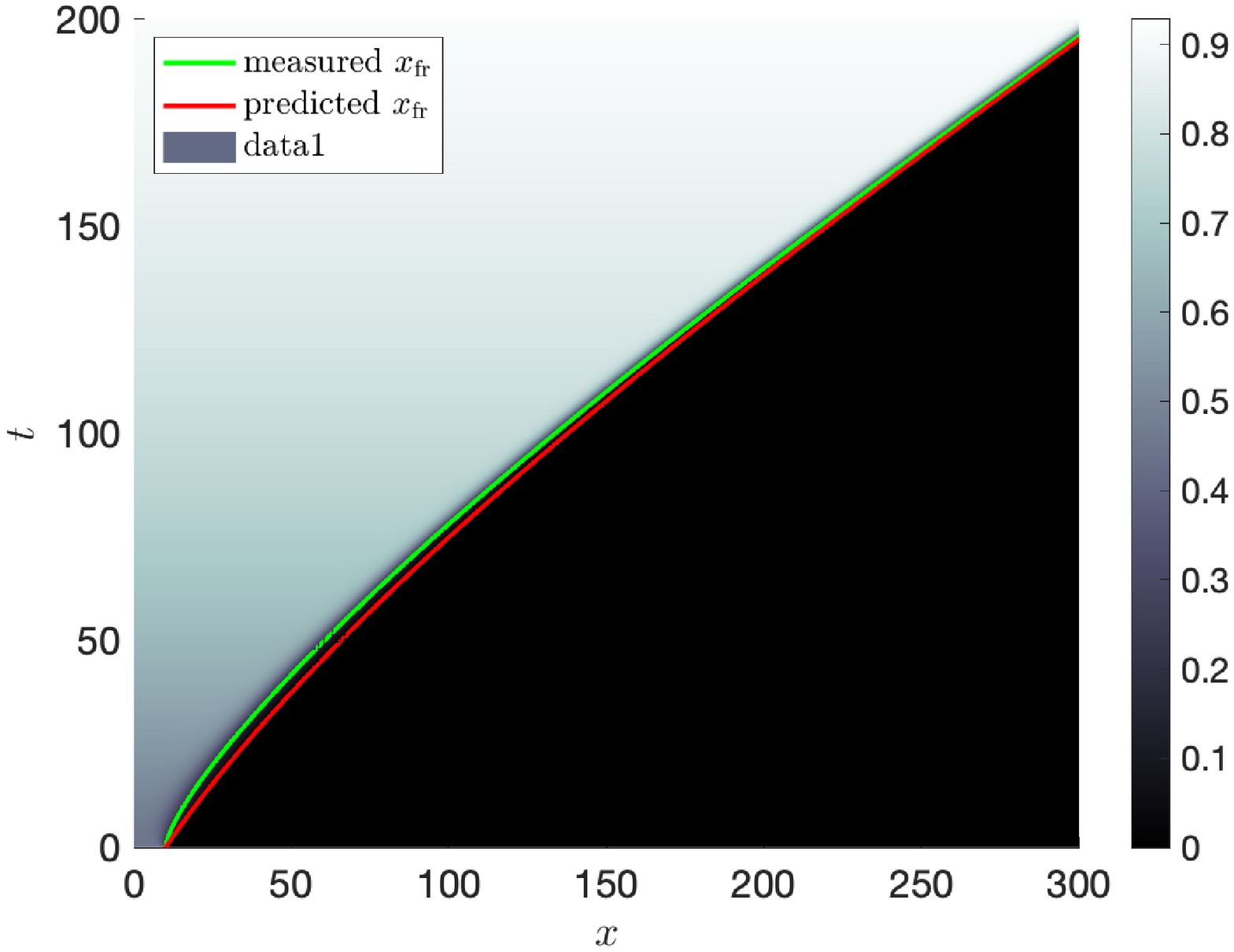}\hspace{-0.2in}
\includegraphics[trim = 0.0cm 0.0cm 0.05cm 0.0cm,clip,width=0.4\textwidth]{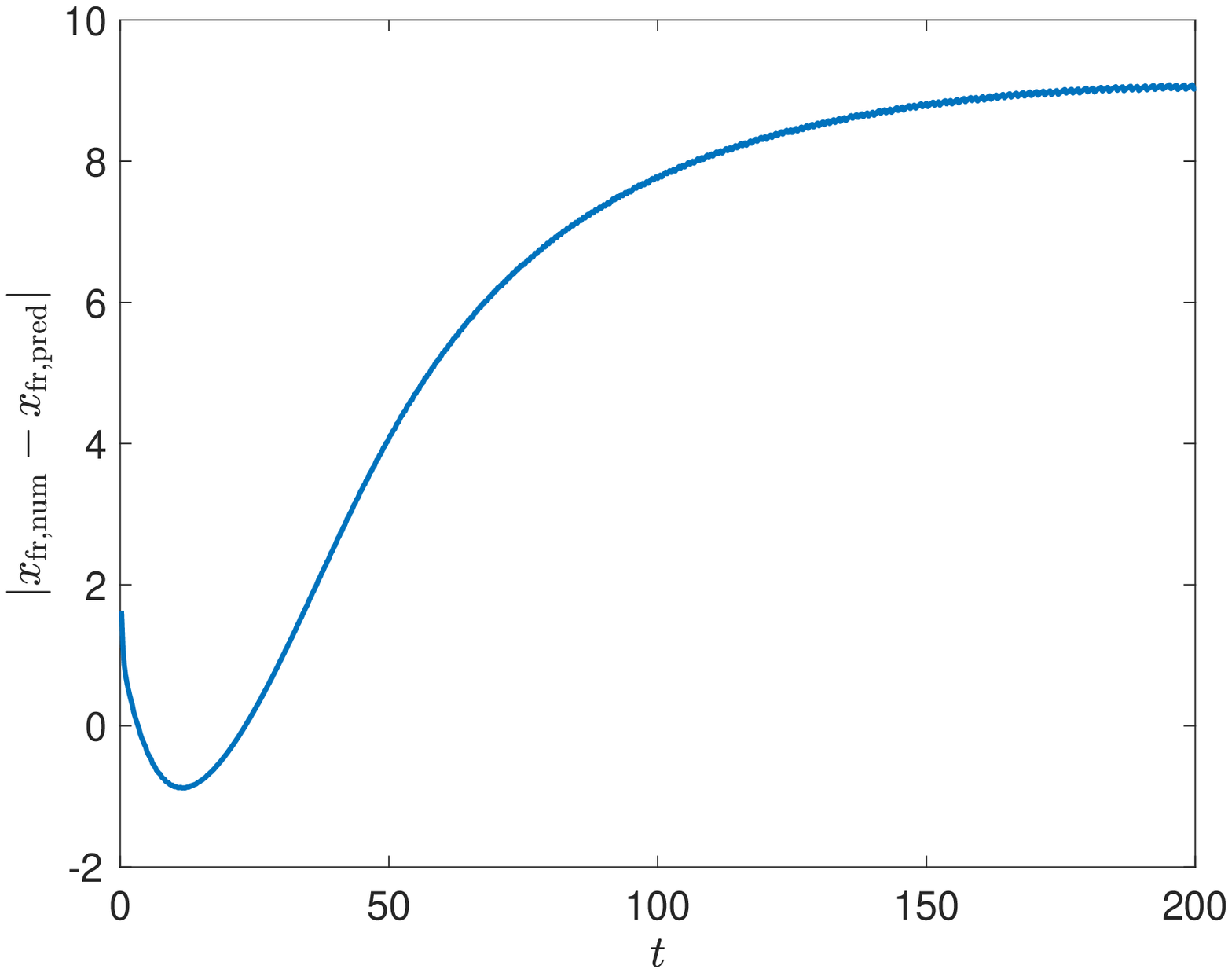}\hspace{-0.2in}\\
\caption{Direct numerical simulation of \eqref{e:alp} with $\epsilon = 0.005, \alpha = 0$; left: Spacetime diagram of the solution with measured front location $x_\mathrm{fr,num}$ (green) where $u(x,t) = 0.2$ and prediction of $x_\mathrm{fr,pred}$ from \eqref{e:xf}; right: depiction of the difference between the measured and predicted front location. }\label{f:homq}
\end{figure}

%\subsection{Other directions??}
%\begin{itemize}
%\item subcritical ramping
%\item Patterned fronts (i.e. asymptotically kink/anti-kink solutions in Allen-Cahn with slow quench? also mention rGL and SH?)
%\item defect mediation?
%\end{itemize}

  \begin{Acknowledgment}
  The authors were partially supported by the National Science Foundation through grants NSF-DMS-2006887 (RG), NSF-DMS-1616064 (TK), and NSF DMS-1907391 and DMS-2205663 (AS). The authors thank the Mathematics Forschungsinstitut Oberwolfach for its hospitality during the workshop “Dynamics of Waves and Patterns” in August 2021. The authors would also like to thank S. Hastings for enlightening email communications on properties of the Hastings-McLeod solution of Painlev\'{e}-II and for the proof of Lemmas \ref{l:a1} and \ref{l:a3}. \end{Acknowledgment}

\appendix
\section{The potential is sign definite.}\label{a:1}

In this appendix, we prove
that the potential obtained from linearizing the
Painlev\'e II equation about the Hastings-McLeod solution
is sign definite. This result (see Lemma \ref{l:a2} below) is not only of use as a direct way to show in Proposition 6.2 that the ground state is sign definite, as remarked above, but it is also of independent interest for the Painlev\'e II equation.
Given the independent interest, we prove the result using the standard form \eqref{e:PII-intro} of the equation.

\begin{Lemma}\label{l:a1} (Hastings \cite{hastings22-privatecommunication})
For the Hastings-McLeod solution,
$w(\eta)$ of 
the Painlev\'e II equation $w'' = \eta w + 2 w^3,$
one has the following lower bound:
$w(\eta=0) 
\ge \mathrm{Ai}(0) 
= \frac{1}{3^{2/3} \Gamma\left( \frac{2}{3} \right)}$.
\end{Lemma}
\begin{proof}
This lemma and its proof are due to
Professor Stuart Hastings \cite{hastings22-privatecommunication}.
By Theorem 2 of \cite{hastings80},
it is known that $\lim_{\eta \to \infty} \frac{w(\eta)}{\rm{Ai}(\eta)} = 1.$
So, suppose that $w(0) < \rm{Ai}(0)$.
Then, there is an $\eta_R > 0$ 
at which $\left( \frac{w}{\rm{Ai}} \right) ' > 0$.
Hence, at $\eta_R$, 
one has $w' \mathrm{Ai} - \mathrm{Ai}' w > 0.$
Next, observe that 
$\left( w' \mathrm{Ai} - \mathrm{Ai}' w \right)' (\eta)= 2 (w(\eta))^3 \rm{Ai}(\eta) > 0$
for all $\eta\ge 0$,
which implies that 
$$
\left( \frac{w}{\mathrm{Ai}} \right)' 
= \frac{ {w}' \mathrm{Ai} - \mathrm{Ai}' w}{\mathrm{Ai}^2} \to \infty, \ \ \ {\rm as} \ \ \eta \to \infty.
$$
This contradicts the asymptotics of $w(\eta)$.
Hence, the supposition that
$w(0) < \mathrm{Ai}(0)$ is incorrect,
and the lemma is proven.
\end{proof}

\bigskip
Lemma \ref{l:a1} is used as a key step in establishing 
the following result about the potential 
$\mathcal{V}(\eta) = \eta + 6(w(\eta))^2$,
obtained by linearizing the right hand side 
of the Painlev\'e II equation 
about the Hastings-McLeod solution.

\begin{Lemma}\label{l:a2} The potential  
$\mathcal{V}(\eta) = \eta + 6 (w(\eta))^2$
evaluated along the Hastings-McLeod solution $w(\eta)$
of the second Painlev\'e equation $w''=\eta w + 2 w^3$
is strictly positive for all $\eta \in \R$.
\end{Lemma}
\begin{proof}
First, for all $\eta \ge 0$,
one sees directly that $\mathcal{V}(\eta)>0$,
since $w(\eta)>0$ for all $\eta$
by Theorem 1 of \cite{hastings80}.
Also, $\mathcal{V}(\eta) > 0$ for $\eta_0 \le \eta < 0$,
where $\eta_0<0$ is the unique point at which 
$w''(\eta)=0$ (recall 
Theorem 1 of \cite{hastings80}), since
$\left(\eta + 2(w(\eta))^2\right) w(\eta)=w''(\eta)>0$
for all $\eta> \eta_0$ 
and $w>0$ for all $\eta$.

The difficult part of the proof
is to show that $\mathcal{V}(\eta)>0$ 
also for all $\eta<\eta_0$.
This may be accomplished as follows.
The potential $\mathcal{V}(\eta)\to \infty$
as $\eta \to -\infty$.
Hence, there is some $\eta_L<0$
sufficiently negative such that
$\mathcal{V}(\eta)>0$ on 
$(-\infty,\eta_L]$.
Now, on the interval
$(\eta_L,\eta_0)$,
we use the coordinate change 
$w(\eta) = \sqrt{-\eta/2} z(\eta)$.
Here, $z(\eta)$ satisfies 
$ \frac{d^2 z}{d\eta^2} + \frac{1}{\eta} \frac{dz}{d\eta} 
= \frac{z}{4\eta^2} + \eta z (1 - z^2),$
which is equation (2.4) 
with $\alpha=0$ in \cite{hastings80}.
It suffices to show that
$\mathcal{V}$,
which is now $\mathcal{V}(\eta)=(-\eta)(3 z^2 - 1)$,
is strictly positive
at any local minimum of $\mathcal{V}$
on $({\eta}_L,{\eta}_0)$.
At a local minimum $\eta_m$ of $\mathcal{V}$,
$\frac{dz}{d\eta} (\eta_m)= \frac{3 (z(\eta_m))^2 - 1}{-6\eta_m z(\eta_m)}.$
Substituting this into the condition that 
$\frac{d^2 \mathcal{V}}{d\eta^2} > 0$ at a local minimum,
one finds that 
$(z(\eta_m))^2-1 > \frac{1}{36(z(\eta_m))^4 (\eta_m)^3}$
at any local minimum of $\mathcal{V}$ on this interval.
Hence, at a local minimum,
the key term in the potential satisfies
$$
3(z(\eta_m))^2 - 1 > 2(z(\eta_m))^2 + \frac{1}{36(z(\eta_m))^4 (\eta_m)^3}.
$$
Now, the term in the right member is strictly positive
as long as 
$w(\eta_m) = \sqrt{-\eta_m/2}\, z(\eta_m) > (576)^{-1/6} = 0.34668\ldots$,
as may be seen by a straightforward calculation.
Moreover, 
$w(\eta_m) > w(0)$,
since $\frac{dw}{d\eta} (\eta)<0$ 
for all $\eta$ by Theorem 1 of \cite{hastings80},
and $w(0) > \mathrm{Ai}(0)= \frac{1}{3^{2/3} \Gamma\left(\frac{2}{3}\right)}
=0.355028\ldots$, by Lemma 1.
Therefore, 
$3(z(\eta_m))^2 - 1 > 0$ 
at any local minimum 
on $(\eta_L,\eta_0)$, 
and hence $\mathcal{V}(\eta)>0$ 
for all $\eta \in 
(\eta_L,\eta_0)$. 
This completes the proof of the lemma.
\end{proof}

The proof of Lemma A.2 involves analysis of local minima of $\mathcal{V}$
and relies on Lemma A.1. An alternative proof of the positivity of the
potential $\cal{V}$ evaluated along the Hastings-McLeod solution $w$ may
be obtained using the method of proof by contradiction, as follows:

\begin{Lemma}\label{l:a3}
(Hastings \cite{hastings22-privatecommunication})
The Hastings-McLeod solution $w(\eta)$
of the second Painlev\'e equation satisfies
$w(\eta) > \sqrt{-\eta/6}$
for $\eta \in (-\infty,0]$.
\end{Lemma}
\begin{proof}
For each $\alpha\ge 6$,
define $f_\alpha (\eta)= \sqrt{-\eta/\alpha}$ on $(-\infty,0]$.
Since $w(\eta) \sim \sqrt{-\eta/2}$ as $\eta\to -\infty$,
there exists an $\eta_L<0$ such that $f_6 (\eta) < w(\eta)$
for $\eta \le \eta_L$.
Moreover, since $f_\alpha (\eta) < f_6 (\eta)$ for $\alpha > 6$
on $(-\infty,0)$,
$f_\alpha (\eta) < w (\eta)$
on $(-\infty,\eta_L]$
for all $\alpha>6$, as well.

Next, since $f_\alpha \to 0$ as $\alpha \to \infty$
uniformly on $[\eta_L,\infty)$,
there is an $\alpha_L >0$ such that
$f_\alpha (\eta)<w(\eta)$ on $(-\infty,0]$
for all $\alpha \ge \alpha_L$.
Hence, if there is a point $\eta$
at which $f_\alpha(\eta)=w(\eta)$
for some $\alpha \ge 6$,
then that point $\eta$ must lie in $[\eta_L,0]$.
Also, if this is true for some $\alpha\ge 6$,
then there must exist a greatest such value,
call it $\alpha^*$.
Moreover, any point of intersection
of $f_{\alpha^*}$ with $w$
must be a point of tangency,
with $f_{\alpha*}(\eta)\le w(\eta)$ on $(-\infty,0]$,
otherwise by continuity $\alpha^*$
would not be the greatest value.

Now, suppose that $\eta^*$ is such a point of tangency
between $f_{\alpha^*}$ and $w$.
At $\eta^*$, one has $f_{\alpha^*}=w>0$,
$f_{\alpha^*}'=w' <0$, and
$f_{\alpha^*}'' \le w''$
(where the sign of $w''$ is unknown).
Also, one has
$w''= \eta^* w + 2 w^3= \eta^* f_{\alpha^*}+2f_{\alpha^*}^3$.
Then, calculating $f_{\alpha^*}''$,
one obtains
$\frac{-1}{4{\alpha^*}^2 f_{\alpha^*}^3}
\le \eta^* f_{\alpha^*} + 2f_{\alpha^*}^3$.
In turn, this implies that
$\frac{-1}{4{\alpha^*}^2}
\le \eta^* f_{\alpha^*}^4 + 2f_{\alpha^*}^6
=\frac{\eta^3}{{\alpha^*}^2}\left( 1 - \frac{2}{\alpha^*}\right)
\le \frac{\eta^3}{{\alpha^*}^2}\left( 1 - \frac{2}{6} \right) <0$.
Hence, ${\eta^*}^3 \ge -\frac{3}{8}$,
and one may bound $\eta^*$ from below as
${\eta^*} \ge -0.73$.
Thus, for any such $\alpha^* \ge 6$,
one finds that
$w(\eta^*)=f_{\alpha^*}(\eta^*)\le 0.349.$
However, this is a contradiction,
since $w(\eta^*) > w(0) \ge \frac{1}{3^{2/3}\Gamma(2/3)} \ge 0.355$,
by Lemma \ref{l:a1}.
Therefore, there cannot be any such $\alpha^* \ge 6$,
and we have $f_6(\eta) < w(\eta)$ for all $\eta \in (-\infty,0]$.
This completes the proof of the lemma.
\end{proof}

\bibliography{ac_slow}
\bibliographystyle{plain}
%\begin{thebibliography}{99}

\end{document}